\documentclass[11pt]{amsart}
\usepackage{macro}

\author{Bharathwaj Palvannan}
\address{Department of Mathematics, David Rittenhouse Laboratory, 209 South 33rd Street University of Pennsylvania, Philadelphia 19104-6395, USA}
\email{pbharath@math.upenn.edu}

\title{On Selmer groups and factoring $p$-adic $L$-functions}
\begin{document}
\maketitle{}
\begin{abstract}
Samit Dasgupta has proved a formula factoring a certain restriction of a 3-variable Rankin-Selberg $p$-adic $L$-function as a product of a 2-variable $p$-adic $L$-function related to the adjoint representation of a Hida family and a Kubota-Leopoldt $p$-adic $L$-function.  We prove a result involving Selmer groups that along with Dasgupta's result is consistent with the main conjectures associated to the 4-dimensional representation (to which the 3-variable $p$-adic $L$-function is associated), the $3$-dimensional  representation (to which the $2$-variable $p$-adic $L$-function is associated) and the $1$-dimensional representation (to which the Kubota-Leopoldt $p$-adic $L$-function is associated). Under certain additional hypotheses, we indicate how one can use work of Urban to deduce main conjectures for the $3$-dimensional representation and the $4$-dimensional representation. One key technical input to our methods is studying the behavior of Selmer groups under specialization.
\end{abstract}
\tableofcontents
\section*{Motivation : Results of Gross and Greenberg}
The purpose of this paper is to prove a result involving Selmer groups, predicted by the Iwasawa main conjectures, corresponding to a factorization formula involving $p$-adic $L$-functions obtained by Dasgupta in \cite{dasgupta2014factorization}. Dasgupta's method of proof is based on an earlier work of Gross \cite{gross1980factorization} in 1980. Gross's work involved factoring a certain restriction of a $2$-variable $p$-adic $L$-function associated to an imaginary quadratic field (constructed by Katz)  into a product of two Kubota-Leopoldt $p$-adic $L$-functions. In 1982, Greenberg \cite{greenberg26iwasawa}  proved the corresponding result on the algebraic side involving classical Iwasawa modules, as predicted by the main conjectures for imaginary quadratic fields and $\Q$. These results provided evidence for the main conjecture for imaginary quadratic fields (before Rubin's proof in \cite{rubin1991main}). Our methods are inspired by this work of Greenberg. We will briefly recall the results of Gross and Greenberg.  \\

From the outset we would like to inform the reader that we will formulate the main conjectures throughout this paper in terms of primitive $p$-adic $L$-functions and \mbox{non-primitive} Selmer groups. As a result, though the earlier formulations of the main conjecture involved classical Iwasawa modules on the algebraic side, we will restate Greenberg's results in terms of non-primitive Selmer groups so that it will be helpful in placing Dasgupta's factorization of primitive $p$-adic $L$-functions and our results involving non-primitive Selmer groups in the context of main conjectures\footnote{Our description in the introduction indicates that one should be able to place the results of Gross and Greenberg along with results of Dasgupta and ours under a general framework. We, however, do not intend to develop such a general framework in this paper, instead leaving the generalities to the interested reader.}. To relate our formulation to the formulation of the main conjecture by Greenberg in \cite{greenberg1994iwasawa} (that involves primitive $p$-adic $L$-functions and primitive Selmer groups) we refer the reader to Section \ref{primitivity-issues}. In that section, we evaluate the difference in the divisors associated to the non-primitive Selmer group and the primitive Selmer group. Let $p \geq 5$ be a prime number. We shall also fix algebraic closures $\overline{\Q}$, $\overline{\Q}_p$ of $\Q$, $\Q_p$ respectively and embeddings $\overline{\Q} \hookrightarrow \overline{\Q}_p$ and $\overline{\Q} \hookrightarrow \mathbb{C}$. Let $\O$ denote the ring of integers in a finite extension of $\Q_p$. We let $\chi_p : \Gal{\overline{\Q}}{\Q} \rightarrow \Z_p^\times$ denote the $p$-adic cyclotomic character. \\

Let $K$ be an imaginary quadratic field where $p$ splits and whose associated quadratic character is given by $\varepsilon : \Gal{\overline{\Q}}{\Q} \rightarrow \{\pm 1\}$. Let $\tilde{K}_\infty$ denote the composite of the $\Z_p$-extensions of $K$. Let $\Q_\infty$ and $K_\infty$ denote the cyclotomic $\Z_p$ extensions of $\Q$ and $K$ respectively. We have the following picture in mind:
\begin{center}
\begin{tikzpicture}[node distance = 0.9cm, auto]
      \node (Q) {$\Q$};
      \node (K) [above of=Q] {$K$};
      \node (Qinfty) [above of=K, left of=Q] {$\Q_\infty$};
      \node (Kinfty) [above of=K, left of = K] {$K_\infty$};
      \node (Ktildeinfty) [right of = K, above of=Kinfty] {$\tilde{K}_\infty$};
       \node (GQ) [ right of = Ktildeinfty, node distance=7cm] {$G_\Q:=\Gal{\overline{\Q}}{\Q}, \quad G_K := \Gal{\overline{\Q}}{K}$};
    \node (tildegamma) [ right of = Kinfty, node distance=6.7cm]       {$\tilde{\Gamma} := \Gal{\tilde{K}_\infty}{K} \cong \Z_p^2$};
     \node (gamma) [ right of = K, node distance=7cm]       {$\Gamma :=\Gal{\Q_\infty}{\Q}  \cong \Gal{K_\infty}{K} \cong \Z_p$};
\node (characters) [right of = Q, node distance = 8cm] {$\tilde{\kappa} : G_K \twoheadrightarrow \tilde{\Gamma} \hookrightarrow \O[[\tilde{\Gamma}]]^\times, \qquad \kappa : G_\Q \twoheadrightarrow \Gamma \hookrightarrow \O[[\Gamma]]^\times$};
            \draw[-] (Q) to node {} (K);
      \draw[-] (Q) to node  {$\Gamma$} (Qinfty);
      \draw[-] (K) to node {}  (Kinfty);
       \draw[-] (Qinfty) to node {} (Kinfty);
        \draw[-] (Kinfty) to node  {} (Ktildeinfty);

        \draw[-] (K) to node [swap] {$\tilde{\Gamma}$}  (Ktildeinfty);

    \end{tikzpicture}
  \end{center}
Let $\psi: G_\Q  \rightarrow \O^\times$ be a finite even continuous character. We shall denote the restriction of $\psi$ to $G_K$ by $\psi_K$. We shall introduce the three Galois representations that occur in the setup of Gross's factorization along with the primitive $p$-adic $L$-functions and the non-primitive Selmer groups associated to them. We shall not make any attempt to precisely define these objects\footnote{The primitive and the non-primitive Selmer groups will be precisely defined in Section \ref{sel-cyc-deformations}.}. Note that the Selmer groups appearing in these main conjectures can be linked to classical Iwasawa modules. See Greenberg's work on $p$-adic Artin $L$-functions \cite{greenberg2014p} for a description of the link. \\

\textbf{(A) The 2-dimensional representation \nopunct} $\Ind_{K}^{\Q}\left(\psi_K \tilde{\kappa}^{-1}\right) : G_\Q \rightarrow \Gl_2\left(\O[[\tilde{\Gamma}]]\right)$.

The main conjecture associated to this two dimensional representation (now known due to work of Rubin \cite{rubin1991main}) predicts the following equality of ideals in $\O[[\tilde{\Gamma}]]$:
 \begin{align}\label{mck}
\tag{$\mathrm{MC-}K$}  (h_{\psi_K} \alpha_{\psi_K \tilde{\kappa}^{-1}}) =  \mathrm{Char}\left( \Sel_{\Ind_{K}^{\Q}\left(\psi_K \tilde{\kappa}^{-1}\right)}(\Q)^\vee\right).
\end{align}
Katz \cite{katz1978p} has constructed a two-variable $p$-adic L-function  $\theta_{\psi_K \tilde{\kappa}^{-1}}$ in $\O_{\mathbb{C}_p}[[\tilde{\Gamma}]]$. Here $\O_{\mathbb{C}_p}$ is the ring of integers in $\mathbb{C}_p$, and $\alpha_{\psi_K \tilde{\kappa}^{-1}}$ is an element in $\O[[\tilde{\Gamma}]]$  that generates the ideal $(\theta_{\psi_K \tilde{\kappa}^{-1}})$ in $\O_{\mathbb{C}_p}[[\tilde{\Gamma}]]$. The element $h_{\psi_K}$ (in the ring $\O[[\tilde{\Gamma}]]$) is an ``error term'' that keeps track of the local Euler factors away from $p$. The characteristic ideal associated to the Pontryagin dual of the non-primitive Selmer group in $\O[[\tilde{\Gamma}]]$ is denoted by $\text{Char}\left( \Sel_{\Ind_{K}^{\Q}\left(\psi_K \tilde{\kappa}^{-1}\right)}(\Q)^\vee\right)$.

\textbf{(B) The even character \nopunct} $\psi \kappa^{-1} : G_\Q \rightarrow \Gl_1(\O[[\Gamma]])$.

\textbf{(C) The odd character \nopunct} $\psi \varepsilon \kappa^{-1} : G_\Q \rightarrow \Gl_1(\O[[\Gamma]])$.

The main conjectures associated to $\psi\kappa^{-1}$ and $\psi \varepsilon \kappa^{-1}$ predict the following equalities of ideals in  $\O[[\Gamma]]$ (now known due to Mazur-Wiles \cite{mazur1984class}):
{\small \begin{align} \label{mcq}
 \tag{$\mathrm{MC-}\Q$} \quad (h_{\psi\kappa^{-1}} \theta_{\psi \kappa^{-1}}) = \text{Char}\left(\Sel_{\psi\kappa^{-1}}(\Q)^\vee\right),  \qquad  (h_{\psi \varepsilon\kappa^{-1}} \theta_{\psi \varepsilon \kappa^{-1}}) = \text{Char}\left(\Sel_{\psi\varepsilon\kappa^{-1}}(\Q)^\vee\right).
\end{align}
}

The $p$-adic $L$-functions $\theta_{\psi \kappa^{-1}}$ and $\theta_{\psi \varepsilon \kappa^{-1}}$, associated to $\psi \kappa^{-1}$ and $\psi \varepsilon \kappa^{-1}$, respectively are in the fraction field of $\O[[\Gamma]]$ and their construction is essentially due to Kubota and Leopoldt \cite{kubota1964p}. Note that one can relate the $p$-adic $L$-function $\theta_{\psi \varepsilon \kappa^{-1}}$ associated to the odd character to another $p$-adic $L$-function associated to the even character $\psi^{-1} \varepsilon^{-1} \kappa \chi_p$. The elements $h_{\psi \kappa^{-1}}$ and $h_{\psi \varepsilon \kappa^{-1}}$ (in the ring $\O[[\Gamma]]$) are ``error terms'' that keep track of the local Euler factors away from $p$ and certain poles of the $p$-adic $L$-functions. The characteristic ideals in $\O[[\Gamma]]$ associated to Pontryagin duals of the non-primitive Selmer groups are denoted by $\text{Char}\left(\Sel_{\psi \kappa^{-1}}(\Q)^\vee\right)$ and  $\text{Char}\left(\Sel_{\psi\varepsilon\kappa^{-1}}(\Q)^\vee\right)$. \\

The surjection $ \tilde{\Gamma} \twoheadrightarrow \Gamma$ of Galois groups gives us ring maps $ \O[[\tilde{\Gamma}]] \rightarrow \O[[\Gamma]]$ and $ \O_{\mathbb{C}_p}[[\tilde{\Gamma}]] \rightarrow \O_{\mathbb{C}_p}[[\Gamma]]$. Abusing notations, we shall denote all of these maps by $\pi_{2,1}$. We have the following decomposition of Galois representations:
\begin{align}\label{quad-decomposition}
\pi_{2,1} \circ \Ind_{K}^{\Q}\left(\psi_K \tilde{\kappa}^{-1}\right) \cong \psi \kappa^{-1} \oplus \psi \varepsilon \kappa^{-1}.
\end{align}
Informally, one can think of the map $\pi_{2,1}$ as setting the ``anti-cyclotomic'' variable to equal zero. As an interesting manifestation of the decomposition of Galois representations in (\ref{quad-decomposition}), we have the following theorem due to Gross and unpublished work of Greenberg-Lundell-Zhang (Gross only considers the case when the conductor of $\psi$ is a power of $p$):
\begin{Oldtheorem}[Gross \cite{gross1980factorization}, Greenberg-Lundell-Zhang \cite{greenberglundell}]\label{gross-factorization}
$\pi_{2,1}\left(\theta_{\psi_K \tilde{\kappa}^{-1}}\right) = \theta_{\psi \kappa^{-1}} \theta_{\psi \varepsilon \kappa^{-1}}$.
\end{Oldtheorem}

Implicit in \cite{greenberg26iwasawa} is the following theorem on the algebraic side:

\begin{Oldtheorem}[Greenberg \cite{greenberg26iwasawa}] \label{greenberg-factorization}
We have the following equality of  ideals in $\O[[\Gamma]]$:
\begin{align*}
\mathrm{Char}\left( \Sel_{\pi_{2,1} \circ \Ind_{K}^{\Q}\left(\psi_K \tilde{\kappa}^{-1}\right)}(\Q)^\vee\right) = \mathrm{Char}\left(\Sel_{\psi \kappa^{-1}}(\Q)^\vee\right) \cdot \mathrm{Char}\left(\Sel_{\psi\varepsilon\kappa^{-1}}(\Q)^\vee\right).
\end{align*}
\end{Oldtheorem}

There is no main conjecture associated to the Galois representation $\pi_{2,1} \circ \Ind_{K}^{\Q}\left(\psi_K \tilde{\kappa}^{-1}\right)$ as it does not satisfy the ``Panchishkin condition''\footnote{The Panchishkin condition is a kind of ``ordinariness'' assumption, introduced by Greenberg, while formulating the Iwasawa main conjecture for Galois deformations. See Section 4 in \cite{greenberg1994iwasawa} for the precise definition.}. So, one can ask the following question: How are Theorem \ref{gross-factorization} and Theorem \ref{greenberg-factorization} related to the main conjectures \ref{mck} and \ref{mcq}? This is answered by the following result of Greenberg, described in Pages 283 and 284 of \cite{greenberg26iwasawa}:

\begin{Oldtheorem}[Greenberg \cite{greenberg26iwasawa}] \label{greenberg-specialization}
Suppose \ref{mck} holds. We have the following equality of ideals in $\O[[\Gamma]]$:
\begin{align}\label{gross-greenberg-specialization}
\mathrm{Char}\left( \Sel_{\pi_{2,1} \circ \Ind_{K}^{\Q}\left(\psi_K \tilde{\kappa}^{-1}\right)}(\Q)^\vee\right) = \left(\pi_{2,1}\left(\alpha_{\psi_K \tilde{\kappa}^{-1}}\right) h_{\psi\kappa^{-1}} h_{\psi \varepsilon \kappa^{-1}}\right).
\end{align}
\end{Oldtheorem}
Greenberg observed that if \ref{mcq} held, we would have had the following equality of~ideals~in~$\O[[\Gamma]]$:
\begin{align} \label{multiply-mcq}
\mathrm{Char}\left(\Sel_{\psi\kappa^{-1}}(\Q)^\vee\right) \mathrm{Char}\left(\Sel_{\psi\varepsilon\kappa^{-1}}(\Q)^\vee\right) = (\theta_{\psi \kappa^{-1}} \theta_{\psi \varepsilon \kappa^{-1}} h_{\psi\kappa^{-1}} h_{\psi \varepsilon\kappa^{-1}}).
\end{align}
Having assumed the validity of the main conjecture \ref{mck}, equation (\ref{gross-greenberg-specialization}) along with Theorem \ref{gross-factorization} and Theorem \ref{greenberg-factorization} showed that equation (\ref{multiply-mcq}) also held true. It is in this manner that Theorem \ref{greenberg-specialization} ascertained that Theorem \ref{gross-factorization} and Theorem \ref{greenberg-factorization} were completely consistent with the various main conjectures. In our setup, Dasgupta's factorization (Theorem \ref{dasgupta-factorization}) is an analog of Theorem \ref{gross-factorization}, while Theorem \ref{selmer-factorization} is an analog of Theorem \ref{greenberg-factorization}.  Just as Theorem \ref{greenberg-specialization} was used to relate Theorem \ref{gross-factorization} and Theorem \ref{greenberg-factorization} to several main conjectures, we will use Theorem \ref{specialization-result} to relate Theorem \ref{dasgupta-factorization} and Theorem \ref{selmer-factorization} to several main conjectures.

\section*{Main results related to Dasgupta's factorization}

Let $F = \sum_{n=1}^{\infty}a_n(F) q^n \in \R[[q]]$ be a Hida family. The ring $\R$ is an integrally closed local domain and a finite integral extension of $\Z_p[[x]]$, where $x$ denotes the ``weight variable'' for $F$. The ring $R$ is the normalization of an irreducible component of Hida's (ordinary, primitive) Hecke algebra. The element $a_p(F)$ is a unit in the local ring $R$. Let $\Sigma$ be a finite set of primes in $\Q$ containing $p$, $\infty$, all the primes dividing the level of $F$ and a non-archimedean prime $l \neq p$. Let $\Sigma_0 =\Sigma \setminus \{p\}$. Let $G_\Sigma$ equal $\Gal{\Q_\Sigma}{\Q}$, where $\Q_\Sigma$ is the maximal extension of $\Q$ unramified outside $\Sigma$. Suppose $F$ satisfies the following hypotheses:
\begin{enumerate}[leftmargin=2cm, style=sameline, align=left, label=\textsc{IRR}, ref=\textsc{IRR}]
\item\label{IRR} The residual representation associated to $F$ is absolutely irreducible.
\end{enumerate}
\begin{enumerate}[leftmargin=2cm, style=sameline, align=left, label=\textsc{$p$-Dis}, ref=\textsc{$p$-Dis}]
\item\label{p-Dis} The restriction of the residual representation to the decomposition subgroup at $p$ (which is reducible) has non-scalar semi-simplification.
\end{enumerate}
Let $\rho_F : G_\Sigma \rightarrow \Gl_2(\R)$ be the Galois representation associated to $F$.  Let $L_F$ be the free $\R$-module of rank $2$ on which $G_\Sigma$ acts to let us obtain $\rho_F$. Without loss of generality, we shall suppose that the ring $\O$ equals the integral closure of $\Z_p$ in $\R$. We let $\T$ equal the completed tensor product $\R \hotimes \R$. The completed tensor product is the co-product in the category of complete semi-local Noetherian $\O$-algebras (where the morphisms are continuous). The ring $\T$ is a complete integrally closed local domain and a finite integral extension of $\Z_p[[x_1 ,x_2]]$, where  $x_1$ and $x_2$ are identified with the ``weight variables''. The completed tensor product $\T$ comes equipped with two natural maps $i_1 :\R \rightarrow \T$
and $i_2 :  \R \rightarrow \T$. We have a $4$-dimensional Galois representation $$\rho_{F,F}: G_\Sigma \rightarrow \Gl_4(\T).$$
given by the action of $G_\Sigma$ on $\Hom_{\T}\left( L_F \otimes_{i_1}\T, \  L_F \otimes_{i_2}\T \right)$, which we denote by $L_{F,F}$, and which is a free $\T$-module of rank $4$. We have a natural map $\pi_{F,F} : \T \rightarrow \R$ obtained by sending an elementary tensor $a\otimes b$ to $ab$. The map $\pi_{F,F}$ is a surjective $\O$-algebra homomorphism. Under this map, we have $\pi_{F,F}(x_1)=x$ and $\pi_{F,F}(x_2)=x$. Informally, we can think of this map as setting the two ``weight variables'' to equal each other.  Composing $\rho_{F,F}$ with $\pi_{F,F}$ gives us the following  Galois representation: $$\pi_{F,F} \circ \rho_{F,F} : G_\Sigma \xrightarrow {\rho_{F,F}} \Gl_4(\T) \xrightarrow {\pi} \Gl_4(\R).$$
We have the following decomposition of Galois representations:
{\small \begin{align}\label{3-variable-decomp}
\underbrace{\pi_{F,F} \circ \rho_{F,F}}_{\substack{\text{Action of $G_\Sigma$}\\ \text{on $M_2(R)$} \\ \text{by conjugation via $\rho_F$}}} \cong \underbrace{\Ad^0(\rho_F)}_{\substack{\text{Action of } G_\Sigma \\ \text{  on the trace-zero matrices} \\ \text{ in $M_2(R)$} \\ \text{by conjugation via $\rho_F$}}} \oplus \underbrace{\pmb{1}}_{\substack{\text{Trivial action of } G_\Sigma \\ \text{ \newline on the scalar matrices} \\ \text{ in $M_2(R)$}}}.
\end{align}}

Consider a (finite order) Dirichlet character $\chi : G_\Sigma \rightarrow \O^\times$.  To a Galois representation $\varrho : G_\Sigma \rightarrow \Gl_d(\RRR)$, we shall associate a $d$-dimensional Galois representation $\varrho \otimes \kappa^{-1} : G_\Sigma \rightarrow \Gl_d(\RRR[[\Gamma]])$ (which is related to the cyclotomic deformation of $\varrho$). The Galois representation $\varrho \otimes \kappa^{-1}$ will be defined in Section \ref{sel-cyc-deformations}.  We shall  introduce the Galois representations that appear in Dasgupta's factorization. Note that there are two numbers in the subscripts appearing in the labels for the various Galois representations below. The first number (in boldface) represents the dimension of the Galois representation while the second number (not in boldface) is a number one less than the Krull dimension of the underlying ring over which the corresponding Galois representation is defined. This second number is also often referred to as the number of variables in the corresponding $p$-adic $L$-function. We would like to explicitly state that all the $p$-adic $L$-functions mentioned here are primitive. We will formulate the main conjectures relating primitive $p$-adic $L$-functions and non-primitive Selmer groups. In this paper, the primitive Selmer groups are mentioned only in Section \ref{sel-cyc-deformations} (where they are defined along with the non-primitive Selmer groups) and Section \ref{primitivity-issues} (where the differences in the divisors associated to the primitive and non-primitive Selmer groups are calculated). The non-primitive Selmer groups are defined, just as the primitive Selmer groups, as the kernel of a ``global-to-local'' map except that we omit the local conditions at primes $\nu \in \Sigma_0$. We prefer working with non-primitive Selmer groups since it is easier to establish that they satisfy better algebraic properties. \\

\textbf{The 4-dimensional representation}
Let $\rho_{\pmb{4},3} : G_\Sigma \rightarrow \Gl_4(\T[[\Gamma]])$ be the 4-dimensional Galois representation given by $\rho_{F,F} (\chi) \otimes{\kappa}^{-1}$. See \cite{hida1988p} and \cite{dasgupta2014factorization} for the properties that the primitive 3-variable $p$-adic $L$-function $\theta_{\pmb{4},3}$, associated to $\rho_{\pmb{4},3}$, satisfies. The $p$-adic $L$-function $\theta_{\pmb{4},3}$ is an element of the fraction field of $T[[\Gamma]]$. For $\theta_{\pmb{4},3}$, we can vary two weight variables and one cyclotomic variable.  We can also associate a  non-primitive Selmer group $\Sel_{\rho_{\pmb{4},3}}(\Q)$ to $\rho_{\pmb{4},3}$.

\textbf{The 3-dimensional representation} We have the 3-dimensional trace-zero adjoint representation $\Ad^0(\rho_F) : G_\Sigma \rightarrow \Gl_3(\R)$. We let $\rho_{\pmb{3},2}=\Ad^0(\rho_F) (\chi) \otimes{\kappa}^{-1}$. See \cite{hida1990p} and \cite{dasgupta2014factorization} for the properties that the 2-variable primitive $p$-adic $L$-function $\theta_{\pmb{3},2}$, associated to $\rho_{\pmb{3},2}$, satisfies. The $p$-adic $L$-function $\theta_{\pmb{3},2}$ is an element of the fraction field of $R[[\Gamma]]$. For $\theta_{\pmb{3},2}$, we can vary one weight variable and one cyclotomic variable. We can also associate a non-primtive Selmer group $\Sel_{\rho_{\pmb{3},2}}(\Q)$ to $\rho_{\pmb{3},2}$.

\textbf{The 1-dimensional representation}
We let the Galois representation $\rho_{\pmb{1},2}: G_\Sigma \rightarrow \Gl_1(R[[\Gamma]])$ equal $ \chi \otimes \kappa^{-1}$.  We have a one variable $p$-adic $L$-function $\theta_{\pmb{1},1}$ (due to \cite{kubota1964p}) in the fraction field of $\O[[\Gamma]]$. We let  $\theta_{\pmb{1},2}$ denote the image of $\theta_{\pmb{1},1}$ under the natural inclusion $\O[[\Gamma]] \hookrightarrow R[[\Gamma]]$. For $\theta_{\pmb{1},2}$, we can vary the cyclotomic variable while it is constant in the weight variable. We can also associate a non-primitive Selmer group $\Sel_{\rho_{\pmb{1},2}}(\Q)$ to $\rho_{\pmb{1},2}$.  \\

The map $\pi_{F,F}$ induces a surjective $\O[[\Gamma]]$-algebra homomorphism $\pi:T[[\Gamma]] \rightarrow R[[\Gamma]]$. Equation (\ref{3-variable-decomp}) gives us the following isomorphism:
\begin{align}\label{decomposition-twist}
\pi \circ \rho_{\pmb{4},3} \cong \rho_{\pmb{3},2} \oplus \rho_{\pmb{1},2}.
\end{align}
We will also associate a non-primitive Selmer group $\Sel_{\pi \circ \rho_{\pmb{4},3}}(\Q)$ to the Galois representation $\pi \circ \rho_{\pmb{4},3}$. The decomposition in (\ref{decomposition-twist}) exhibits an interesting phenomenon involving $p$-adic $L$-functions and Selmer groups. We have the following theorem due to Dasgupta (which was originally conjectured by Citro~\cite{citro2008invariants}):

\begin{Theorem} [Dasgupta \cite{dasgupta2014factorization}]\label{dasgupta-factorization} $\pi(\theta_{\pmb{4},3}) = \theta_{\pmb{3},2} \cdot \theta_{\pmb{1},2}$.
\end{Theorem}

The differences between non-primitive Selmer groups and primitive Selmer groups have been studied systematically in Section 3 of \cite{greenberg2010surjectivity}. For the cases we are interested in, these differences can be evaluated explicitly (see Proposition \ref{primitive-non-primitive-difference}) in terms of certain local factors at primes $\nu \in \Sigma_0$ (which are given below).
\begin{align*}
\Loc(\nu,\rho_{\pmb{4},3}) := H^1(I_\nu,D_{ \rho_{\pmb{4},3}})^{\Gamma_\nu}, & \qquad  \Loc(\nu,\pi \circ \rho_{\pmb{4},3}) := H^1(I_\nu,D_{\pi \circ \rho_{\pmb{4},3}})^{\Gamma_\nu}, \\
\Loc(\nu,\rho_{\pmb{3},2}) := H^1(I_\nu,D_{ \rho_{\pmb{3},2}})^{\Gamma_\nu}, & \qquad
\Loc(\nu,\rho_{\pmb{1},2}) := H^1(I_\nu,D_{ \rho_{\pmb{1},2}})^{\Gamma_\nu}.
\end{align*}

Here, $I_\nu$ is the inertia subgroup inside $\Gal{\overline{\Q}_\nu}{\Q_{\nu}}$ and $\Gamma_\nu$ is defined to be the quotient $\Gal{\overline{\Q}_\nu}{\Q_{\nu}}/I_\nu$. The discrete modules  $D_{\rho_{\pmb{4},3}}$, $D_{\rho_{\pmb{3},2}}$, $D_{\rho_{\pmb{1},2}}$ and $D_{\pi \circ \rho_{\pmb{4},3}}$ associated to $\rho_{\pmb{4},3}$, $\rho_{\pmb{3},2}$, $\rho_{\pmb{1},2}$ and $\pi \circ \rho_{\pmb{4},3}$ will be defined in Section \ref{sel-cyc-deformations}. \\

Note that $M^\vee$ will denote the Pontryagin dual of a module $M$ over a profinite ring. Let us label a hypothesis that we shall invoke frequently.

\begin{enumerate}[leftmargin=3cm, style=sameline, align=left, label=\textsc{$\mathrm{AD-TOR}$}, ref=\textsc{$\mathrm{AD-TOR}$}]
\item\label{ad-tor} $\Sel_{\rho_{\pmb{3},2}}(\Q)^\vee$ is a torsion $R[[\Gamma]]$-module.
\end{enumerate}

\begin{Theorem} \label{selmer-factorization}
The hypothesis \ref{ad-tor} holds if and only if $\Sel_{\pi \circ \rho_{\pmb{4},3}}(\Q)^\vee$ is a torsion $R[[\Gamma]]$-module. If \ref{ad-tor} holds, we have the following equality in the divisor group\footnote{We refer the reader to the end of the introduction where  various terminologies used in the paper  are explained, including the Pontryagin dual of a module, the divisor group and the divisor class group of an integrally closed domain $\RRR$. If $\RRR$ is an integrally closed domain, we also associate to a finitely generated torsion $\RRR$-module and to a non-zero element of $\RRR$, an element of its divisor~group.} of $R[[\Gamma]]$ relating the non-primitive Selmer groups:
\begin{align*}
\Div \left(\Sel_{\pi \circ \rho_{\pmb{4},3}}(\Q)^\vee \right) = \Div\left( \Sel_{\rho_{\pmb{3},2}}(\Q)^\vee\right) + \Div\left(\Sel_{\rho_{\pmb{1},2}}(\Q)^\vee \right) .
\end{align*}
We also have the following decomposition of local factors away from $p$:
\begin{align*}
\Loc(\nu,\pi \circ \rho_{\pmb{4},3}) & \cong \Loc(\nu,\rho_{\pmb{3},2}) \oplus \Loc(\nu,\rho_{\pmb{1},2}), \qquad \text{for all $\nu \in \Sigma_0$}.
\end{align*}
\end{Theorem}

The difficulty in establishing Theorem \ref{dasgupta-factorization} arises due to the fact that the ring homomorphism $\pi$ lies outside the critical range for $\rho_{\pmb{4},3}$. That is, the $p$-adic $L$-function $\theta_{\pmb{4},3}$ satisfies an ``interpolation property'' at various critical specializations $\varphi \in \Hom_{\cont}(T[[\Gamma]], \overline{\Q}_p)$.
And for every such critical specialization, we have $\ker(\pi) \not \subset \ker(\varphi)$. As a result, there is no main conjecture associated to the Galois representation $\pi \circ \rho_{\pmb{4},3}$ as it does not satisfy the ``Panchishkin condition''. \\

This leads us to the following questions: How are Theorem \ref{dasgupta-factorization} and Theorem \ref{selmer-factorization} related to the main conjectures? And, what is the relationship between $\Sel_{\pi \circ \rho_{\pmb{4},3}}(\Q)$ and $\pi(\theta_{\pmb{4},3})$? The purpose of proving Theorem \ref{specialization-result} is to answer these questions and hence to ascertain that Theorem \ref{dasgupta-factorization} and Theorem \ref{selmer-factorization} are completely consistent with the main conjectures for $\rho_{\pmb{4},3}$, $\rho_{\pmb{3},2}$ and $\rho_{\pmb{1},2}$.

\begin{Theorem} \label{specialization-result}
$\Sel_{\pi \circ \rho_{\pmb{4},3}}(\Q)^\vee$ is a torsion $R[[\Gamma]]$-module if and only if the height one prime ideal $\ker(\pi)$ in $T[[\Gamma]]$ does not belong to the support of $\Sel_{\rho_{\pmb{4},3}}(\Q)^\vee$. \\

Suppose $\Sel_{\pi \circ \rho_{\pmb{4},3}}(\Q)^\vee$ is a torsion $R[[\Gamma]]$-module. Also, suppose we have the following inequality in the divisor group of $\T[[\Gamma]]$:
\begin{align} \label{euler-inequality}
 \Div \left( \theta_{\pmb{4},3}\right)  +  \sum \limits_{\nu \in \Sigma_0}  \Div \left(  \Loc(\nu,\rho_{\pmb{4},3})^\vee \right) \geq \Div \left( \Sel_{\rho_{\pmb{4},3}}(\Q)^\vee \right) -  \Div \left(H^0(G_\Sigma, D_{\rho_{\pmb{4},3}})^\vee\right)  . \tag{$\mathrm{ES}$}
\end{align}
Then, we have the following inequality in the divisor group of $R[[\Gamma]]$:
{\small \begin{align} \label{first-main-inequality}
\Div \left( \pi(\theta_{\pmb{4},3})\right) + \sum \limits_{\nu \in \Sigma_0}  \Div \left(  \Loc(\nu, \pi \circ \rho_{\pmb{4},3})^\vee \right)  \geq  \Div \left( \Sel_{\pi \circ \rho_{\pmb{4},3}}(\Q)^\vee \right) - \Div \left(H^0(G_\Sigma, D_{\pi \circ \rho_{\pmb{4},3}})^\vee\right) .
\end{align}}
Furthermore, if the divisor $\Div \left( \Sel_{\rho_{\pmb{4},3}}(\Q)^\vee \right) - \Div \left(H^0(G_\Sigma, D_{ \rho_{\pmb{4},3}})^\vee\right)$ generates a torsion element in the divisor class group of $T[[\Gamma]]$, then equality holds in (\ref{euler-inequality}) if and only if equality holds in (\ref{first-main-inequality}).
\end{Theorem}

By Proposition \ref{divisor-local-factors} and Proposition \ref{torsion-local-factors}, for each prime $\nu \in \Sigma_0$, the divisor $\Div \left( \Loc(\nu, \pi \circ \rho_{\pmb{4},3})^\vee \right)$ in $T[[\Gamma]]$ is principal. As a result, note that if equality holds in \ref{euler-inequality}, then Theorem \ref{specialization-result} lets us deduce that equality holds in (\ref{first-main-inequality}) as well. \\

The results of Theorem \ref{greenberg-specialization} and Theorem \ref{specialization-result} fall under the topic of ``specializations'' of Selmer groups. Though Theorem \ref{specialization-result} is an analog of Theorem \ref{greenberg-specialization}, it is more difficult to prove Theorem \ref{specialization-result} since we allow the ring $R$ to be fairly general. See Example \ref{first-non-regular-example} and Example \ref{second-non-regular-example} which illustrate some of the difficulties that one encounters in the general case. The main difficulty that one encounters, is that, the kernel of the specialization map $\pi$ is no longer known to be a principal ideal. As a result, even though one is interested in a result (such as the one in equation (\ref{first-main-inequality})) involving the height one prime ideals of the ring $R[[\Gamma]]$, one is forced to consider the localizations of the ring $T[[\Gamma]]$ at height two prime ideals containing $\ker(\pi)$ (note that $\ker(\pi)$ itself is a height one prime ideal in the ring $T[[\Gamma]]$). The main novelty of our work, described in Section \ref{specialization-section}, is to address this question on (height one) specializations of Selmer groups from a general perspective. When the ring $R$ is a regular local ring, Theorem \ref{specialization-result} is easier to prove. Another interesting point to note is that, if we had results asserting the vanishing of the $\mu$-invariant for the Selmer group associated to the cyclotomic deformation of twists of the adjoint representation of a cusp form, Theorem \ref{specialization-result} would be easier to prove. We, however, do not place these restrictions.

\section*{Relation to the main conjectures}

We will recall the main conjectures stated in \cite{greenberg1994iwasawa} but formulate it in terms of non-primitive Selmer groups, the local factors at primes $\nu \in \Sigma_0$ and the primitive $p$-adic $L$-function. Note that the main conjecture is known for $\rho_{\pmb{1},2}$ due to work of \cite{mazur1984class}. We have the following equality in the divisor group of $R[[\Gamma]]$:
{\small \begin{align} \label{mainconj-1}
 \Div \left( \theta_{\pmb{1},2}\right)   + \sum \limits_{\nu \in \Sigma_0}  \Div \left(  \Loc(\nu,\rho_{\pmb{1},2})^\vee \right) &= \Div \left( \Sel_{\rho_{\pmb{1},2}}(\Q)^\vee \right) - \Div \left(H^0(G_\Sigma, D_{\rho_{\pmb{1},2}})^\vee\right).\tag{\text{MC-$\rho_{\pmb{1},2}$}}
\end{align}}

First assume that \ref{ad-tor} holds.  Proposition \ref{primitive-non-primitive-difference} will allow us to also write the main conjectures for $\rho_{\pmb{4},3}$ and $\rho_{\pmb{3},2}$ in terms of non-primitive Selmer groups. The main conjecture for $\rho_{\pmb{4},3}$ predicts the following equality in the divisor group of $T[[\Gamma]]$:
{\small \begin{align}\label{mainconj-3}
 \Div \left( \theta_{\pmb{4},3}\right)   + \sum \limits_{\nu \in \Sigma_0}  \Div \left(  \Loc(\nu,\rho_{\pmb{4},3})^\vee \right) &\stackrel{?}{=} \Div \left( \Sel_{\rho_{\pmb{4},3}}(\Q)^\vee \right) -  \Div \left(H^0(G_\Sigma, D_{\rho_{\pmb{4},3}})^\vee\right) . \tag{\textsc{MC-$\rho_{\pmb{4},3}$}}
 \end{align} }
The main conjecture for $\rho_{\pmb{3},2}$ predicts the following equality in the divisor group of $R[[\Gamma]]$:
{\small \begin{align} \label{mainconj-2}
 \Div \left( \theta_{\pmb{3},2}\right) + \sum \limits_{\nu \in \Sigma_0}  \Div \left(  \Loc(\nu,\rho_{\pmb{3},2})^\vee \right)  &\stackrel{?}{=} \Div \left( \Sel_{\rho_{\pmb{3},2}}(\Q)^\vee \right)- \Div \left(H^0(G_\Sigma, D_{\rho_{\pmb{3},2}})^\vee\right). \tag{\textsc{MC-$\rho_{\pmb{3},2}$}} \end{align}
}
We expect the inequality in \ref{euler-inequality} to follow from the recent work  of Lei-Loeffler-Zerbes \cite{lei2012euler} on Euler systems. Suppose \ref{euler-inequality} does hold. Combining Theorem \ref{dasgupta-factorization}, Theorem \ref{selmer-factorization}, Theorem \ref{specialization-result} and \ref{mainconj-1}, we obtain the following inequality in the divisor group of $R[[\Gamma]]$:
\begin{align*}
 \Div \left( \theta_{\pmb{3},2}\right)  +  \sum \limits_{\nu \in \Sigma_0}  \Div \left(  \Loc(\nu,\rho_{\pmb{3},2})^\vee \right)  \geq  \Div \left( \Sel_{\rho_{\pmb{3},2}}(\Q)^\vee \right) - \Div \left(H^0(G_\Sigma, D_{\rho_{\pmb{3},2}})^\vee\right).
\end{align*}

Under certain additional hypotheses mentioned in a work of Urban  (see Theorem 3.7 in Urban's work involving the Eisenstein-Klingen ideal \cite{urban1998selmer}), we have the following inequality in the divisor group~of~$R[[\Gamma]]$:
\begin{align} \label{urban-inequality}
 \Div \left( \theta_{\pmb{3},2}\right) +  \sum \limits_{\nu \in \Sigma_0}  \Div \left(  \Loc(\nu,\rho_{\pmb{3},2})^\vee \right)  \leq \Div \left( \Sel_{\rho_{\pmb{3},2}}(\Q)^\vee \right) -  \Div \left(H^0(G_\Sigma, D_{\rho_{\pmb{3},2}})^\vee\right). \tag{$\mathrm{UR}$}
\end{align}

Let us grant ourselves the validity of \ref{urban-inequality} too. Assuming \ref{euler-inequality} and \ref{urban-inequality} hold, we obtain the main conjecture \ref{mainconj-2}. Further, if the divisor $\Div\left(\Sel_{\rho_{\pmb{4},3}}(\Q)^\vee\right)~-~\Div\left(H^0(G_\Sigma,D_{\rho_{\pmb{4},3}})^\vee\right)$ generates a torsion element in the divisor class group of $T[[\Gamma]]$, we obtain the main conjecture \ref{mainconj-3} too. We now state this as a theorem.

\begin{Theorem}
Suppose \ref{ad-tor}, \ref{euler-inequality} and \ref{urban-inequality} hold. Then, the main conjecture \ref{mainconj-2} holds. In addition, if the divisor $\Div\left(\Sel_{\rho_{\pmb{4},3}}(\Q)^\vee\right)~-~\Div\left(H^0(G_\Sigma,D_{\rho_{\pmb{4},3}})^\vee\right)$ generates a torsion element in the divisor class group of $T[[\Gamma]]$, we obtain the main conjecture \ref{mainconj-3} too.
\end{Theorem}

Now assume \ref{ad-tor} does not hold, i.e. $\Sel_{\rho_{\pmb{3},2}}(\Q)^\vee$  has positive $R[[\Gamma]]$-rank. In this case, the main conjecture for $\rho_{\pmb{3},2}$ predicts that $\theta_{\pmb{3},2}=0$.  The $p$-adic $L$-function $\theta_{\pmb{1},2}$, constructed by Kubota-Leopoldt, is not equal to zero.  Due to a result of Iwasawa \cite{MR0349627} (see also Proposition 2.1 in \cite{greenberg2014p}), $\Sel_{\rho_{\pmb{1},2}}(\Q)^\vee$ is $R[[\Gamma]]$-torsion. These observations along with Theorem \ref{dasgupta-factorization}, Theorem \ref{selmer-factorization} and Theorem \ref{specialization-result} are consistent with the main conjecture for $\rho_{\pmb{4},3}$ in the following way. On the analytic side,  we have $\pi(\theta_{\pmb{4},3})=0$. On the algebraic side, one can assert that the height one prime ideal $\ker(\pi)$ belongs to the support of $ \Sel_{\rho_{\pmb{4},3}}(\Q)^\vee$ as a $T[[\Gamma]]$-module.

\begin{Remark}
We would like to make a remark about the hypothesis \ref{ad-tor}. Let $\chi$ be an even character for this remark. In this case, there are results of Hida establishing \ref{ad-tor}. See Corollary 3.90 in Hida's book \cite{hida2006hilbert} for the hypotheses under which \ref{ad-tor} is known to hold. Consider the odd character $\omega \chi^{-1}$. We briefly indicate how one can use the results of Hida to establish \ref{ad-tor} for the odd character $\omega \chi^{-1}$ too. Let us assume that \ref{ad-tor} holds for the even character $\chi$. One can use a \textit{control theorem} to conclude that  $\Sel_{\Ad^0(\rho_f)(\chi) \otimes \kappa^{-1}}(\Q)^\vee$  is $\Z_p[[\Gamma]]$-torsion for at least one classical specialization $f$ of $F$ with weight $k \geq 2$ (and in fact for all but possibly finitely many classical specializations). One can then use a result of Greenberg (Theorem 2 in \cite{greenberg1989iwasawa}) and a control theorem to conclude that $\Sel_{\Ad^0(\rho_F)(\omega \chi^{-1})\otimes\kappa^{-1}}(\Q)^\vee$  is $\R[[\Gamma]]$-torsion too.
\end{Remark}

\begin{Remark}
We only need the hypotheses \ref{IRR} and \ref{p-Dis} to ascertain that $\rho_F$ satisfies the Panchishkin condition. Besides this, none of our proofs require these hypotheses. The reason we include the auxillary prime  number $l$ in $\Sigma$ is to deduce Proposition \ref{surjectivity-greenberg} and Proposition \ref{strict-pseudo} from Corollary 3.2.3 in \cite{greenberg2010surjectivity} and Proposition 4.2.1 in \cite{greenberg2014pseudonull}. Although conjecturally the divisor $\Div\left(\Sel_{\rho_{\pmb{4},3}}(\Q)^\vee\right)-\Div\left(H^0(G_\Sigma,D_{\rho_{\pmb{4},3}})^\vee\right)$ should be principal, none of our results address this or even whether it generates a torsion element in the divisor class group of $T[[\Gamma]]$.
\end{Remark}

\subsubsection*{Acknowledgements} The author would like to thank Ralph Greenberg for sharing his deep insights, for patiently answering all his questions and for being a constant source of encouragement. This article was written while the author was his student at the University of Washington; the article would not have come to fruition without the support provided by everyone in the Department of Mathematics at the University of Washington. The author would like to thank the organizers of the workshop ``New directions in Iwasawa theory'' for providing the author an opportunity to present this work and he is grateful for the facilities provided by the Banff Research Center during the workshop. The author would also like to thank both referees for a very thorough reading of the manuscript, for pointing out various inaccuracies and for making many detailed suggestions to improve the manuscript.

\subsubsection*{Outline} Section 1 is largely spent recalling Greenberg's results in Iwasawa theory. The results of Section 2 and 3 are related to the general setups involving specialization results and the control theorems respectively. The results of Sections 4, 5 and 6 are related to the setup of Dasgupta's factorization.

\subsubsection*{Terminology} We shall follow some standard terminology throughout this paper. Let $\RRR$ be an integrally closed domain. The divisor group of $\RRR$ is the free abelian group on the set of height one prime ideals of $\RRR$. To a finitely generated torsion $\RRR$-module $\M$, one can associate an element, $\Div(\M)$, in the divisor group of $\RRR$ following Chapter VII in \cite{bourbaki1989commutative}. Suppose $y=rs^{-1}$ is a non-zero element in the fraction field of $\RRR$, such that $r,s \in \RRR$. Then, $\Div(y)$ is defined to be $\Div\left(\frac{\RRR}{(r)}\right) - \Div\left(\frac{\RRR}{(s)}\right)$. Note also that if $0\rightarrow\M_1\rightarrow\M_2\rightarrow\M_3\rightarrow 0$ is a short exact sequence of $\RRR$-modules, then $\Div(\M_2) = \Div(\M_1) + \Div(\M_3)$. A finitely generated torsion $\RRR$-module $\mathcal{M}$ is said to be pseudo-null if $\Div(\mathcal{M})=0$. The divisor class group of $\RRR$ is the quotient of the divisor group of $\RRR$ by the subgroup $\mathrm{Prin}(\RRR)$, which is generated by $\Div\left(\frac{\RRR}{(r)}\right)$, for all non-zero elements $r$ in $\RRR$. The rank of a finitely-generated $\RRR$-module $\mathcal{M}$ equals the dimension of the vector space $\mathcal{M} \otimes_{\RRR} \Frac(\RRR)$ over $\Frac(\RRR)$, the fraction field of $\RRR$. The Pontryagin dual of a profinite ring $\TTT$ will be denoted by $\hat{\TTT}$, while the Pontryagin dual of a module $\mathcal{N}$ over such a profinite ring $\TTT$ will be denoted by $\mathcal{N}^\vee$. We will repeatedly use Pontryagin duality for modules over profinite rings as stated in Theorem 1.1.11 in \cite{neukirch2008cohomology}.

\subsubsection*{List of abbreviations}
\mbox{}\\

\begin{tabularx}{\textwidth}{|ll|ll|ll|ll|}
\ref{ad-tor}, & Page \pageref{ad-tor}  &
\ref{euler-inequality}, &  Page \pageref{euler-inequality}  &
\ref{filtration}, & Page \pageref{filtration} &
\ref{Fin-Proj}, & Page \pageref{Fin-Proj}   \\

\ref{IRR}, & Page \pageref{IRR} &
\ref{mck}, & Page \pageref{mck} &
\ref{mcq}, & Page \pageref{mcq} &
\ref{mainconj-1}, & Page \pageref{mainconj-1}\\

\ref{mainconj-2}, & Page \pageref{mainconj-2} &
\ref{mainconj-3}, & Page \pageref{mainconj-3} &
\ref{No-PN}, & Page \pageref{No-PN} &
\ref{1-ev}, & Page \pageref{1-ev}  \\

\ref{p-critical}, & Page \pageref{p-critical} &
\ref{p-Dis}, & Page \pageref{p-Dis} &
\ref{Tor}, & Page \pageref{Tor}  &
\ref{urban-inequality}, & Page \pageref{urban-inequality}
\end{tabularx}

\section{Review of Greenberg's results} \label{sel-cyc-deformations}

The purpose of this section is to recall results from Greenberg's foundational works \cite{MR2290593}, \cite{greenberg2010surjectivity} and \cite{greenberg2014pseudonull} on Galois cohomology groups and Selmer groups, pertinent to Iwasawa theory. We shall mainly be interested in restating the results there in a manner useful to our purposes. Since the section is intended to serve as an exposition to Greenberg's works, we will simply sketch the proofs, instead indicating references where the proofs are discussed more elaborately.

\subsection{The general setup}
Let $\RRR$ be a finite integral extension of $\Z_p[[u_1,\dotsc,u_n]]$ and assume it is an integrally closed domain. As in the introduction, let $\Sigma$ be a finite set of primes in $\Q$ containing $p$, $\infty$ and a non-archimedean prime $l \neq p$. Let $\Sigma_0 =\Sigma \setminus \{p\}$. We shall associate a primitive Selmer group $\S_{\varrho}(\Q)$ and a non-primitive Selmer group $\Sel_{\varrho}(\Q)$  to the Galois representation $\varrho : G_\Sigma \rightarrow \Gl_d(\RRR)$, whose associated Galois lattice $\LLL_\varrho$ is free over $\RRR$, and to which we can associate a short exact sequence of free $\RRR$-modules that~is~$\Gal{\overline{\Q}_p}{\Q_p}$-equivariant.

\begin{align} \label{filtration}
\tag{Fil-$\varrho$} 0\rightarrow \Fil^+\LLL_\varrho \rightarrow \LLL_\varrho \rightarrow \frac{\LLL_\varrho}{\Fil^+\LLL_\varrho} \rightarrow 0.
\end{align}

We shall call the short exact sequence \ref{filtration} as the filtration\footnote{The filtration \ref{filtration} is an extra datum for the Galois representation $\varrho$. A prototypical example to keep in mind is that of an elliptic curve $E$, defined over $\Q$, that has good ordinary reduction at $p$. In this case, one has a short exact sequence $0 \rightarrow \ker(j) \rightarrow T_p(E) \xrightarrow {j} T_p(\overline{E}) \rightarrow 0$ of free $\Z_p$-modules that is $\Gal{\overline{\Q}_p}{\Q_p}$-equivariant. Here, $T_p(E)$ and $T_p(\overline{E})$ are the $p$-adic Tate modules associated to the elliptic curves $E$ and $\overline{E}$ (the reduction of the elliptic curve $E$ at $p$) respectively, and $j$ is the natural reduction map.} associated to $\varrho$. We also define the discrete modules $\D_{\varrho}$ and $\Fil^+\D_{\varrho}$ to equal $\LLL_\varrho \otimes_\RRR \hat{\RRR}$ and $\Fil^+\LLL_\varrho \otimes_\RRR \hat{\RRR}$ respectively. The (discrete) Selmer groups associated to $\varrho$ are defined below.
{\small
\begin{align*}
\Sel_{\varrho}(\Q) &:= \ker \left( H^1(G_\Sigma, \D_{\varrho}) \xrightarrow {\phi_{\varrho}^{\Sigma_0}} H^1\left(I_p,\frac{\D_{\varrho}}{\Fil^+D_{\varrho}} \right)^{\Gamma_p} \right),\\
\S_{\varrho}(\Q) &:= \ker \left( H^1(G_\Sigma, \D_{\varrho}) \xrightarrow {\phi_{\varrho}} H^1\left(I_p,\frac{\D_{\varrho}}{\Fil^+D_{\varrho}} \right)^{\Gamma_p} \times \prod \limits_{\nu \in \Sigma_0} \ \Loc(\nu,\varrho) \right).
\end{align*}
}
Here, for each prime $\nu \in \Sigma$, we let $\Loc(\nu,\varrho)$ denote $H^1(I_\nu,\D_{\varrho})^{\Gamma_\nu}$.
\begin{Remark}
Our choice of fonts in this paper will be guided by the following convention.  For results specifically pertaining to Dasgupta's factorization, we will use a ``normal'' font to denote the corresponding Galois representations, rings, modules etc (e.g. $\rho$, $R$, $L$, $D$ etc). While describing results of a general kind, we will use a ``curly'' or ``calligraphic'' font to denote the corresponding objects (e.g. $\varrho$, $\mathcal{R}$, $\mathcal{L}$, $\mathcal{D}$ etc).
\end{Remark}

We will also need to introduce the Galois representation $\varrho^* : G_\Sigma \rightarrow \Gl_d(\RRR)$  whose associated lattice $\LLL_\varrho^*$ is defined by $\Hom_\RRR \left(\LLL_\varrho,\RRR(\chi_p) \right)$. The discrete module $\D^*_\varrho$ associated to $\varrho^*$ is given by $\LLL_\varrho^* \otimes_\RRR \hat{\RRR}$. Using Greenberg's terminology in \cite{MR2290593}, we observe that the $\RRR$-modules $\D_{\varrho}$ and $\D^*_{\varrho}$  are $\RRR$-cofree, and hence $\RRR$-coreflexive and $\RRR$-codivisible. Note that a discrete $\RRR$-module $\D$ is said to be cofree, coreflexive, codivisible respectively if the $\RRR$-module $\D^\vee$ is free, reflexive, torsion-free respectively.

\subsection{Cyclotomic deformations}

We shall fix a topological generator $\gamma_0$ of $\Gamma$ throughout the paper. We shall consider the completed group ring $\RRR[[\Gamma]]$, which is an integrally closed domain. There is a non-canonical isomorphism $\RRR[[\Gamma]] \cong \RRR[[s]]$, obtained by sending the topological generator $\gamma_0$ of $\Gamma$ to $s~+~1$.  This allows us to view $\RRR[[\Gamma]]$ as a finite integral extension of $\Z_p[[u_1,\dotsc,u_n,s]]$. We shall define a Galois representation $\varrho \otimes \kappa^{-1} : G_\Sigma \rightarrow \Gl_d(\RRR[[\Gamma]])$, that is related to the cyclotomic deformation of $\varrho$. Cyclotomic deformations frequently arise in Iwasawa theory. We let $\RRR[[\Gamma]](\kappa^{-1})$ denote the free $\RRR[[\Gamma]]$-module on which $G_\Sigma$ acts by the character $\kappa^{-1}$. Similarly, we let $\hat{\RRR[[\Gamma]]}(\kappa^{-1})$ denote $\RRR[[\Gamma]](\kappa^{-1}) \otimes_{\RRR[[\Gamma]]} \hat{\RRR[[\Gamma]]}$. The  deformation $\varrho \otimes \kappa^{-1}$ is given by the action of $G_\Sigma$ on $\LLL_{\varrho \otimes \kappa^{-1}}$ (defined below). We can associate the following free $\RRR[[\Gamma]]$-modules to $\varrho \otimes \kappa^{-1}$.
{
\footnotesize
\begin{align*}
\LLL_{\varrho \otimes \kappa^{-1}} := \LLL_\varrho \otimes \RRR[[\Gamma]](\kappa^{-1}), & \quad  \Fil^+\LLL_{\varrho \otimes \kappa^{-1}} := \Fil^+\LLL_\varrho \otimes \RRR[[\Gamma]](\kappa^{-1}), && \quad \frac{\LLL_{\varrho \otimes \kappa^{-1}}}{\Fil^+\LLL_{\varrho \otimes \kappa^{-1}}} := \frac{\LLL_\varrho}{\Fil^+\LLL_\varrho} \otimes \RRR[[\Gamma]](\kappa^{-1}), \\
\D_{\varrho \otimes \kappa^{-1}} := \LLL_{\rho } \otimes \hat{\RRR[[\Gamma]]}(\kappa^{-1}), & \quad  \Fil^+\D_{\varrho \otimes \kappa^{-1}} := \Fil^+\LLL_{\varrho } \otimes \hat{\RRR[[\Gamma]]}(\kappa^{-1}), && \quad \frac{\D_{\varrho \otimes \kappa^{-1}}}{\Fil^+\D_{\varrho \otimes \kappa^{-1}}} := \frac{\LLL_{\varrho}}{\Fil^+\LLL_{\varrho}} \otimes \hat{\RRR[[\Gamma]]}(\kappa^{-1}).
\end{align*}
}

All the tensor products given above are taken over the ring $\RRR$. Note that the cyclotomic deformation of $\varrho$ (as defined in \cite{greenberg1994iwasawa}) is given by the action of $G_\Sigma$ on $\LLL_{\varrho} \otimes_\RRR \RRR[[\Gamma]](\kappa)$. One can form the primitive Selmer group $\S_{\varrho \otimes \kappa^{-1}}(\Q)$ and the non-primitive Selmer group $\Sel_{\varrho \otimes \kappa^{-1}}(\Q)$ corresponding to the filtration given below.
\begin{align}
\tag{Fil-$\varrho \otimes \kappa^{-1}$} 0 \rightarrow \Fil^+\LLL_{\varrho \otimes \kappa^{-1}} \rightarrow \LLL_{\varrho \otimes \kappa^{-1}} \rightarrow \frac{\LLL_{\varrho \otimes \kappa^{-1}}}{\Fil^+\LLL_{\rho \otimes \kappa^{-1}}} \rightarrow 0.
\end{align}

To relate the discrete Galois modules $\D_{\varrho \otimes \kappa^{-1}}$ and $\D_{\varrho}$,  we have the following isomorphism of discrete $\RRR[[\Gamma]]$-modules (described in section 3 of \cite{greenberg1994iwasawa}):
\begin{align}\label{iso-for-cyc}
\D_{\varrho \otimes \kappa^{-1}}  \cong \Hom_\RRR \bigg( \RRR[[\Gamma]](\kappa), \D_\varrho \bigg),  \qquad \D^*_{\rho \otimes \kappa^{-1}} \cong \Hom_\RRR \bigg( \RRR[[\Gamma]](\kappa^{-1}), \D^*_\varrho \bigg).
\end{align}
We will need to consider the $\RRR$-linear involution $\iota : \RRR[[\Gamma]] \rightarrow \RRR[[\Gamma]]$ obtained by sending $\gamma_0$ to $\gamma_0^{-1}$. Given an $\RRR[[\Gamma]]$-module $\M$, we will define a new $\RRR[[\Gamma]]$-module $\M^\iota$. The underlying $\RRR$-module structure on $\M^\iota$ is the same. There is a new action of $\Gamma$ on $\M^\iota$; the topological generator $\gamma_0$ now acts on $\M^\iota$ via $\iota(\gamma_0)$. Using the isomorphisms in (\ref{iso-for-cyc}), we obtain the following proposition which is described in Section 3 of \cite{greenberg1994iwasawa}.

\begin{proposition}\label{greenberg-cyc}
We have the following isomorphisms of $\R[[\Gamma]]$-modules :
{\small \begin{align*}
& H^0(G_\Sigma ,\D_{\varrho \otimes \kappa^{-1}}) \cong H^0(\Gal{\Q_\Sigma}{\Q_\infty},\D_\varrho),&&H^0(G_\Sigma ,\D^*_{\varrho \otimes \kappa^{-1}}) \cong H^0(\Gal{\Q_\Sigma}{\Q_\infty},\D^*_\varrho)^\iota \\  & H^0(\Gal{\overline{\Q}_p}{\Q_p} ,\D_{\varrho \otimes \kappa^{-1}}) \cong H^0(G_{\eta_p},\D_\varrho),  && H^0(I_p ,\D_{\varrho \otimes \kappa^{-1}}) \cong H^0(I_{\eta_p},\D_\varrho), \\
& H^0(\Gal{\overline{\Q}_p}{\Q_p} ,\Fil^+ \D_{\varrho \otimes \kappa^{-1}}) \cong H^0(G_{\eta_p},\Fil^+\D_\varrho),  && H^0(I_p ,\Fil^+\D_{\varrho \otimes \kappa^{-1}}) \cong H^0(I_{\eta_p},\Fil^+\D_\varrho), \\ &
  H^0\left(\Gal{\overline{\Q}_p}{\Q_p} ,\frac{\D_{\varrho \otimes \kappa^{-1}}}{\Fil^+\D_{\varrho \otimes \kappa^{-1}}}\right) \cong H^0\left(G_{\eta_p},\frac{\D_\varrho}{\Fil^+\D_\varrho}\right), &&
  H^0\left(I_p ,\frac{\D_{\varrho \otimes \kappa^{-1}}}{\Fil^+\D_{\varrho \otimes \kappa^{-1}}}\right) \cong H^0\left(I_{\eta_p},\frac{\D_\varrho}{\Fil^+\D_\varrho}\right).
\end{align*}
}
For each $\nu$ in $\Sigma_0$, we also have the isomorphism
\begin{align*}
H^0\left( \Gal{\overline{\Q}_{\nu}}{\Q_\nu},\D_{\varrho \otimes \kappa^{-1}}\right) \cong \Ind^{\Gamma}_{\Delta_{\eta_\nu}}H^0\left(G_{\eta_\nu},\D_{\varrho} \right).
\end{align*}
Here, for each $\nu \in \Sigma$, we choose a single prime $\eta_\nu$  in $\Q_\infty$ lying above $\nu$. The decomposition and inertia subgroups at $\eta_\nu$ are denoted by $G_{\eta_\nu}$ and $I_{\eta_\nu}$. The quotient $\Gal{\overline{\Q}_\nu}{\Q_\nu}/G_{\eta_\nu}$, denoted by $\Delta_{\eta_\nu}$, is a subgroup of finite index inside $\Gamma$.

The Pontryagin duals of all the $0^{th}$-cohomology groups, appearing above, are finitely generated over $\RRR$.
\end{proposition}

Let $d^+$ denote the dimension of the $+1$ eigenspace for the action of complex conjugation on $\varrho$.
Throughout this section, we shall assume that the following conditions hold.
\begin{align} \label{p-critical}
\tag{$p$-critical}
\mathrm{Rank}_{\RRR[[\Gamma]]}\left(\Fil^+\LLL_{\varrho \otimes \kappa^{-1}}\right) = d^+.
\end{align}
\begin{enumerate}[leftmargin=3cm, style=sameline, align=left, label=\textsc{(TOR)}, ref=\textsc{TOR}]
\item \label{Tor} The $\RRR[\Gamma]]$-module $\S_{\varrho \otimes \kappa^{-1}}(\Q)^\vee$ is torsion.
\end{enumerate}
\begin{Remark}
When $\varrho \otimes \kappa^{-1}$ satisfies an additional hypothesis (called the ``Panchishkin condition''), one can associate (conjecturally) a $p$-adic $L$-function $\theta_{\varrho \otimes \kappa^{-1}}$ and a \textit{main conjecture} to $\varrho \otimes \kappa^{-1}$ . The hypotheses \ref{p-critical} and \ref{Tor} come into play. See Greenberg's work \cite{greenberg1994iwasawa} for a precise description of the ``Panchishkin condition''. The Galois representations $\rho_{\pmb{4},3}$, $\rho_{\pmb{3},2}$ and $\rho_{\pmb{1},2}$ satisfy the Panchishkin condition. The Galois representation $\pi \circ \rho_{\pmb{4},3}$ does not satisfy the Panchishkin condition.
\end{Remark}

\subsection{Galois cohomology groups}
We shall now proceed to describe the results of Greenberg \cite{MR2290593} regarding Galois cohomology groups.  The proof of Proposition \ref{local-cohomology-not-p} is essentially described in the proof of Proposition 3.2 of \cite{greenberg1994iwasawa}. Hence, instead of providing the proof of Proposition \ref{local-cohomology-not-p}, we shall simply provide the outline of the proof.  Let $\nu \in \Sigma_0$. Let $\eta_\nu$ denote a prime in $\Q_\infty$ lying above $\nu$. Let $\Q_{\nu,\infty}$ denote the cyclotomic $\Z_p$ extension of $\Q_\nu$ and let $G_{\eta_\nu}$ denote the decomposition subgroup of $\Gal{\overline{\Q}_\nu}{\Q_{\nu,\infty}}$. Let $\Delta_{\eta_\nu}$, a subgroup of $\Gamma$,  denote the decomposition subgroup corresponding to $\eta_\nu$ lying above $\nu$. Note that both $\Delta_{\eta_\nu}$ and $\Gamma$ have $p$-cohomological dimension equal to $1$. The key point is to notice that \begin{align*}
H^1\big(\Gal{\Q_{\nu,\infty}}{\Q_\nu}, H^0\left(\Gal{\overline{\Q}_\nu}{\Q_{\nu,\infty}},\D_{\varrho\otimes \kappa^{-1}}\right) \big)=0, \qquad H^1\left(\Delta_{\eta_\nu}, H^0(G_{\eta_\nu},\D_{\varrho\otimes \kappa^{-1}})\right)=0.
\end{align*}
\begin{proposition} \label{local-cohomology-not-p}
Let $\nu \in \Sigma_0$. The natural restriction map gives us an isomorphism \begin{align}\label{loc-decomp-not-p-iso}
H^1(\Gal{\overline{\Q}_\nu}{\Q_\nu},\D_{\varrho \otimes \kappa^{-1}}) \xrightarrow {\cong} \Loc(\nu,\varrho \otimes \kappa^{-1}).
\end{align}
Both these groups are  isomorphic to $$\Ind^{\Gamma}_{\Delta_{\eta_\nu}}H^1(G_{\eta_\nu}, \D_\varrho) .$$
\end{proposition}

\begin{corollary}\label{H1-torsion-local-factors}
Let $\nu \in \Sigma_0$. The $\RRR[[\Gamma]]$-module $\Loc(\nu,\varrho \otimes \kappa^{-1})^\vee$ is finitely generated over $\RRR$. As a result, there exists a monic polynomial $h(s)$ in $\RRR[s]$ such that $h(\gamma_0)$ annihilates $\Loc(\nu,\varrho \otimes \kappa^{-1})^\vee$.
\end{corollary}

\begin{proof}
Since $\Delta_{\eta_\nu}$ is of finite index in $\Gamma$, it suffices to show that $H^1(G_{\eta_\nu},\D_\varrho)^\vee$ is finitely generated over $\RRR$. This can be done by using results from local class field theory. The Galois group $\Gal{\Q_{\nu}^{ur}}{\Q_{\nu,\infty}}$ is of profinite order prime to $p$. Here, $\Q_{\nu,\infty}$ is the cyclotomic $\Z_p$-extension of $\Q_\nu$ and $\Q_\nu^{ur}$ is the maximal unramified extension of $\Q_\nu$. This gives us the isomorphism.
\begin{align*}
H^1(G_{\eta_\nu},\D_{\varrho}) \cong H^0(\Gal{\Q^{ur}_\nu}{\Q_{\nu,\infty}},H^1(I_\nu,\D_\varrho)).
\end{align*}
To prove the corollary, it is enough to show that $H^1(I_\nu,\D_\varrho)^\vee$ is finitely generated over $\RRR$. Local class field theory lets us obtain a short exact sequence $0 \rightarrow W_\nu \rightarrow I_\nu \rightarrow P_\nu \rightarrow 0$ where $W_\nu$ is a profinite group of profinite order prime to $p$ and $P_\nu$ is a profinite group that is isomorphic to $\Z_p$. The action of $W_\nu$ factors through a finite group of order prime to $p$. These observations let us obtain the following isomorphisms of $\RRR$-modules:
\begin{align*}
H^1(I_\nu,\D_{\varrho}) \cong H^1\left(P_\nu, H^0(W_\nu,\D_\varrho)\right) \cong H^1\left(P_\nu, \left(\D_\varrho\right)_{W_\nu} \right) \cong  \left(\D_{\varrho}\right)_{I_\nu}.
\end{align*}
Here, $\left(\D_\varrho\right)_{W_\nu}$ and $\left(\D_{\varrho}\right)_{I_\nu}$ denote the maximal quotient of $\D_{\varrho}$ on which $W_\nu$ and $I_\nu$ respectively act trivially. Since $H^1(I_\nu,\D_{\varrho})^\vee$ is isomorphic to a sub-module of $\D_\varrho^\vee$ (which is finitely generated over $\RRR$), the corollary follows.
\end{proof}

\begin{proposition} \label{h2-cohomology}
Let $\nu \in \Sigma_0$. The second cohomology group $H^2(\Gal{\overline{\Q}_\nu}{\Q_\nu},\D_{\varrho \otimes \kappa^{-1}})$ equals $0$. If we let $\mathcal{N}$ equal $\D_{\varrho \otimes \kappa^{-1}}$, $\Fil^+\D_{\varrho \otimes \kappa^{-1}}$ or $\frac{\D_{\rho \otimes \kappa^{-1}}}{\Fil^+\D_{\varrho \otimes \kappa^{-1}}}$, then $H^2(\Gal{\overline{\Q}_p}{\Q_p},\mathcal{N})$ equals $0$.
\end{proposition}

Once again, instead of providing the proof of Proposition \ref{h2-cohomology}, we shall simply mention that Proposition \ref{h2-cohomology} can be deduced from the arguments given in Section 5 of \cite{greenberg2010surjectivity}. The arguments involved in the proof combine local duality, Proposition 3.10 in \cite{MR2290593} along with Proposition \ref{greenberg-cyc}. Note that the statement of local duality, that we need, is given below (see Section 0.3 of Nekov\'a\v{r}'s Selmer complexes \cite{nekovar2006selmer}).
\begin{align*}
H^i\left(\Gal{\overline{\Q}_\nu}{\Q_{\nu}}, \D_{\rho}\right)^\vee \cong  H^{2-i}\left(\Gal{\overline{\Q}_\nu}{\Q_{\nu}}, \LLL^*_{\rho}\right), \qquad  \text{ for all $\nu \in \Sigma$, for all $i \in \{0,1,2\} $}.
\end{align*}

Now consider the prime $p$. The map $$\alpha_p : H^1(\Gal{\overline{\Q}_p}{\Q_p},\D_{\varrho \otimes \kappa^{-1}}) \rightarrow H^1\left(I_p,\frac{\D_{\varrho \otimes \kappa^{-1}}}{\Fil^+\D_{\varrho \otimes \kappa^{-1}}}\right)^{\Gamma_p}$$ can be written as $\alpha_p = \sigma_p \circ \beta_p$, where the maps $\beta_p$ and $\sigma_p$ are given below.
\begin{align*}
&\beta_p : H^1(\Gal{\overline{\Q}_p}{\Q_p},\D_{\varrho \otimes \kappa^{-1}}) \rightarrow H^1\left(\Gal{\overline{\Q}_p}{\Q_p},\frac{\D_{\varrho \otimes \kappa^{-1}}}{\Fil^+\D_{\varrho \otimes \kappa^{-1}}}\right), \\
& \sigma_p : H^1\left(\Gal{\overline{\Q}_p}{\Q_p},\frac{\D_{\rho \otimes \kappa^{-1}}}{\Fil^+\D_{\varrho \otimes \kappa^{-1}}}\right) \rightarrow H^1\left(I_p,\frac{\D_{\varrho \otimes \kappa^{-1}}}{\Fil^+\D_{\varrho \otimes \kappa^{-1}}}\right)^{\Gamma_p}.
\end{align*}
Since $H^2(\Gal{\overline{\Q}_p}{\Q_p}, \Fil^+\D_{\varrho \otimes \kappa^{-1}})=0$, the map $\beta_p$ is surjective. The inflation-restriction exact sequence and the fact that $\Gamma_p$ has $p$-cohomological dimension $1$ let us conclude that $\sigma_p$ is also surjective. Consequently, $\alpha_p$ is surjective too. This gives us the following lemma:
\begin{lemma} \label{local-cohomology-p}
We have a short exact sequence $0\rightarrow \ker(\beta_p) \rightarrow \ker(\alpha_p) \rightarrow \ker(\sigma_p) \rightarrow 0$ and the following isomorphism:
\begin{align*}
H^1\left(I_p,\frac{\D_{\varrho \otimes \kappa^{-1}}}{\Fil^+\D_{\varrho \otimes \kappa^{-1}}}\right)^{\Gamma_p} \cong \frac{H^1(\Gal{\overline{\Q}_p}{\Q_p},\D_{\varrho \otimes \kappa^{-1}})}{\ker(\alpha_p)}.
\end{align*}
\end{lemma}

The Selmer group is defined as the kernel of a global-to-local map. In Greenberg's works (\cite{greenberg2010surjectivity} and \cite{greenberg2014pseudonull}), the local factors are described as quotients of local cohomology groups involving the decomposition subgroup. We prefer to describe the local factors in terms of local cohomology groups involving the inertia subgroup.  Proposition \ref{local-cohomology-not-p} and Lemma \ref{local-cohomology-p} in fact also provide a relationship between these descriptions.

To apply many of Greenberg's results in \cite{greenberg2010surjectivity} and \cite{greenberg2014pseudonull}, a hypothesis called ``Weak Leopoldt conjecture'' (sometimes also called ``Hypothesis L'') needs to be verified. The Weak Leopoldt conjecture for $\varrho\otimes\kappa^{-1}$ states that the Pontryagin dual of
\begin{align*}
\Sha^2\left(\D_{\varrho \otimes \kappa^{-1}}\right) &:= \ker\left( H^2(G_\Sigma, \D_{\varrho \otimes \kappa^{-1}}) \rightarrow \prod_{\nu \in \Sigma} H^2(\Gal{\overline{\Q}_\nu}{\Q_\nu},\D_{\varrho \otimes \kappa^{-1}}) \right)
\end{align*}
is $\RRR[[\Gamma]]$-torsion. We will now show that the Weak Leopoldt conjecture for $\varrho\otimes\kappa^{-1}$ holds whenever the hypotheses \ref{Tor} and \ref{p-critical} hold. \\

Let $i \in \{0,1,2\}$. Let $\nu \in \Sigma$.  Let us label the ranks of various cohomology groups.
\begin{align*}
& h_i:=\mathrm{Rank}_{\RRR[[\Gamma]]}\left(H^i(G_\Sigma, \D_{\varrho\otimes\kappa^{-1}})^\vee   \right),   && h_i^{p,-}:=  \mathrm{Rank}_{\RRR[[\Gamma]]}\left(H^i\left(\Gal{\overline{\Q}_p}{\Q_p}, \frac{\D_{\varrho\otimes\kappa^{-1}}}{\Fil^+\D_{\varrho\otimes\kappa^{-1}}}\right)^\vee   \right), \\ &  h_i^{(p)} :=\mathrm{Rank}_{\RRR[[\Gamma]]}\left(H^i(\Gal{\overline{\Q}_p}{\Q_p}, \D_{\varrho\otimes\kappa^{-1}})^\vee   \right), && h_i^{(\nu)}:=  \mathrm{Rank}_{\RRR[[\Gamma]]}\left(H^i(\Gal{\overline{\Q}_\nu}{\Q_\nu}, \D_{\varrho\otimes\kappa^{-1}})^\vee   \right).
\end{align*}

Now let $\nu \in \Sigma_0$. By Proposition \ref{greenberg-cyc}, we obtain the following equality $$h_0=h_0^{p,-}=h_0^{(\nu)}=0.$$ Note that the dimension of the Galois representation $\varrho$ and the dimension of its $+1$-eigenspace for complex conjugation equal $d$ and $d^+$ respectively. By Proposition \ref{h2-cohomology}, we obtain the following equality $$h_2^{(p)}=h_2^{p,-}=h_2^{(\nu)}=0.$$
These observations along with the formulas for the global Euler-Poincar\'e characteristics (Proposition 4.1 in \cite{MR2290593}) and local Euler-Poincar\'e characteristics (Proposition 4.2 in \cite{MR2290593}) then give us the following equalities:
\begin{align*}
h_1 &=  (d-d^+) + h_2 .\\
h_1^{p,-} &=   (d-d^+), \ h_1^{(\nu)}=0. \\
h_2 &= \mathrm{Rank}_{\RRR[[\Gamma]]}\left(\Sha^2\left(\D_{\varrho \otimes \kappa^{-1}}\right)^\vee \right).
\end{align*}
In fact, by Proposition \ref{h2-cohomology},   $\Sha^2\left(\D_{\varrho \otimes \kappa^{-1}}\right)$ equals $H^2(G_\Sigma, \D_{\varrho\otimes\kappa^{-1}})$. Since $\sigma_p$ is surjective, the quantity $h_1^{p,-}$ is greater than or equal to the $\RRR[[\Gamma]]$-rank of the Pontryagin dual of $H^1\left(I_p,\frac{\D_{\varrho\otimes\kappa^{-1}}}{\Fil^+\D_{\varrho\otimes\kappa^{-1}}} \right)^{\Gamma_p}$. As a result of the hypothesis \ref{Tor}, we obtain the following set of implications:
\begin{align*}
&& h_1 +  \mathrm{Rank}_{\RRR[[\Gamma]]}\left(\coker(\phi_{\varrho\otimes \kappa^{-1}})^\vee \right) & \leq  h_1^{p,-},
\\  \implies &&  (d-d^+) + h_2 +  \mathrm{Rank}_{\RRR[[\Gamma]]}\left(\coker(\phi_{\varrho\otimes \kappa^{-1}})^\vee \right) & \leq (d-d^+), \\
   \implies && \mathrm{Rank}_{\RRR[[\Gamma]]}\left(\Sha^2\left(\D_{\varrho \otimes \kappa^{-1}}\right)^\vee\right) +  \mathrm{Rank}_{\RRR[[\Gamma]]}\left(\coker(\phi_{\varrho\otimes \kappa^{-1}})^\vee\right) & \leq 0,\\
 \implies && \mathrm{Rank}_{\RRR[[\Gamma]]}\left(\Sha^2\left(\D_{\varrho \otimes \kappa^{-1}}\right)^\vee\right) = \mathrm{Rank}_{\RRR[[\Gamma]]}\left(\coker(\phi_{\varrho\otimes \kappa^{-1}})^\vee\right) & = 0.
\end{align*}
By Proposition 5.2.3 in \cite{greenberg2010surjectivity}, the $\RRR[[\Gamma]]$-module $H^2(G_\Sigma,\D_{\varrho\otimes \kappa^{-1}})^\vee$, which equals $\Sha^2\left(\D_{\varrho \otimes \kappa^{-1}}\right)^\vee$, is torsion-free.  Since $\coker(\phi^{\Sigma_0}_{\varrho \otimes \kappa^{-1}})$ is a quotient of $\coker(\phi_{\varrho \otimes \kappa^{-1}})$, we get the following proposition:
\begin{proposition}\label{coker-weak}The following statements hold whenever the hypotheses \ref{Tor} and \ref{p-critical} hold:
\begin{enumerate}
\item The $\RRR[[\Gamma]]$-modules $\coker\left(\phi_{\varrho \otimes \kappa^{-1}}\right)^\vee$ and $\coker\left(\phi^{\Sigma_0}_{\varrho \otimes \kappa^{-1}}\right)^\vee$ are torsion.
\item $H^2\left(G_\Sigma,\D_{\varrho \otimes \kappa^{-1}}\right)$ equals $0$. As a result, the Weak Leopoldt conjecture for $\varrho\otimes \kappa^{-1}$ holds.
\end{enumerate}

\end{proposition}

\subsection{Cokernel of the map defining Selmer groups}

We shall now proceed to describe the results of Greenberg \cite{greenberg2010surjectivity} concerning the cokernel of $\phi_{\varrho\otimes\kappa^{-1}}$. These results allow us to analyze the difference between the primitive and the non-primitive Selmer group. Using results from \cite{greenberg2010surjectivity} and Proposition \ref{coker-weak}, we obtain the following result:

\begin{proposition}[Greenberg, \cite{greenberg2010surjectivity}] \label{surjectivity-greenberg}
Suppose the hypotheses \ref{Tor} and \ref{p-critical} hold. The map $\phi^{\Sigma_0}_{\varrho \otimes \kappa^{-1}}$ defining the non-primitive Selmer group is surjective. As a result, we have the short exact sequence
\begin{align} \label{ses-primitive-non}
0 \rightarrow \S_{\varrho \otimes \kappa^{-1}}(\Q) \rightarrow \Sel_{\varrho \otimes \kappa^{-1}}(\Q) \rightarrow \prod_{\nu \in \Sigma_0} \Loc(\nu,\varrho\otimes \kappa^{-1}) \rightarrow  \coker\left(\phi_{\rho \otimes \kappa^{-1}}\right)\rightarrow 0.
\end{align}
For the map $\phi_{\varrho \otimes \kappa^{-1}}$ defining the primitive Selmer group, we have the following relation: \begin{align} \label{coker-sub}
\coker\left(\phi_{\varrho \otimes \kappa^{-1}}\right)^\vee \subset H^1(G_\Sigma, \LLL^*_{\varrho \otimes \kappa^{-1}})_{\mathrm{tor}}.
\end{align}

In addition, if $\ker(\alpha_p)$ is a sub-module of the maximal $\RRR[[\Gamma]]$-divisible subgroup of the local cohomology group $H^1\left(\Gal{\overline{\Q}_p}{\Q_p},\D_{\varrho \otimes \kappa^{-1}}\right)$, we obtain isomorphism
\begin{align} \label{iso-cokernel-surjectivity}
\coker(\phi_{\varrho \otimes \kappa^{-1}})^\vee \cong H^1(G_\Sigma, \LLL^*_{\varrho \otimes \kappa^{-1}})_{\mathrm{tor}}.
\end{align}
\end{proposition}

\begin{proof}
The surjectivity of $\phi^{\Sigma_0}_{\varrho \otimes \kappa^{-1}}$ follows from Proposition 3.2.1 in \cite{greenberg2010surjectivity}. The short exact sequence (\ref{ses-primitive-non}) follows from Corollary 3.2.3, Remark 3.2.4 and Corollary 3.2.5 in \cite{greenberg2010surjectivity}. Equation (\ref{coker-sub}) follows from Remark 3.1.3 in \cite{greenberg2010surjectivity}. The isomorphism (\ref{iso-cokernel-surjectivity}) follows from Proposition 2.3.1 and Proposition 3.1.1 in \cite{greenberg2010surjectivity}.
\end{proof}

To investigate the $\RRR[[\Gamma]]$-torsion submodule of $H^1(G_\Sigma, \LLL^*_{\varrho \otimes \kappa^{-1}})$, which has been denoted by $H^1(G_\Sigma, \LLL^*_{\varrho \otimes \kappa^{-1}})_{\mathrm{tor}}$, we will follow the arguments given in Section 2 of  \cite{greenberg2010surjectivity}. Let $\xi$ be a non-zero annihilator of $H^1(G_\Sigma, \LLL^*_{\varrho \otimes \kappa^{-1}})_{\mathrm{tor}}$ in the ring $\RRR$. We have  $$H^1(G_\Sigma, \LLL^*_{\varrho \otimes \kappa^{-1}})_{\mathrm{tor}} = H^1(G_\Sigma, \LLL^*_{\varrho \otimes \kappa^{-1}})[\xi].$$  Note that local duality tells us that $H^0(\Gal{\overline{\Q}_p}{\Q_p}, \LLL^*_{\varrho \otimes \kappa^{-1}}) = H^2(\Gal{\overline{\Q}_p}{\Q_p},\D_{\varrho \otimes \kappa^{-1}})=0$. Hence, $H^0\left(G_\Sigma, \LLL^*_{\varrho \otimes \kappa^{-1}}\right)$ also equals zero. The short exact sequence $$0 \rightarrow \LLL^*_{\varrho \otimes \kappa^{-1}} \xrightarrow {\xi} \LLL^*_{\varrho \otimes \kappa^{-1}} \rightarrow \frac{\LLL^*_{\varrho \otimes \kappa^{-1}}}{ \xi \LLL^*_{\varrho \otimes \kappa^{-1}}} \rightarrow 0$$ then gives us the isomorphism
\begin{align}\label{useful-tor-h1}
H^1\left(G_\Sigma, \LLL^*_{\varrho \otimes \kappa^{-1}}\right)_{\mathrm{tor}} = H^1\left(G_\Sigma, \LLL^*_{\varrho \otimes \kappa^{-1}}\right)[\xi] \cong H^0\left(G_\Sigma, \frac{\LLL^*_{\rho \otimes \kappa^{-1}} }{\xi \LLL^*_{\varrho \otimes \kappa^{-1}}}\right).
\end{align}
For the following lemmas, we will introduce one bit of notation. For every prime ideal $\p$ in $\RRR[[\Gamma]]$, we let $k_\p$ the fraction field of $\frac{\RRR[[\Gamma]]}{\p}$.
\begin{lemma}
If $H^0(G_{\Sigma},\LLL^*_{\varrho \otimes_{\RRR[[\Gamma]]} \kappa^{-1}}\otimes k_\p)$ equals $0$ for all height one prime ideals $\p$ in $\RRR[[\Gamma]]$, then $\coker(\phi_{\varrho \otimes \kappa^{-1}})$ equals $0$.
\end{lemma}
\begin{proof}

Consider the short exact sequence of $\RRR[[\Gamma]]$-modules, with the trivial $G_\Sigma$-action (which follows Proposition 5.1.7 in \cite{neukirch2008cohomology} and from 2C (2) in \cite{MR2290593})
\begin{align}\label{ses-tgamma}
0 \rightarrow \frac{\RRR[[\Gamma]]}{(\xi)} \rightarrow \bigoplus_{ht(\p_i)=1} \frac{\RRR[[\Gamma]]}{\p_i^{n_i}} \rightarrow \mathcal{Z} \rightarrow 0,
\end{align}
where the sum is taken over the height one prime ideals $\p_i$ of $\RRR[[\Gamma]]$ in the support of $\frac{\RRR[[\Gamma]]}{(\xi)}$ and where $\mathcal{Z}$ is a $\RRR[[\Gamma]]$ pseudo-null module. The integers $n_i$ are non-negative. After we tensor the sequence (\ref{ses-tgamma}) over $\RRR[[\Gamma]]$ with the $G_\Sigma$-module $\LLL^*_{\varrho \otimes \kappa^{-1}}$, we can take $G_\Sigma$-invariants. To prove the  lemma it suffices to show that for every height one prime ideal $\p$ in $\RRR[[\Gamma]]$ and every natural number $n$,
$$H^0\left(G_\Sigma, \frac{\LLL^*_{\varrho \otimes \kappa^{-1}}}{\p^{n} \LLL^*_{\varrho \otimes \kappa^{-1}}}\right)\stackrel{?}{=}0.$$
Let us fix such a height one prime ideal $\p$ and a natural number $n$. Consider the tower of $G_\Sigma$-modules: $$\LLL^*_{\varrho \otimes \kappa^{-1}} \supset \p \LLL^*_{\varrho \otimes \kappa^{-1}} \supset \dotsb \supset \p^n \LLL^*_{\varrho \otimes \kappa^{-1}}.$$
The lemma would follow if we are able to show that $$H^0\left(G_\Sigma, \frac{\p^{i}\LLL^*_{\varrho \otimes \kappa^{-1}}}{\p^{i+1}\LLL^*_{\varrho \otimes \kappa^{-1}}}\right)\stackrel{?}{=}0, , \ \ \text{for } 0 \leq i \leq n-1.
$$
Now let's also fix an integer $i$ such that $0 \leq i \leq n-1$. We have the following observations:
\begin{itemize}
\item Since $\LLL^*_{\varrho \otimes \kappa^{-1}}$ is a free $\RRR[[\Gamma]]$-module, we have $\frac{\p^{i}\LLL^*_{\varrho \otimes \kappa^{-1}}}{\p^{i+1}\LLL^*_{\varrho \otimes \kappa^{-1}}} \cong \left(\frac{\p^{i}}{\p^{i+1}}\right) \otimes_{\RRR[[\Gamma]]} \LLL^*_{\varrho \otimes \kappa^{-1}}.$
\item Since $\frac{\p^{i}}{\p^{i+1}}$ is torsion-free over $\frac{\RRR[[\Gamma]]}{\p}$, we have an inclusion $\frac{\p^{i}}{\p^{i+1}} \hookrightarrow k_\p^g$, for some~integer~$g$.
\end{itemize}
As modules over $G_\Sigma$, we have the inclusion $\frac{\p^{i}\LLL^*_{\varrho \otimes \kappa^{-1}}}{\p^{i+1}\LLL^*_{\varrho \otimes \kappa^{-1}}} \hookrightarrow \left(\LLL^*_{\varrho \otimes \kappa^{-1}} \otimes k_p\right)^g$. These observations now let us conclude the lemma.
\end{proof}

\begin{lemma} \label{nakayama-lemma}
Let $\p$ be a prime ideal in $\RRR[[\Gamma]]$. The following statements~are~equivalent:
\begin{enumerate}
\item $\p$ belongs to the support of $H^0\left(G_\Sigma,\D^*_{\varrho \otimes \kappa^{-1}}\right)^\vee$.
\item $H^0(G_{\Sigma},\LLL^*_{\varrho \otimes \kappa^{-1}}\otimes k_\p) \neq 0$.
\end{enumerate}
\end{lemma}
\begin{proof}
To prove the lemma, it suffices to observe the following equivalence which one obtains using Nakayama's Lemma and Proposition 3.10 in \cite{MR2290593} :
{\small\begin{align*}
H^0(G_\Sigma, \LLL^*_{\varrho \otimes \kappa^{-1}} \otimes k_\p)=0 \iff H^0\left(G_\Sigma,\frac{\LLL^*_{\varrho \otimes \kappa^{-1}}}{\p\LLL^*_{\varrho \otimes \kappa^{-1}}}\right) = 0 \iff \p \notin \mathrm{Supp}_{\RRR[[\Gamma]]}\left(H^0\left(G_\Sigma,\D^*_{\varrho \otimes \kappa^{-1}}\right)^\vee\right).
\end{align*}
}
\end{proof}
Let $\mathcal{A}(\RRR)$ denote the algebraic closure of the fraction field of $\RRR$. For any character $\Psi:G_\Sigma \rightarrow \mathcal{A}(\RRR)^\times$, we let $\mathcal{V}_{\varrho^*(\Psi)}:= \LLL^*_{\varrho} \otimes_{\RRR} \mathcal{A}(\RRR)(\Psi)$. Using Proposition \ref{greenberg-cyc}, one can argue that the $\RRR[[\Gamma]]$-module $H^0(G_\Sigma, \D^*_{\varrho \otimes \kappa^{-1}})^\vee$ is finitely generated over $\RRR$. The generalized Cayley-Hamilton theorem then gives us a monic polynomial $h(s) \in \RRR[s]$ such that $h(\gamma_0)$ annihilates $H^0(G_\Sigma, \D^*_{\varrho \otimes \kappa^{-1}})^\vee$. Let $\p$ be a height one prime ideal of $\RRR[[\Gamma]]$ in the support of $H^0(G_\Sigma, \D^*_{\varrho \otimes \kappa^{-1}})^\vee$. Then, $\p$ must contain $h(\gamma_0)$. This implies that $\p \cap \RRR=\{0\}$ and hence that $\frac{\RRR[[\Gamma]]}{\p}$ is a finite integral extension of $\RRR$. The Galois representation obtained via the composition
\begin{align*}
G_\Sigma \xrightarrow {\left(\varrho \otimes \kappa^{-1}\right)^*} \Gl_d(\RRR[[\Gamma]]) \rightarrow \Gl_d\left(\frac{\RRR[[\Gamma]]}{\p}\right) \hookrightarrow \Gl_d(\mathcal{A}(\RRR))
\end{align*}
is given by $\varrho^*(\Psi)$, for some character $\Psi : G_\Sigma \rightarrow \mathcal{A}(\RRR)^\times$. We obtain the following implication:
$$H^0\left(G_\Sigma,\mathcal{V}_{\varrho^* (\Psi)} \right)=0 \implies H^0(G_{\Sigma},\LLL^*_{\varrho \otimes \kappa^{-1}}\otimes k_\p) =0.$$
As a result of these observations and Lemma \ref{nakayama-lemma}, we have the following proposition:

\begin{proposition} \label{surj-PN-residuefield}
Suppose  $H^0\left(G_\Sigma,\mathcal{V}_{\varrho^* (\Psi)} \right)=0$, for every character $\Psi : G_\Sigma \rightarrow \mathcal{A}(\RRR)^\times$.~Then,
\begin{enumerate}
\item $H^0\left(G_\Sigma, \D^*_{\varrho \otimes \kappa^{-1}}\right)^\vee$ is a pseudo-null $\RRR[[\Gamma]]$-module.
\item $\coker(\phi_{\varrho \otimes \kappa^{-1}})$ equals $0$.
\end{enumerate}
\end{proposition}
As an immediate corollary to the previous proposition, we have the following result:
\begin{corollary} \label{abs-irreducible}
If the Galois representation $\varrho$ is absolutely irreducible of dimension greater than one, then
\begin{enumerate}
\item $H^0\left(G_\Sigma, \D^*_{\varrho \otimes \kappa^{-1}}\right)^\vee$ is a pseudo-null $\RRR[[\Gamma]]$-module.
\item $\coker(\phi_{\varrho \otimes \kappa^{-1}})$ equals $0$.
\end{enumerate}
\end{corollary}

\subsection{No non-trivial pseudo-null submodules} The results of this section will play a key role in Section \ref{specialization-section} when we study the behavior of modules under specialization with respect to height one prime ideals. Over 2-dimensional regular local rings, we have the following useful proposition (which follows from Theorem 5.1.10 in \cite{neukirch2008cohomology}).

\begin{proposition} \label{reg-dim-2-no-pn}
A finitely generated module over a $2$-dimensional regular local ring  has no non-trivial pseudo-null submodules if and only if its projective dimension is less than or equal to $1$.
\end{proposition}
There are quite a number of useful results in \cite{MR2290593} concerning when the Pontryagin dual of Galois cohomology groups have no non-trivial pseudo-null submodules. In \cite{MR2290593}, Greenberg uses the terminology - an ``almost-divisible'' module. A discrete module is said to be ``almost-divisible'' if its Pontryagin dual has no non-trivial pseudo-null submodules. First we shall concern ourselves with the local factors away from $p$. Let $\nu \in \Sigma_0$.
\begin{itemize}
\item Proposition \ref{local-cohomology-not-p} establishes the isomorphism $H^1(\Gal{\overline{\Q}_\nu}{\Q_\nu},\D_{\varrho \otimes \kappa^{-1}})~\cong~\Loc(\nu,\varrho~\otimes~\kappa^{-1})$,
\item Proposition \ref{h2-cohomology} establishes the fact that $H^2(\Gal{\overline{\Q}_\nu}{\Q_\nu},\D_{\varrho \otimes \kappa^{-1}})$ equals zero.
\end{itemize}
Combining these observations along with Proposition 5.3 in \cite{MR2290593} gives us the next proposition.

\begin{proposition} \label{loc-No-PN}
For every prime $\nu \in \Sigma_0$, the $\RRR[[\Gamma]]$-module $\Loc(\nu,\varrho \otimes \kappa^{-1})^\vee$ has no non-trivial pseudo-null submodules .
\end{proposition}

It is easier to establish that the Pontryagin dual of the ``strict'' non-primitive Selmer group has no non-trivial pseudo-null modules. The strict non-primitive Selmer group $\mathrm{Sel}_{\varrho \otimes \kappa^{-1}}^{\Sigma_0,str}(\Q)$ is defined below.
\begin{align*}
\mathrm{Sel}_{\varrho \otimes \kappa^{-1}}^{\Sigma_0,str}(\Q) := \ker \left( H^1(G_\Sigma, \D_{\varrho\otimes \kappa^{-1}}) \xrightarrow {\phi_{\varrho\otimes \kappa^{-1}}^{\Sigma_0,str}} H^1\left(\Gal{\overline{\Q}_p}{\Q_p},\frac{\D_{\varrho \otimes \kappa^{-1}}}{\Fil^+\D_{\varrho \otimes \kappa^{-1}}}\right) \right).
\end{align*}

Note that  $H^2(\Gal{\overline{\Q}_p}{\Q_p},\Fil^+\D_{\varrho \otimes \kappa^{-1}})=0$ by Proposition \ref{h2-cohomology}. Combining Proposition 4.2.1 and Proposition 4.3.2 from \cite{greenberg2014pseudonull}, we obtain the following proposition:

\begin{proposition} \label{strict-pseudo}
Suppose the hypotheses \ref{Tor} and \ref{p-critical} hold. The $\RRR[[\Gamma]]$-module $\text{Sel}_{\varrho \otimes \kappa^{-1}}^{\Sigma_0,str}(\Q)^\vee$ has no non-trivial~pseudo-null~submodules.
\end{proposition}

Using an argument similar to the one used to prove Proposition \ref{surjectivity-greenberg}, one can show that the ``global-to-local'' map $\phi_{\varrho\otimes \kappa^{-1}}^{\Sigma_0,str}$ defining the strict non-primitive Selmer group is surjective. Proposition \ref{strict-difference}, which follows from Lemma \ref{local-cohomology-p}, allows us to evaluate the difference between the non-primitive Selmer group and the strict non-primitive Selmer group.

\begin{proposition} \label{strict-difference}

Suppose the hypotheses \ref{Tor} and \ref{p-critical} hold. We have the following short exact sequence:
\begin{align*}
0 \rightarrow \text{Sel}_{\varrho \otimes \kappa^{-1}}^{\Sigma_0,str}(\Q) \rightarrow \Sel_{\varrho \otimes \kappa^{-1}}(\Q) \rightarrow H^1\left(\Gamma_p, H^0 \left(I_p,\frac{\D_{\varrho \otimes \kappa^{-1}}}{\Fil^+\D_{\varrho \otimes \kappa^{-1}}}\right) \right)\rightarrow 0.
\end{align*}
\end{proposition}

\subsection{Regular local rings of dimension $2$}

Classical Iwasawa theory involves the study of modules over the ring $\Z_p[[s]]$. The ring $\Z_p[[s]]$ is a regular local ring of Krull dimension $2$. In the setup of Hida theory, the rings we are interested in are not always known to be regular. To obtain a workaround, we shall frequently use Lemma \ref{reg-2}.

\begin{lemma} \label{reg-2}
Let $h(s)$ be a monic polynomial in $\RRR[[s]]$ with positive degree. Let $\QQ$ be a height two prime ideal in $\RRR[[s]]$ containing $h(s)$. The 2-dimensional local ring $\RRR[[s]]_\QQ$ is regular.
\end{lemma}
\begin{proof}
The height two prime ideal $\QQ$ in $\RRR[[s]]$ corresponds to a height one prime ideal $\q'$ in $\RRR[[s]]/(h(s))$. Note that $\RRR[[s]]/(h(s))$ is an integral extension of $\RRR$. As a result, the prime ideal $\QQ \cap \RRR$ (call it $\q$) is of height one in the ring $\RRR$. Since $\RRR$ is integrally closed, the ring $\RRR_{\q}$ is a discrete valuation ring.  The Weierstrass preparation theorem tells us that for every element $y \in \RRR[[s]]$ and every natural number $m$, there exists an element $d(s) \in \RRR[[s]]$ such that the degree of $y-d(s)h(s)^m$ is a polynomial (whose degree is less than the degree of $h(s)^m$). We denote the prime ideal $\QQ \cap \RRR[s]$ in $\RRR[s]$ by $\QQ'$. The height of $\QQ'$ equals two since it contains $\q$ and $h(s)$. We have the following isomorphism:
\begin{align} \label{completions}
\varprojlim_m \frac{(\RRR[s])_{\QQ'}}{\left(\QQ'\right)^m} \cong \varprojlim_m \frac{(\RRR[[s]])_\QQ}{\QQ^m}.
\end{align}
A Noetherian local ring is regular if and only if its completion is regular (Proposition 11.24 in \cite{atiyah1969introduction}). Equation (\ref{completions}) tells us that, to prove the lemma, it suffices to show $\RRR[s]_{\QQ'}$ is regular. Observe that we have the inclusions
\begin{align} \label{inside-frac-field}
\RRR[s] \stackrel{i_1}{\hookrightarrow} \RRR_\q[s] \stackrel{i_2}{\hookrightarrow}\RRR[s]_{\QQ'}.
\end{align}
It will be convenient to view all of the rings, appearing in (\ref{inside-frac-field}), inside the fraction field of $\RRR[s]$. Let $\mathfrak{P}$ equal $i_2^{-1}(\QQ'\RRR[s]_{\QQ'})$. Note that $\QQ'$ equals $(i_2 \circ i_1)^{-1}(\QQ'\RRR[s]_{\QQ'})$. Localizing the rings appearing in (\ref{inside-frac-field}) with respect to the multiplicative set $\RRR[s] \setminus \QQ'$, we have
\begin{align} \label{inside-frac-field-localization}
\RRR[s]_{\QQ'} \stackrel{i_1}{\hookrightarrow} \left(\RRR_\q[s]\right)_{\mathfrak{P}} \stackrel{i_2}{\hookrightarrow} \RRR[s]_{\QQ'}.
\end{align}
See also Theorem 4.3 in \cite{matsumura1989commutative}. Since the composition of the two natural inclusion maps in (\ref{inside-frac-field-localization}) equals the identity map, we have $\RRR[s]_{\QQ'}$ equals $\left(\RRR_\q[s]\right)_{\mathfrak{P}}$. These observations let us conclude that the ring $\RRR[s]_{\QQ'}$ is the localization at a height two prime ideal of the polynomial ring $\RRR_{\q}[s]$ over the discrete valuation ring $\RRR_{\q}$ and hence is regular too. The lemma follows.
\end{proof}
The next proposition follows from Lemma \ref{reg-2} and an elementary  application of the generalized form of Cayley-Hamilton theorem (Theorem 4.3 in \cite{eisenbud1995commutative}). Let $\mathcal{M}$ be equal to one of the $0^{\text{th}}$ Galois cohomology groups appearing in Proposition \ref{greenberg-cyc} and let $\mathcal{N} =\mathcal{M}^\vee$. Proposition \ref{greenberg-cyc} tells us that $\mathcal{N}$ is finitely generated over $\RRR$ with an $\RRR$-linear action of $\gamma_0$. There exists a monic polynomial $h(s)$ in $\RRR[s]$ such that $h(\gamma_0)$ annihilates $\mathcal{N}$.

\begin{proposition}
Let $\mathcal{M}$ be one of the $0^{\text{th}}$ Galois cohomology groups appearing in Proposition \ref{greenberg-cyc}. Let $\mathcal{N} = \mathcal{M}^\vee$. For every height 2 prime ideal $\QQ$ in $\RRR[[\Gamma]]$, the projective dimension of $\mathcal{N}_\QQ$ as an $\RRR[[\Gamma]]_\QQ$-module is finite.
\end{proposition}

\section{Specialization of Selmer groups} \label{specialization-section}

An important topic which we would like to pursue, that may be of independent interest, is the question of studying the behavior of Selmer groups under specialization. It may first be useful to approach this topic from a general persepective. Let $\TTT$ and $\RRR$ be two integrally closed domains that are finite (and hence integral) extensions of $\Z_p[[u_1,\dotsc, u_n]]$ and $\Z_p[[v_1,\dotsc,v_{n-1}]]$ respectively. If $\QQ$ is a height two prime ideal in $\TTT$, note that by Serre's criterion for normality (Theorem 23.8 in \cite{matsumura1989commutative}), the localization $\TTT_\QQ$ is Cohen-Macaulay.\\

Consider a ring map $\varpi : \TTT \rightarrow \RRR$ such that the following conditions hold:
\begin{itemize}
\item $\ker(\varpi)$ is a height one prime ideal in $\TTT$.
\item $\RRR$ is the integral closure of $\frac{\TTT}{\ker(\varpi)}$ (note that as a result of Cohen's structure theorems, $\RRR$ is finitely generated as an $\TTT$-module).
\end{itemize}

Consider a Galois representation $\varrho : G_\Sigma \rightarrow \Gl_d(\TTT)$ with an associated Galois lattice $\LLL_\varrho$ and an associated $\Gal{\overline{\Q}_p}{\Q_p}$-equivariant filtration \ref{filtration}, so that we can associate to it a non-primitive Selmer group $\Sel_{\varrho}(\Q)$ as well. One can associate to the Galois representation $\varpi\circ\varrho:G_\Sigma \rightarrow \Gl_d(\RRR)$ the following filtration of free $\RRR$-modules that is $\Gal{\overline{\Q}_p}{\Q_p}$-equivariant:
\begin{align}
\tag{Fil-$\varpi\circ\varrho$} 0 \rightarrow \Fil^+ \LLL_\varrho \otimes_\varpi \RRR \rightarrow  \LLL_\varrho \otimes_\varpi \RRR \rightarrow \frac{\LLL_\varrho \otimes_\varpi \RRR}{\Fil^+\LLL_\varrho \otimes_\varpi \RRR} \rightarrow 0.
\end{align}
In a natural way, one can then associate a non-primitive Selmer group $\Sel_{\varpi \circ \rho}(\Q)$ to the Galois representation $\varpi \circ \rho$ as well. Given the element $\Div\left(\Sel_{\varrho}(\Q)^\vee\right)$ in the divisor group of $\TTT$, the problem we are interested in is finding the element $\Div\left(\Sel_{\varpi \circ \rho}(\Q)^\vee\right)$ in the divisor group of $\RRR$. We shall proceed in two steps. First, we will relate the element $\Div\left(\Sel_{\varrho}(\Q)^\vee\right)$ in the divisor group of $\TTT$ to the element $\Div\left(\Sel_{\varrho}(\Q)^\vee\otimes_{\varpi}\RRR\right)$ in the divisor group of $\RRR$. The key results for this purpose are Proposition \ref{specialization} and Proposition \ref{specialization-strict}. Second, we will prove a control theorem relating $\Sel_{\varrho}(\Q)^\vee \otimes_\varpi \RRR$ to $\Sel_{\varpi \circ \varrho}(\Q)^\vee$. The key result for this purpose is Proposition \ref{control-theorem-selmer-groups}. The results of Section \ref{specialization-section} and Section \ref{control-theorems-section} do not rely on each other.

\subsection{A commutative algebra perspective}

For a finitely generated $\TTT$-module $\N$, we shall frequently need to consider the following hypotheses in this section:

\begin{enumerate}[leftmargin=2.5cm, style=sameline, align=left, label=\textsc{No-PN}, ref=\textsc{No-PN}]
\item\label{No-PN} For every height two prime ideal $\QQ$ in $\TTT$ containing $\ker(\varpi)$, the $\TTT_\QQ$-module $\N_\QQ$ has no non-zero pseudo-null submodules.
\end{enumerate}

\begin{enumerate}[leftmargin=2.5cm, style=sameline, align=left, label=\textsc{Fin-Proj}, ref=\textsc{Fin-Proj}]
\item\label{Fin-Proj} For every height two prime ideal $\QQ$ in $\TTT$ containing $\ker(\varpi)$, the $\TTT_\QQ$-module $\N_\QQ$ has finite projective dimension.
\end{enumerate}

If $\p$ is a height one prime ideal in $\TTT$, then the $\TTT$-module $\frac{\TTT}{\p}$ satisfies \ref{No-PN}. If $a$ is a non-zero element of $\TTT$,  then the $\TTT$-module $\frac{\TTT}{(a)}$ also satisfies \ref{No-PN}. When $\TTT$ is a regular local ring, the hypothesis \ref{Fin-Proj} is automatically satisfied. Before discussing the main proposition, it will be illustrative to discuss an example that highlights why the hypothesis \ref{Fin-Proj} is important in the case when $\TTT$ is not regular.

\begin{example} \label{first-non-regular-example}
Let $\TTT=\Z_p[[u]][\vartheta]$, where $\vartheta^2 = u(u-p)$. This ring is a quadratic extension of the UFD $\Z_p[[u]]$. The height one prime ideals $\p_1 = (\vartheta,u)$ and $\p_2=(\vartheta,u-p)$ in $\TTT$ lying above the height one prime ideals $(u)$ and $(u-p)$  in $\Z_p[[u]]$ are not principal. We keep the picture below in mind.
\begin{center}
\begin{tikzpicture}[node distance = 1cm, auto]
      \node (Lambda) {$\Z_p[[u]]$};
      \node (R) [above of=Lambda] {$\RRR$};

\node (u) [right of=Lambda, node distance = 4 cm] {$(u)$};
\node (p1) [above of = u] {$\p_1=(\vartheta,u)$};
\node (u-p) [right of=u, node distance = 4 cm] {$(u-p)$};
\node (p2) [above of = u-p] {$\p_2=(\vartheta,u-p)$};
\draw[-] (R) to node  {} (Lambda);
\draw[-] (u) to node  {} (p1);
\draw[-] (u-p) to node  {} (p2);
\end{tikzpicture}
\end{center}

Consequently, $\TTT$ is not a UFD, and hence not a regular local ring too. We have the following equality in the divisor group of $\TTT$:
\begin{align}\label{eg-equality-start}
 \Div(\vartheta) = \Div\left(\frac{\TTT}{\p_1} \oplus \frac{\TTT}{\p_2}\right).
\end{align}
Let $\RRR = \Z_p$. For simplicity, suppose $2$ is a square in $\Z_p^\times$ so that $\sqrt{2} \in \Z_p$. Consider the map $\varpi : \TTT \rightarrow \RRR$ defined below.
\begin{align*}  \varpi : \underbrace{\Z_p[[u]][\vartheta]}_{\TTT} \rightarrow & \underbrace{\Z_p}_{\RRR}, \\
 \varpi(u) = 2p, \qquad &  \varpi(\vartheta)= \sqrt{2}p.
\end{align*}
In the divisor group of $\RRR$, we have:
\begin{align}\label{eg-equality-end}
\Div(p) =  \Div(\varpi(\vartheta)) \neq \Div\left(\left( \frac{\TTT}{\p_1} \oplus \frac{\TTT}{\p_2} \right) \otimes_\varpi \RRR \right) =\Div(p^2)  .
\end{align}
Note that the $\TTT$-module $\frac{\TTT}{\p_1} \oplus \frac{\TTT}{\p_2}$ satisfies \ref{No-PN}. However, its projective dimension over $\TTT$ is infinite. While studying specialization of Selmer groups, keeping equations (\ref{eg-equality-start}) and (\ref{eg-equality-end}) in mind, we would like to avoid a situation when the module $\frac{\TTT}{\p_1} \oplus \frac{\TTT}{\p_2}$ plays the role of a Selmer group and the element $\vartheta$ plays the role of a $p$-adic $L$-function.
\end{example}

\begin{proposition} \label{specialization}
Suppose $\mathcal{Y}_1$, $\mathcal{Y}_2$ and $\mathcal{M}$ are three finitely generated torsion $\TTT$-modules  satisfying the following conditions:
\begin{enumerate}[style=sameline, align=left, label=\textsc{Hyp}-\arabic* , ref=\textsc{Hyp}-\arabic*]
\item\label{Hyp-1} The height one prime ideal $\ker(\varpi)$ does not belong to the support of $\mathcal{Y}_1$, $\mathcal{Y}_2$ or $\M$.
\item\label{Hyp-2} The $\TTT$-modules $\mathcal{Y}_1$ and $\mathcal{Y}_2$ satisfy the hypotheses \ref{No-PN} and \ref{Fin-Proj}.
\item\label{Hyp-3}  If $\QQ$ is a height two prime ideal containing the height one prime ideal $\ker(\varpi)$ and contained in the support of $\M$, then
\begin{enumerate}
\item The $2$-dimensional local ring $\TTT_\QQ$ is regular,
\item The $1$-dimensional local domain $\left(\frac{\TTT}{\ker(\varpi)}\right)_\QQ$ is integrally closed.
\end{enumerate}
\end{enumerate}
Fix the symbol $\lhd \rhd$ to denote either ``$\geq$'' or ``$\leq$''. Suppose further that we have the following equality in the divisor group of $\TTT$:
\begin{align*}
\Div\left( \mathcal{Y}_1 \right)  \ \lhd \rhd  \ \Div\left(\mathcal{Y}_2\right) - \Div\left(\mathcal{M}\right).
\end{align*}
Under these assumptions, we obtain the following equality in the divisor group of $\rrr$:
\begin{align*}
\Div\left(\mathcal{Y}_1 \otimes_\varpi \RRR \right)   \ \lhd \rhd \ \Div\left(\mathcal{Y}_2 \otimes_\varpi \RRR \right) + \Div \left( \Tor_{1}^{\TTT}\left( \RRR , \mathcal{M} \right) \right) - \Div\left(\mathcal{M}\otimes_\varpi \RRR \right).
\end{align*}
\end{proposition}

\begin{Remark} \label{equality-specialization-div}
As a consequence of Proposition \ref{specialization}, if $\Div\left( \mathcal{Y}_1 \right) +  \Div\left(\mathcal{M}\right) = \Div\left(\mathcal{Y}_2\right)$ in the divisor group of $\TTT$, we obtain the following equality in the divisor group of $\RRR$:
\begin{align*}
\Div\left(\mathcal{Y}_1 \otimes_\varpi \RRR \right)   = \Div\left(\mathcal{Y}_2 \otimes_\varpi \RRR \right) + \Div \left( \Tor_{1}^{\TTT}\left( \rrr , \mathcal{M} \right) \right)- \Div\left(\mathcal{M}\otimes_\varpi \RRR \right).
\end{align*}
\end{Remark}

\begin{Remark} \label{dvr-localization}
 Whenever $\left(\frac{\TTT}{\ker(\varpi)}\right)_\QQ$ is integrally closed, it is a DVR; and there is exactly one prime ideal in $\RRR$ lying above the height one prime ideal in $\frac{\TTT}{\ker(\varpi)}$ that $\QQ$ corresponds to.  Hence, $\RRR_\QQ$ equals $\left(\frac{\TTT}{\ker(\varpi)}\right)_\QQ$. Whenever the map $\varpi :\TTT \rightarrow \RRR$ is surjective, we have $\RRR\cong \frac{\TTT}{\ker(\varpi)}$; so in this case $\left(\frac{\TTT}{\ker(\varpi)}\right)_\QQ$ is automatically an integrally closed domain.
\end{Remark}

\begin{proof}
Fix the symbol $\lhd \rhd$ to denote ``$\geq$'' throughout the proof. The proof proceeds similarly when the symbol $\lhd \rhd$ denotes $\leq$. Let us fix a height one prime ideal $\q$ in $\RRR$. Let $\q'=\q \cap \frac{\TTT}{\ker(\varpi)}$. Let $\QQ$ be the height two prime ideal, containing $\ker(\varpi)$, that  corresponds to $\q'$.  We shall use the notion of lengths, denoted by $\len$, to describe divisors (see \cite{atiyah1969introduction}). In the divisor group of $\TTT$, we have
\begin{align*}
& \Div(\mathcal{Y}_1) = \sum_{ \substack{\p \text{ height one} \\ \text{ prime in $\TTT$}}} \len_{\TTT_\p}(\mathcal{Y}_1)_\p \cdot \p,  \qquad  \Div(\mathcal{Y}_2) = \sum_{ \substack{\p \text{ height one} \\ \text{ prime in $\TTT$}}} \len_{\TTT_\p}(\mathcal{Y}_2)_\p \cdot \p, \  \\ & \Div(\M) = \sum_{ \substack{\p \text{ height one} \\ \text{ prime in $\TTT$}}} \len_{\TTT_\p}\left(\M_\p\right) \cdot \p.
\end{align*}

The hypotheses stated in the proposition tell us that, for all height one prime ideals $\p$~in~$\TTT$,
\begin{align}\label{inequality-divisor-notstrict}\len_{\TTT_\p}(\mathcal{Y}_1)_\p \geq \len_{\TTT_\p}(\mathcal{Y}_2)_\p - \len_{\TTT_\p}\left( \M_\p\right).
\end{align}
\ref{Hyp-1} tells us that all the $\RRR_\q$-modules appearing in (\ref{length-needo-to-prove}) below are torsion. To prove the proposition, we will need to show that the following inequality holds:
\begin{align}\label{length-needo-to-prove}
\len_{\RRR_\q} \left(\Y_1 \otimes_\varpi \RRR_\q\right)    \stackrel{?}{\geq} \len_{\RRR_\q} \left( \Y_2 \otimes_\varpi \RRR_\q\right)  + \len_{\RRR_\q} \left( \Tor_1^\TTT(\RRR,\M)_\q\right) - \len_{\RRR_\q} \left(\M \otimes_\varpi \RRR_\q\right).
\end{align}

We shall first analyze $\Y_1$ and $\Y_2$. Let $\QQ_0$ equal $\Z_p[[u_1,\dotsc,u_n]] \cap \QQ$. The extension $\Z_p[[u_1,\dotsc,u_n]]_{\QQ_0} \hookrightarrow \TTT_\QQ$ is integral. Using this observation and \ref{Hyp-2}, we can conclude that the $\Z_p[[u_1,\dotsc,u_n]]_{\QQ_0}$-modules $\left(\Y_1\right)_\QQ$ and $\left(\Y_2\right)_\QQ$ have no pseudo-null submodules too. By Proposition \ref{reg-dim-2-no-pn}, we have
$$ \pd_{\Z_p[[u_1,\dotsc,u_n]]_{\QQ_0}} \left( \Y_1\right)_\QQ \leq 1, \qquad  \pd_{\Z_p[[u_1,\dotsc,u_n]]_{\QQ_0}} \left( \Y_2\right)_\QQ \leq 1. $$
The depth of an $\TTT_\QQ$-module over $\TTT_\QQ$ equals its depth over the ring $\Z_p[[u_1,\dotsc,u_n]]_{\QQ_0}$  (see the Appendix in \cite{MR2919145}). Using the Auslander-Buchsbaum formula over the 2-dimensional regular local ring $\Z_p[[u_1,\dotsc,u_n]]_{\QQ_0}$, we obtain
$$ \depth_{\Z_p[[u_1,\dotsc,u_n]]_{\QQ_0}} \left( \Y_1\right)_\QQ = \depth_{\TTT_\QQ} \left( \Y_1\right)_\QQ \geq 1, \quad  \depth_{\Z_p[[u_1,\dotsc,u_n]]_{\QQ_0}} \left( \Y_2\right)_\QQ = \depth_{\TTT_\QQ} \left( \Y_2\right)_\QQ \geq 1. $$

For the $2$-dimensional local ring $\TTT_\QQ$, by Theorem 17.2 in \cite{matsumura1989commutative}, we have $\depth_{\TTT_\QQ}\TTT_\QQ \leq 2$ (in fact, we have equality since $\TTT_\QQ$ is integrally closed and is hence Cohen-Macaulay, by Serre's criterion for normality; see Theorem 23.8 in \cite{matsumura1989commutative}). Using Auslander-Buchsbaum formula over the 2-dimensional local ring $\TTT_\QQ$, we have
\begin{align*}
&\pd_{\TTT_\QQ} \left( \Y_1\right)_\QQ + \depth_{\TTT_\QQ} \left( \Y_1\right)_\QQ \leq 2, \qquad & \pd_{\TTT_\QQ} \left( \Y_2\right)_\QQ + \depth_{\TTT_\QQ} \left( \Y_2\right)_\QQ \leq 2.\\
& \implies \pd_{\TTT_\QQ} \left( \Y_1\right)_\QQ \leq 1, \qquad & \implies \pd_{\TTT_\QQ} \left( \Y_2\right)_\QQ \leq 1.
\end{align*}
We write down projective resolutions for $(\Y_1)_\QQ$ and $(\Y_2)_\QQ$ respectively.
\begin{align} \label{pd-resolutions}
0 \rightarrow \TTT_\QQ^{n_1} \xrightarrow {\omega_1} \TTT_\QQ^{n_1} \rightarrow (\Y_1)_\QQ \rightarrow 0, \qquad  0 \rightarrow \TTT_\QQ^{n_2} \xrightarrow {\omega_2} \TTT_\QQ^{n_2} \rightarrow (\Y_2)_\QQ \rightarrow 0.
\end{align}
Here, $\omega_1$ and $\omega_2$ are $n_1 \times n_1$ and $n_2 \times n_2$ matrices respectively over $\TTT_\QQ$. The elements $\det(\omega_1)$ and $\det(\omega_2)$ do not belong to $\ker(\varpi)$. So, $\varpi(\det(\omega_1))$ and $\varpi(\det(\omega_2))$  do not equal zero. Tensoring the sequences in (\ref{pd-resolutions}) with $\RRR_\q$ (over $\TTT_\QQ$), we obtain the following exact sequences:
\begin{align} \label{pd-resolutions-r}
0 \rightarrow \RRR^{n_1}_\q \xrightarrow {\varpi(\omega_1)} \RRR^{n_1}_\q \rightarrow (\Y_1 \otimes \RRR)_q \rightarrow 0, \qquad 0 \rightarrow \RRR^{n_2}_\q \xrightarrow {\varpi(\omega_2)} \RRR^{n_2}_\q \rightarrow (\Y_2 \otimes \RRR)_q \rightarrow 0,
\end{align}
Using (\ref{pd-resolutions-r}) and the fact that $\RRR_\q$ is a DVR, we have
\begin{align} \label{lengths-y1-y2}
\len_{\RRR_\q}\left(\Y_1 \otimes_\varpi \RRR_\q \right) = \len_{\RRR_\q}\left(\frac{\RRR_\q}{\varpi(\det(\omega_1)}\right), \qquad \len_{\RRR_\q}\left(\Y_2 \otimes_\varpi\RRR_\q \right) = \len_{\RRR_\q}\left(\frac{\RRR_\q}{\varpi(\det(\omega_2)}\right).
\end{align}

There are two cases we need to consider. First we shall suppose that the height one prime ideal $\q$ in $\RRR$ does not belong to the support of $\M \otimes_\varpi \RRR$ (and hence the height two prime ideal $\QQ$ in $\TTT$ does not belong to the support of $\M$). In this case, for all height one prime ideals $\p$ inside $\QQ$, all of the modules displayed below in equation (\ref{all-zero-not-support}) equal zero.
\begin{align} \label{all-zero-not-support}
& \M \otimes \RRR_q =  \left(\frac{\M}{\ker(\varpi)\M}\right) _{\q'} = \M_\QQ  =  \Tor_1^\TTT(\RRR,\M)_\QQ = \Tor_1^\TTT(\RRR,\M)_\q = 0. \\ \notag
\implies &\len_{\TTT_\p}(\M_\p)= \len_{\RRR_\q}\left(\Tor_1^\TTT(\RRR,\M)_\q\right) = \len_{\RRR_\q}\left( \M \otimes_\varpi \RRR_\q \right) =0.
\end{align}
Since $\TTT_\QQ$ is integrally closed, the ring $\TTT_\QQ$ equals $\bigcap \TTT_\p$ (viewed as subsets of the fraction field of $\TTT$), where the intersection is taken over all the height one prime ideals $\p$ contained in the height two prime ideal $\QQ$. As a result, if an element, belonging to the fraction field of $\TTT$, lies in $\TTT_\p$ for all height one prime ideals $\p$ contained in $\QQ$, then this element lies in $\TTT_\QQ$. Using (\ref{inequality-divisor-notstrict}), we have
\begin{align*}
&\len_{\TTT_\p}\left(\Y_1\right)_\p \geq \len_{\TTT_\p}\left(\Y_2\right)_\p, \text{ for all height one prime ideals $\p$ in $\QQ$.} \\
\implies & \det(\omega_1) = z \det(\omega_2), \text{ for some non-zero element $z$ in $\TTT_\QQ$.} \\
\implies & \varpi(\det(\omega_1)) = \varpi(z) \cdot \varpi(\det(\omega_2)), \text{ and $\varpi(z) \neq 0$}. \\
\implies & \len_{\RRR_\q}\left(\Y_1\otimes_\varpi \RRR_\q\right) \geq \len_{\RRR_\q}\left(\Y_2\otimes_\varpi \RRR_\q\right).
\end{align*}

Thus, equation (\ref{length-needo-to-prove}) holds in this case. For the second case, we shall suppose that $\q$ belongs to the support of $\M \otimes_{\varpi} \RRR$ (and hence the height two prime ideal $\QQ$ in $\TTT$ does belong to the support of $\M$). \ref{Hyp-3} tells us that $\TTT_\QQ$ is a regular local ring. So, the height one prime ideal generated by $\ker(\varpi)$ in $\TTT_\QQ$ is principal, say equal to $(\zeta)$. Using the structure theorem for regular local rings of Krull dimension $2$ (Theorem 5.1.10 in \cite{neukirch2008cohomology}), we obtain a short exact sequence $0 \rightarrow \oplus \frac{\TTT_\QQ}{(\epsilon_i^{a_i})} \rightarrow \M_\QQ \rightarrow \mathcal{Z} \rightarrow  0 $. Here, the $\TTT_\QQ$-module $\mathcal{Z}$ is pseudo-null and the elements $\epsilon_i$ generate height one prime ideals in the regular local ring $\TTT_\QQ$. The non-negative integers $a_i$ are equal to zero for all but finitely many $i$'s. For all height one prime ideals $\p$ in $\TTT_\QQ$,
\begin{align} \label{lengths-second-case}
& \len_{\TTT_\p}(\mathcal{Y}_1)_\p+ \len_{\TTT_\p}(\M)_\p \geq \len_{\TTT_\p}(\mathcal{Y}_2)_\p. \\
\notag \implies & \len_{\TTT_\p}\left(\frac{\TTT_\p}{(\det(\omega_1) \cdot \prod \epsilon_i^{a_i})}\right)   \geq \len_{\TTT_\p}\left(\frac{\TTT_\p}{(\det(\omega_2))}\right)\\
\notag  \implies & \det(\omega_1)\prod \epsilon_i^{a_i}= z \det(\omega_2), \text{ for some element $z$ in $\TTT_\QQ$} ,
\\ \notag \implies & \varpi(\det(\omega_1))\prod \varpi(\epsilon_i)^{a_i}= \varpi(z) \varpi(\det(\omega_2)), \text{ and $\varpi(z) \neq 0$}. \\ \notag
\implies & \len_{\RRR_\q} \left(\frac{\RRR_\q}{(\varpi(\det(\omega_1)))}\right) + \len_{\RRR_\q} \left(\frac{\RRR_\q}{\left(\prod\varpi(\epsilon_i)^{a_i}\right)}\right) \geq \len_{\RRR_\q} \left(\frac{\RRR_\q}{(\varpi(\det(\omega_2)))}\right).
\end{align}
In this second case, by Remark \ref{dvr-localization}, we  have $\RRR_q \cong \RRR_\QQ \cong \frac{\TTT_\QQ}{(\zeta)}$. As a result,
\begin{align*}
\Tor_1^\TTT(\RRR,\M)_\q \cong \Tor_1^\TTT(\RRR,\M)_\QQ \cong \Tor_1^{\TTT_\QQ}\left(\frac{\TTT_\QQ}{(\zeta)},\M_\QQ\right) \cong \M_\QQ[\zeta]
\end{align*}
We also have the following commutative diagram:
{\small
\begin{align*}
\xymatrix{0 \ar[r]& \oplus \frac{\TTT_\QQ}{(\epsilon_i^{a_i})} \ar[d]_{\zeta}\ar[r]& \M_\QQ \ar[d]_{\zeta}\ar[r]& \mathcal{Z} \ar[d]_{\zeta}\ar[r]& 0 \\
0 \ar[r]& \oplus \frac{\TTT_\QQ}{(\epsilon_i^{a_i})} \ar[r]& \M_\QQ \ar[r]& \mathcal{Z} \ar[r]& 0
 }
\end{align*}
}
Using the hypothesis \ref{Hyp-1}, one can conclude that none of the elements $\epsilon_i$ generate $\ker(\varpi)$. So, $\zeta \neq \epsilon_i$.  The snake lemma then gives us the following exact sequence:
\begin{align}\label{ses-intermediate}
0 \rightarrow \underbrace{\M_\QQ [\zeta]}_{\Tor_1^\TTT(\RRR,\M)_\QQ} \rightarrow \mathcal{Z}[\zeta] \rightarrow \oplus  \frac{\RRR_\QQ}{\varpi(\epsilon_i)^{a_i}} \rightarrow \underbrace{\frac{\M_\QQ}{\zeta \cdot \M_\QQ}}_{\M \otimes \RRR_\QQ} \rightarrow \frac{\mathcal{Z}}{\zeta\cdot \mathcal{Z}} \rightarrow 0.
\end{align}
Length is additive in exact sequences. To show that (\ref{length-needo-to-prove}) holds in this second case, we will use equations (\ref{lengths-second-case}) and (\ref{ses-intermediate}). It only remains to show that $\len_{\RRR_\q} \left(\mathcal{Z}[\zeta]\right) = \len_{\RRR_\q} \left(\frac{\mathcal{Z}}{\zeta \cdot \mathcal{Z}}\right)$. Since the local ring $\TTT_\QQ$ is $2$-dimensional and since the $\TTT_\QQ$-module $\mathcal{Z}$ is pseudo-null, $\len_{\TTT_\QQ} \left(\mathcal{Z} \right) <\infty$. Also since the map $\TTT_\QQ \rightarrow \RRR_\q$ is surjective in this case, we have the following equalities:
$$\len_{\RRR_\q} \left(\mathcal{Z}[\zeta]\right) = \len_{\TTT_\QQ} \left(\mathcal{Z}[\zeta]\right), \qquad \len_{\RRR_\q} \left(\frac{\mathcal{Z}}{\zeta \cdot \mathcal{Z}}\right) = \len_{\TTT_\QQ} \left(\frac{\mathcal{Z}}{\zeta \cdot \mathcal{Z}}\right).  $$
The proposition now follows from these observations because we have the exact sequence $0 \rightarrow \mathcal{Z}[\zeta] \rightarrow \mathcal{Z} \xrightarrow {\zeta} \mathcal{Z} \rightarrow \frac{\mathcal{Z}}{\zeta \cdot \mathcal{Z}}\rightarrow 0$ of $\TTT_\QQ$-modules.
\end{proof}

It will be interesting to consider,  whenever  $\Div\left( \mathcal{Y}_1 \right)  > \Div\left(\mathcal{Y}_2\right)- \Div\left(\mathcal{M}\right)$ in the divisor group of $\TTT$, whether we have the following inequality in the divisor group of $\RRR$:
\begin{align*}
\Div\left(\mathcal{Y}_1 \otimes_\varpi \RRR\right)   \stackrel{?}{>} \Div\left(\mathcal{Y}_2 \otimes_\varpi \RRR\right) + \Div \left( \Tor_{1}^{\TTT}\left( \RRR , \mathcal{M} \right) \right) - \Div\left(\mathcal{M}\otimes_\varpi \RRR\right).
\end{align*}
We would require this stronger result for Theorem \ref{specialization-result}. One can modify the proof of Proposition \ref{specialization} to obtain this stronger result if the ideal generated by the height one prime ideals $\ker(\varpi)$ and $\p$ is of height $2$ inside $\TTT$, for all height one prime ideals $\p$ in the support of $\Y_1$, $\Y_2$ or $\M$. In a UFD, the sum of two height $1$ prime ideals generates an ideal of height $2$. However, as Example \ref{second-non-regular-example} shows, this is not generally true. It will be helpful to discuss this example that illustrates the pathologies which we wish to avoid.

\begin{example}\label{second-non-regular-example}
Let $\TTT = \Z_p[[u]][\vartheta]$ be the ring discussed in Example \ref{first-non-regular-example}. Recall that $\TTT$ is not a UFD and that $\vartheta^2=u(u-p)$. Let us identify the completed tensor product $\TTT \hotimes \TTT$ with $\Z_p[[u_1,u_2]][\vartheta_1,\vartheta_2]$ such that $\vartheta_1^2 = u_1(u_1-p)$ and $\vartheta_2^2=u_2(u_2-p)$. The ring $\TTT \hotimes \TTT$ has Krull dimension $3$ and is not a UFD. The maximal ideal of $\TTT \hotimes \TTT$ equals $(\vartheta_1,\vartheta_2,u_1,u_2,p)$. We keep the following picture in mind that describes how the height one prime ideals $(u_1-u_2)$ and $(u_1+u_2-p)$ of $\Z_p[[u_1,u_2]]$ split in $\TTT \hotimes \TTT$:
\begin{center}
\begin{tikzpicture}[node distance = 1.5cm, auto]
      \node (Lambda) {$\Z_p[[u_1,u_2]]$};
      \node (R) [above of=Lambda] {$\TTT \hotimes \TTT$};

\node (u) [right of=Lambda, node distance = 4 cm] {$(u_1-u_2)$};
\node (p1) [above of = u, left of =u] {$(\vartheta_1 -\vartheta_2,u_1-u_2)$};
\node (p2) [right of =p1, node distance = 3.5cm] {$(\vartheta_1 +\vartheta_2,u_1-u_2)$};

\node (u-p) [right of=u, node distance = 7.5 cm] {$(u_1+u_2-p)$};
\node (q1) [above of = u-p, left of =u-p] {$(\vartheta_1 -\vartheta_2,u_1+u_2-p)$};
\node (q2) [right of =q1, node distance = 4.5cm] {$(\vartheta_1 +\vartheta_2,u_1+u_2-p)$};
\draw[-] (R) to node  {} (Lambda);
\draw[-] (u) to node  {} (p1);
\draw[-] (u) to node  {} (p2);
\draw[-] (u-p) to node  {} (q1);
\draw[-] (u-p) to node  {} (q2);
\end{tikzpicture}
\end{center}
Let us label these height one prime ideals in $\TTT \hotimes \TTT$.
\begin{align*}
& \p_1 := (\vartheta_1 -\vartheta_2,u_1-u_2), \quad && \p_2 :=(\vartheta_1 +\vartheta_2,u_1-u_2), \\ & \p_3 :=(\vartheta_1 -\vartheta_2,u_1+u_2-p), \quad && \p_4 :=(\vartheta_1 +\vartheta_2,u_1+u_2-p).
\end{align*}
We have the following inequality in the divisor group of $\TTT \hotimes \TTT$:
\begin{align}\label{inequality-second-eg}
 \Div(\vartheta_1 + \vartheta_2) = 1 \cdot \p_2 + 1 \cdot \p_4 >  \Div \left(\frac{\TTT \hotimes \TTT}{\p_2} \right) = 1 \cdot \p_2 .
\end{align}

Consider the map $\varpi : \TTT \hotimes \TTT \rightarrow \TTT $ defined below.
\begin{align*}
\varpi : \TTT \hotimes \TTT  & \rightarrow \TTT, \\
\varpi(\vartheta_1) = \vartheta, \ \varpi(\vartheta_2) = \vartheta, \ & \varpi(u_1) = u, \ \varpi(u_2) = u.
\end{align*}

The kernel of $\varpi$ is $\p_1$. Let us make a few observations. First note that the $\TTT \hotimes \TTT$-module $\frac{\TTT \hotimes \TTT}{\p_2}$ satisfies \ref{No-PN}. Now, we claim that the $\TTT \hotimes \TTT$-module $\frac{\TTT \hotimes \TTT}{\p_2}$ satisfies \ref{Fin-Proj}. To see this, consider the ideal $\p_1 + \p_2$ in $\TTT \hotimes \TTT$ which equals the ideal $(\vartheta_1, \vartheta_2, u_1 - u_2)$. The minimal primes containing the ideal $\p_1 + \p_2$ are $(\vartheta_1,\vartheta_2,u_1,u_2)$ and $(\vartheta_1,\vartheta_2,u_1-p,u_2-p)$, both of which have height $2$. Thus, the height of the ideal $\p_1 + \p_2$ equals $2$ in $\TTT \hotimes \TTT$. We have the following equality:
$$(\vartheta_1 - \vartheta_2)(\vartheta_1 + \vartheta_2) = (u_1-u_2)(u_1 + u_2 - p).$$
Let $\QQ$ be a height two prime ideal containing $\p_1 + \p_2$. In particular, as we observed earlier, it equals $(\vartheta_1,\vartheta_2,u_1,u_2)$ or $(\vartheta_1,\vartheta_2,u_1-p,u_2-p)$. The element $u_1 + u_2 - p$ does not belong to $\QQ$ (otherwise $\QQ$ would have height $3$). So, $u_1 + u_2 -p$ is a unit in $(\TTT \hotimes \TTT)_\QQ$. In $(\TTT \hotimes \TTT)_\QQ$, the height one prime ideal generated by $\p_2$ must thus be principal (in fact, it can be generated by the element $\vartheta_1 + \vartheta_2$). Our observations indicate that the projective dimension of the $(\TTT \hotimes \TTT)_\QQ$-module  $\left(\frac{\TTT \hotimes \TTT}{\p_2}\right)_\QQ$  equals $1$. \\
Our observations indicate that, though we started with a strict inequality in equation (\ref{inequality-second-eg}), we obtain the following equality in the divisor group of $\TTT$:
\begin{align*}
\Div\big(\varpi(\vartheta_1+\vartheta_2)\big) = \Div(\vartheta) = \Div\bigg(\frac{\TTT \hotimes \TTT}{\p_2} \otimes_\varpi \TTT \bigg).
\end{align*}
Now note that the minimal prime containing the ideal $\p_1 + \p_4$ in $\TTT \hotimes \TTT$ equals the maximal ideal $(\vartheta_1,\vartheta_2,u_1,u_2,p)$ in $\TTT \hotimes \TTT$. This implies that the height of $\p_1 + \p_4$ equals $3$~in~$\TTT \hotimes \TTT$. This can only happen if the height one prime ideals in the set $\{\p_1,\p_4\}$ each generate an element of infinite order in the divisor class group of $\TTT \hotimes \TTT$. Similarly, the height one prime ideals in the set $\{\p_2,\p_3\}$ also each generate an element of infinite order in the divisor class group of $\TTT \hotimes \TTT$. This is precisely the pathology that we wish to avoid.
\end{example}

\begin{proposition} \label{specialization-strict}
In addition to the hypotheses \ref{Hyp-1}, \ref{Hyp-2} and \ref{Hyp-3} given in Proposition \ref{specialization}, suppose also that both $\Div(\Y_1)$ and $\Div(\Y_2)-\Div(\M)$ generate torsion elements in the divisor group of $\TTT$.  Fix the symbol $\lhd \rhd$ to denote either ``$>$'' or ``$<$''.

Suppose we have the following equality of divisors in $\TTT$:
\begin{align*}
\Div(\Y_1) \lhd \rhd \Div(\Y_2) - \Div\left(\mathcal{M}\right).
\end{align*}
Then, we have the following equality of divisors in $\RRR$:
\begin{align*}
\Div\left(\mathcal{Y}_1 \otimes_\varpi \RRR\right)   \lhd \rhd \Div\left(\mathcal{Y}_2 \otimes_\varpi \RRR\right) + \Div \left( \Tor_{1}^{\TTT}\left( \RRR , \mathcal{M} \right) \right) - \Div\left(\mathcal{M}\otimes_\varpi \RRR\right).
\end{align*}
\end{proposition}

\begin{proof}
Let the symbol $\lhd \rhd$ denote $>$. The proof proceeds similarly when the symbol $\lhd \rhd$ denotes $<$. There exist a positive integer $n$ and two non-zero elements $\vartheta_1$, $\vartheta_2$ in $\TTT$ such that we have the following equality in the divisor group of $\TTT$:
\begin{align*}
n \cdot \Div(\Y_1) =\Div(\vartheta_1), \qquad n \cdot \left(\Div(\Y_2) -\Div(\M) \right) = \Div(\vartheta_2).
\end{align*}
Let $\mathfrak{Y}_1 = \oplus \Y_1$,  $\mathfrak{Y}_2 = \oplus \Y_2$, and $\mathfrak{M} = \oplus \M$ denote direct sums of $n$ copies of $\Y_1$, $\Y_2$ and $\M$ respectively. We have the following equality in the divisor group of $\TTT$:
\begin{align} \label{strict-inequality-divisors-prop}
 \Div(\mathfrak{Y}_1) =\Div(\vartheta_1), \qquad  \Div(\mathfrak{Y}_2) -\Div(\mathfrak{M})  = \Div(\vartheta_2).
\end{align}
Note that $\TTT$-modules $\frac{\TTT}{(\vartheta_1)}$ and  $\frac{\TTT}{(\vartheta_2)}$ have no non-zero pseudo-null submodules and have finite projective dimension. Proposition \ref{specialization}, applied to the equality of divisors appearing in (\ref{strict-inequality-divisors-prop}), gives us the following equality in the divisor group of $\RRR$:
\begin{align*}
\Div(\mathfrak{Y}_1 \otimes \RRR) = \Div(\varpi(\vartheta_1)), \quad \Div(\mathfrak{Y}_2 \otimes \RRR) + \Div \left( \Tor_{1}^{\TTT}\left( \RRR , \mathfrak{M} \right) \right) - \Div(\mathfrak{M} \otimes \RRR) = \Div(\varpi(\vartheta_2))
\end{align*}
Since $\Div(\vartheta_1) > \Div(\vartheta_2)$ and $\TTT$ is integrally closed, there exists an element $z \in \TTT$, not a unit, such that $\vartheta_1 = z \cdot \vartheta_2$. Since $\varpi:\TTT \rightarrow \RRR$ is a map of local rings, we get that $\varpi(z)$ is not a unit in $\RRR$. Combining these observations, we have $\varpi(\vartheta_1) = \varpi(z) \cdot \varpi(\vartheta_2)$ and that $\varpi(z)$ is not a unit in $\RRR$. Hence, we obtain the following inequality in the divisor group of $\RRR$:
\begin{align*}
\Div(\varpi(\vartheta_1))   &>  \Div(\varpi(\vartheta_2)) \\ \implies
n \cdot \Div(\Y_1 \otimes \RRR) &> n \cdot \bigg(\Div(\Y_2 \otimes \RRR) + \Div \left( \Tor_{1}^{\TTT}\left( \RRR , \mathcal{M} \right) \right) - \Div(\M \otimes \RRR)  \bigg).
\end{align*}
This completes the proof of the Proposition since $n$ is positive.
\end{proof}

\subsection{Projective Dimensions}

Let $\G$ be a profinite group. Let $\M$ be a finitely generated free $\RRR$-module on which there is a $\TTT$-linear continuous $\G$-action.  Let $\N =\M \otimes_\TTT \hat{\TTT}$.  Since $\N^\vee$ is isomorphic to $\Hom_\TTT(\M,\TTT)$ as an $\TTT$-module, $\N^\vee$ is also a free $\TTT$-module. Let $\SSS$ be a multiplicative set in $\TTT$. Let $\W_0=H^0(\G,\N)^\vee$,  $\W_1=H^1(\G,\N)^\vee$ and $\W_2 = H^2(\G,\N)^\vee$. We would like to thank Jan Nekov\'a\v{r} for indicating to us on how these projective dimensions can be examined.

\begin{proposition} \label{H0criterion}
Suppose the following conditions hold:
\begin{enumerate}
\item The $p$-cohomological dimension $cd_p(\G)$ of $\G$ is less than or equal to $2$.
\item $\SSS^{-1} \W_2 $ has finite projective dimension.
\end{enumerate}
 Then, $\SSS^{-1} \W_0$ has  finite projective dimension over $\SSS^{-1} \TTT$ if and only if $\SSS^{-1}\W_1$ has finite projective dimension over $\SSS^{-1}\TTT$.
\end{proposition}

\begin{proof}
Proposition 5.2.7 and Corollary 5.2.9 in \cite{neukirch2008cohomology} give us the following isomorphisms for all $i \geq 0$:
\begin{align*}
H^i(\G,\N)^\vee \cong \Tor_i^{\TTT_1[[\G]]} \left( \TTT, \N^\vee \right).
\end{align*} Let $\I_\G$ denote the augmentation ideal in $\TTT[[\G]]$. We have the short exact sequence of $\TTT[[\G]]$-modules
\begin{align}\label{Aug-R}
0 \rightarrow \I_\G \rightarrow \TTT[[\G]] \rightarrow \TTT \rightarrow 0.
\end{align}
Tensoring this short exact sequence by $\N^\vee$, we obtain an exact sequence
{\small \begin{align} \label{h0-h1-proj}
\underbrace{0}_{ \Tor_1^{\TTT[[\G]]} \left(\TTT[[\G]], \N^\vee \right)} & \rightarrow \underbrace{\Tor_1^{\TTT[[\G]]} \left(\TTT, \N^\vee \right)}_{H^1(\G,\N)^\vee} \rightarrow \\ & \notag \rightarrow \underbrace{\Tor_0^{\TTT[[\G]]} \left(\I_\G, \N^\vee \right)}_{\I_G \otimes_{\TTT[[\G]]}\N^\vee}\rightarrow  \underbrace{\Tor_{0}^{\TTT[[\G]]} \left(\TTT[[\G]],\N^\vee \right)}_{\N^\vee} \rightarrow \\ & \notag \rightarrow   \underbrace{\Tor_{0}^{\TTT[[\G]]} \left(\TTT,\N^\vee \right)}_{\cong H^0(\G,\N)^\vee}\rightarrow 0,
\end{align}
}
and for each $i \geq 1$, an isomorphism
\begin{align*}
\underbrace{\Tor^{\TTT[[\G]]}_{i+1}\left(\TTT,\N^\vee\right)}_{H^{i+1}(\G,\N)^\vee} \cong \Tor_i^{\TTT[[\G]]} \left( \I_\G,\N^\vee \right).
\end{align*}
By Corollary 5.2.13 in \cite{neukirch2008cohomology}, the projective dimension of $\TTT$ as an $\TTT[[\G]]$-module is equal to $\text{cd}_p(\G)$ (which is assumed to be less than or equal to $2$). By (\ref{Aug-R}), the projective dimension of the augmentation ideal $\I_\G$ over the ring $\TTT[[\G]]$ is less than or equal to $1$. Consider a projective resolution $0 \rightarrow \mathcal{P}_1 \rightarrow \mathcal{P}_0 \rightarrow \I_\G \rightarrow 0$ of $\TTT[[\G]]$-modules for the augmentation ideal $\I_\G$. Tensoring with $\N^\vee$, we obtain the following exact sequence:
\begin{align*}
0 \rightarrow \Tor_1^{\TTT[[\G]]} \left( \I_\G ,\N^\vee \right) \rightarrow \mathcal{P}_1 \otimes_{\TTT[[\G]]} \N^\vee \rightarrow \mathcal{P}_0 \otimes_{\TTT[[\G]]} \N^\vee \rightarrow  \I_\G \otimes_{\TTT[[\G]]} \N^\vee \rightarrow 0.
\end{align*}
The $\TTT$-modules $\mathcal{P}_1 \otimes_{\TTT[[\G]]} \N^\vee$ and $\mathcal{P}_0 \otimes_{\TTT[[\G]]} \N^\vee$ are direct summands of free $\TTT$-modules (since $\mathcal{P}_1$ and $\mathcal{P}_0$ are direct summands of free $\TTT[[\G]]$-modules) and hence they are projective $\TTT$-modules. Combining this observation with the fact that $\SSS^{-1}{\W_2}$ (which equals $\SSS^{-1}\Tor_1^{\TTT[[\G]]}\left(\I_\G,\N^\vee \right)$) has finite projective dimension over $\SSS^{-1}\TTT$, we can conclude that $\SSS^{-1}(\I_\G \otimes_{\TTT[[\G]]} \N^\vee)$ also has finite projective dimension over $\SSS^{-1}\TTT$. These observations along with (\ref{h0-h1-proj}) then complete the proof of the Proposition.
\end{proof}

\section{Control theorems} \label{control-theorems-section}

Let $\RRR_1$ and $\RRR_2$ be two integrally closed local domains that are finite extensions of $\Z_p[[u_1,\dotsc, u_n]]$ and $\Z_p[[v_1,\dotsc, v_m]]$ respectively. Consider a ring map $\RRR_1 \rightarrow \RRR_2$  such that $\RRR_2$ is finitely generated as an $\RRR_1$-module. Note that $\RRR_1$ and $\RRR_2$ are profinite rings. Let $\G$ be a profinite group. Let $\M$ be a finitely generated $\RRR_1$-module with a continuous $\RRR_1$-linear (left) action of $\G$.  In this way, $\M$ obtains an action of the completed group ring $\RRR_1[[\G]]$. Let us define the following modules:
$$\D_{\RRR_1} :=\M \otimes_{\RRR_1} \hat{\RRR_1} , \quad \D_{\RRR_2}:=  \M \otimes_{\RRR_1}  \hat{\RRR_2}, \quad  \C_{\RRR_1}:= \D_{\RRR_1}^\vee, \quad \C_{\RRR_2}:= \D_{\RRR_2}^\vee. $$

We also have the following isomorphisms:
\begin{align*}
& \D_{\RRR_2} \cong \M \otimes_{\RRR_1} \RRR_2 \otimes_{\RRR_2} \hat{\RRR_2},  \\ & \C_{\RRR_1} \cong \Hom_{\RRR_1}\left(\M , \RRR_1\right), \quad \C_{\RRR_2} \cong \Hom_{\RRR_1} \left(\M,\RRR_2\right) \cong \Hom_{\RRR_2} \left(\M \otimes_{\RRR_1} \RRR_2,\RRR_2\right).
\end{align*}
The last two isomorphisms are obtained by applying Hom-Tensor adjunction. Note that if $\M$ is a free $\RRR_1$-module, then so is $\C_{\RRR_1}$. Pontryagin duality (see Theorem 2.6.9 in \cite{neukirch2008cohomology}) also gives us the isomorphism $H^0(\G,\D_{\RRR_1})^\vee \cong H_0(\G,\C_{\RRR_1})$. Here, the $\RRR_1$-module $H_0(\G,\C_{\RRR_1})$ denotes the maximal quotient of $\C_{\RRR_1}$ on which $\G$ acts trivially. We now state the proposition which we will use to obtain our control theorems. In this proposition, we shall also frequently consider $\RRR_1$ and $\RRR_2$ as $\G$-modules with the trivial $\G$-action. \\

We will need to consider the notion of pseudo-compact modules as defined in \cite{brumer1966pseudocompact}. Consider the following categories:
\begin{align*}
& \mathrm{CAT}_{\RRR_1[[\G]]}: \text{category of pseudo-compact $\RRR_1[[\G]]$-modules}, \
& \mathrm{CAT}_{\RRR_1}: \text{category of pseudo-compact $\RRR_1$-modules}, \\
& \mathrm{CAT}_{\RRR_2[[\G]]}: \text{category of pseudo-compact $\RRR_2[[\G]]$-modules}, \   &  \mathrm{CAT}_{\RRR_2}: \text{category of pseudo-compact $\RRR_2$-modules}.
\end{align*}
Let $i \in \{1,2\}$. The completed tensor product $\text{---}\otimes_{\RRR_i[[\G]]} \text{---}$ over the group ring $\RRR_i[[\G]]$ and its higher derived functors given below are described in Chapter V, Section 2 in \cite{neukirch2008cohomology}. Consider the following commutative diagram and the following functors:
\begin{align*}
\xymatrix{
\mathrm{CAT}_{\RRR_1[[\G]]} \ar[r]^{\mathfrak{F}_{1}} \ar[d]_{\RRR_2[[\G]] \otimes_{\RRR_1[[\G]]} (\text{---})} & \mathrm{CAT}_{\RRR_1} \ar[d]^{\RRR_2 \otimes_{\RRR_1} (\text{---})}\\
\mathrm{CAT}_{\RRR_2[[\G]]} \ar[r]^{\mathfrak{F}_2} & \mathrm{CAT}_{\RRR_2}
} \qquad & \mathfrak{F}_1(\text{---}) = \RRR_1 \otimes_{\RRR_1[[\G]]} \text{---}, \quad \mathfrak{F}_2(\text{---}) = \RRR_2 \otimes_{\RRR_2[[\G]]} \text{---}.
\end{align*}

Note that there is a forgetful functor from $\mathrm{CAT}_{\RRR_i}$ to the category of $\RRR_i$-modules. The forgetful functor is an exact functor. Suppose $\M'$ and $\M''$ are finitely generated $\RRR_i$-modules. The tensor product $\M' \otimes_{\RRR_i} \M''$ in $\mathrm{CAT}_{\RRR_i}$,  in fact, coincides with the usual tensor product in the category of $\RRR_i$-modules (via the forgetful functor) (see Proposition 5.5.3 in \cite{ribes2000profinite}). Note that, for each positive integer $n$, the $\RRR_i$-module $\RRR_i^n$ is a projective object in $\mathrm{CAT}_{\RRR_i}$ and in the category of $\RRR_i$-modules. One can use this observation and the construction of the higher derived functors (for example, see Chapter 6 in \cite{ribes2000profinite}) to make the following deduction: for each non-negative integer $j$, the higher derived functors $\Tor_{j}^{\RRR_i}\left(\M',\  \M'' \right)$ in $\mathrm{CAT}_{\RRR_i}$, also coincides with the higher derived functors of the usual tensor product in the category of $\RRR_i$-modules (via the forgetful functor). We shall use this observation to view the $\RRR_i$-module $\Tor_{j}^{\RRR_i}\left(\M',\  \M'' \right)$ in both $\mathrm{CAT}_{\RRR_i}$ and in the category of $\RRR_i$-modules.

\begin{proposition}\label{control-theorem}
Let $\M$ be a free $\RRR_1$-module. We have the following exact sequence:
{\small \begin{align}
\Tor_{2}^{\RRR_1}\left(\RRR_2,\ H^0(\G,\D_{\RRR_1})^\vee \right) \rightarrow H^1(\G,\D_{\RRR_1})^\vee \otimes_{\RRR_1} \RRR_2 \rightarrow H^1(\G,\D_{\RRR_2})^\vee \rightarrow  \Tor^{\RRR_1}_1\left(\RRR_2,\ H^0(\G,\D_{\RRR_1})^\vee \right) \rightarrow 0.
\end{align}}
\end{proposition}

\begin{proof}
We begin with the following isomorphism:
\begin{align*}
\RRR_2 \otimes_{\RRR_1} \left(\RRR_1 \otimes_{\RRR_1[[\G]]} \C_{\RRR_1}\right) \cong \RRR_2 \otimes_{\RRR_1[[\G]]} \C_{\RRR_1}\cong \RRR_2 \otimes_{\RRR_2[[\G]]} \left(\RRR_2[[\G]] \otimes_{\RRR_1[[\G]]} \C_{\RRR_1}\right).
\end{align*}
Corollary 5.8.4 in \cite{weibel1995introduction} gives us the following spectral sequences:
\begin{align}\label{first-spectral}
\Tor_i^{\RRR_1}\left(\RRR_2,\Tor^{\RRR_1[[\G]]}_j\left(\RRR_1,\C_{\RRR_1} \right) \right)  \implies \Tor_{i+j}^{\RRR_1[[\G]]} \left(\RRR_2,\C_{\RRR_1}\right).
\end{align}
\begin{align} \label{second-spectral}
\Tor_i^{\RRR_2[[G]]}\left(\RRR_2,\Tor_j^{\RRR_1[[\G]]}\left(\RRR_2[[\G]],\C_{\RRR_1}\right)\right)  \implies \Tor_{i+j}^{\RRR_1[[\G]]} \left(\RRR_2,\C_{\RRR_1}\right)
\end{align}

The first spectral sequence (\ref{first-spectral}) gives us the exact sequence:
{ \begin{align} \label{first-long-es}
& \Tor_2^{\RRR_1}\left(\RRR_2,\Tor_0^{\RRR_1[[\G]]}\left(\RRR_1,\C_{\RRR_1} \right)\right) \rightarrow \RRR_2 \otimes_{\RRR_1} \Tor_1^{\RRR_1[[\G]]}\left(\RRR_1,\C_{\RRR_1}\right)   \rightarrow \\ & \notag \rightarrow  \Tor_1^{\RRR_1[[\G]]}\left( \RRR_2,\C_{\RRR_1}\right) \rightarrow \Tor_1^{\RRR_1}\left(\RRR_2,\Tor_0^{{\RRR_1}[[\G]]}\left(\RRR_1,\C_{\RRR_1}\right)\right) \rightarrow 0.
\end{align}
}
The second spectral sequence (\ref{first-spectral}) gives us the exact sequence:
{\begin{align*}
& \Tor_2^{\RRR_2[[\G]]}\left(\RRR_2,\RRR_2[[\G]] \otimes_{\RRR_1[[\G]]}\C_{\RRR_1} \right) \rightarrow   \RRR_2  \otimes_{\RRR_2[[\G]]} \Tor_1^{\RRR_1[[\G]]}\left(\RRR_2[[\G]],\C_{\RRR_1}\right)  \rightarrow \\ & \notag \rightarrow \Tor_1^{\RRR_1[[\G]]}\left( \RRR_2,\C_{\RRR_1}\right) \rightarrow \Tor_1^{\RRR_2[[\G]]}\left(\RRR_2,\RRR_2[[\G]] \otimes_{\RRR_1[[\G]]} \C_{\RRR_1} \right) \rightarrow 0.
\end{align*}}
According to Lemma 4.5 in \cite{brumer1966pseudocompact}, $\RRR_1[[\G]]$ is a projective object in $\mathrm{CAT}_{\RRR_1}$. Also, since $\RRR_2$ is a finitely generated $\RRR_1$-module, we have $\RRR_2[[\G]] \cong \RRR_2 \otimes_{\RRR_1} \RRR_1[[\G]]$ (for example, see Lemma 5.5.1 in \cite{ribes2000profinite}). By (a slight modification of the) flat base change theorems for the Tor functor (Proposition 3.2.9 in \cite{weibel1995introduction}), we obtain the following isomorphism for all $i\geq 0$:
\begin{align*}
\Tor_i^{\RRR_1[[\G]]} \left( \RRR_2[[\G]],\C_{\RRR_1} \right) \cong \Tor_i^{\RRR_1}\left( \RRR_2,\C_{\RRR_1}\right).
\end{align*}
Both these modules are zero whenever $i \geq 1$ since $\C_{\RRR_1}$ is a free $\RRR_1$-module. So, we have the following isomorphisms:
\begin{align*}
\Tor_1^{\RRR_1[[\G]]} \left( \RRR_2,\C_{\RRR_1}\right) \cong \Tor_1^{\RRR_2[[\G]]} \left( \RRR_2, \RRR_2[[\G]] \otimes_{\RRR_1[[\G]]} \C_{\RRR_1} \right) \cong \Tor_1^{\RRR_2[[\G]]} \left( \RRR_2, \RRR_2 \otimes_{\RRR_1} \C_{\RRR_1} \right).
\end{align*}
Since $\M$ is a free $\RRR_1$-module, we have the following $\G$-equivariant isomorphisms:
\begin{align*}
\C_{\RRR_2} \cong \Hom_{\RRR_1}\left(\M,\RRR_2\right) \cong \Hom_{\RRR_1}\left(\M,\RRR_1\right)\otimes_{\RRR_1}\RRR_2 \cong \C_{\RRR_1} \otimes_{\RRR_1}\RRR_2 \cong \RRR_2 \otimes_{\RRR_1} \C_{\RRR_1}.
\end{align*}
We obtain the following isomorphism:
\begin{align} \label{intermediate-iso-tor}
\Tor_1^{\RRR_1[[\G]]} \left( \RRR_2,\C_{\RRR_1}\right) \cong \Tor_1^{\RRR_2[[\G]]} \left( \RRR_2, \C_{\RRR_2} \right).
\end{align}
Proposition 5.2.7 and Corollary 5.2.9 in \cite{neukirch2008cohomology} give us the following isomorphisms for all $i \geq 0$:
\begin{align*}
H^i(\G,\D_{\RRR_1})^\vee \cong \Tor_i^{\RRR_1[[\G]]} \left( \RRR_1, \C_{\RRR_1} \right), &\quad H^i(\G,\D_{\RRR_2})^\vee \cong \Tor_i^{\RRR_2[[\G]]} \left( \RRR_2, \C_{\RRR_2}\right).
\end{align*}
The above isomorphisms along with (\ref{first-long-es}) and (\ref{intermediate-iso-tor}) complete the proof of the Proposition.
\end{proof}

We will now work with the ring map $\varpi: \TTT \rightarrow \RRR$ and all the notations and hypotheses given in the beginning of Section \ref{specialization-section}. Note that the group $\Gamma_p$ is defined to be the quotient $\Gal{\overline{\Q}_p}{\Q_p}/I_p$, where $I_p$ denotes the inertia subgroup at $p$. Suppose there is an $\TTT$-linear action of $\Gamma_p$ on a finitely generated $\TTT$-module $\M$, we will let $\M_{\Gamma_p}$ denote the maximal quotient of $\M$ on which $\Gamma_p$ acts trivially. Since $\Gamma_p$ is topologically generated by the Frobenius $\Frob_p$, we can identify $\M_{\Gamma_p}$ with the cokernel of the natural map $\M \xrightarrow {\Frob_p-1} \M$.

\begin{proposition} \label{control-theorem-selmer-groups}
For all $i \geq 1$, suppose that the $\RRR$-modules
{\small\begin{align*}
\Tor_i^\TTT\left(\RRR,H^0(G_\Sigma,\D_\varrho)^\vee\right), \qquad \Tor_i^\TTT\left(\RRR,H^0\left(I_p,\frac{\D_\varrho}{\Fil^+\D_\varrho}\right)^\vee\right)
\end{align*}}
are torsion. The $\RRR$-module $\Sel_{\varpi \circ \varrho}(\Q)^\vee$ is torsion if and only if the height one prime ideal $\ker(\varpi)$ in $\TTT$ does not belong to the support of the $\TTT$-module $\Sel_{\varrho}(\Q)^\vee$. \\

Suppose that the $\RRR$-module $\Sel_{\varpi \circ \varrho}(\Q)^\vee$ is torsion. In addition, we shall make the following assumptions:
\begin{enumerate}
\item Suppose $\QQ$ is a height two prime ideal in $\TTT$ containing $\ker(\varpi)$ and in the support of the $\TTT$-module $H^0(G_\Sigma, \D_\varrho)^\vee$ or $H^0\left(I_p,\frac{\D_{\varrho}}{\Fil^+\D_\varrho}\right)^\vee$. Then,  the $2$-dimensional local ring $\TTT_\QQ$ is regular and $\RRR_\QQ = \left(\frac{\TTT}{\ker(\varpi)}\right)_\QQ$.
\item\label{prop-surjectivity-assumption} The global-to-local map $\phi^{\Sigma_0}_{\varpi \circ \varrho}$ defining the non-primitive Selmer group $\Sel_{\varpi \circ \varrho}(\Q)$ is surjective.
\end{enumerate}
Under these assumptions, we have the following equality of divisors in $\RRR$:
{\footnotesize
\begin{align*}
 \Div\bigg(\Sel_{\varrho}(\Q)^\vee \otimes_{\TTT} \RRR \bigg) &+ \Div \bigg( \Tor_1^\TTT \left(\RRR, H^0(G_\Sigma, \D_{\varrho})^\vee\right) \bigg) =  \Div\bigg(\Sel_{\varpi \circ \varrho}(\Q)^\vee \bigg) &+ \Div\bigg(\Tor_1^\TTT\left(\RRR,H^0\left(I_p,\frac{\D_{\varrho}}{\Fil^+\D_{\varrho}}\right)^\vee\right)_{\Gamma_p}\bigg).
\end{align*}
}
\end{proposition}

\begin{proof}
We begin by first proving that the following $\RRR$-modules are pseudo-null:
{\small \begin{align}\label{PN-tor-2}
\Tor_2^\TTT \left(\RRR,H^0(G_\Sigma,\D_\varrho)^\vee\right), \qquad \Tor_2^\TTT\left(\RRR,H^0\left(I_p,\frac{\D_{\varrho}}{\Fil^+\D_\varrho}\right)^\vee\right).
\end{align}}
Let $\QQ$ be a height two prime ideal in $\TTT$ containing $\ker(\varpi)$. It is enough to show that the localizations at the prime ideal $\QQ$ of the $\TTT$-modules given in equation (\ref{PN-tor-2}) vanish. Further, it is enough to consider the case when $\QQ$ belongs to the support of $H^0(G_\Sigma,\D_\varrho)^\vee$ or $H^0\left(I_p,\frac{\D_{\varrho}}{\Fil^+\D_\varrho}\right)^\vee$. Observe that we have the following isomorphisms:
{\footnotesize\begin{align} \label{tor-iso-2}
\Tor_2^\TTT \left(\RRR,H^0(G_\Sigma,\D_\varrho)^\vee\right)_\QQ &\cong \Tor_2^{\TTT_\QQ} \left(\RRR_\QQ,\left(H^0(G_\Sigma,\D_\varrho)^\vee\right)_\QQ\right)   \cong \Tor_2^{\TTT_\QQ} \left(\frac{\TTT_\QQ}{\ker(\varpi)_\QQ},\left(H^0(G_\Sigma,\D_\varrho)^\vee\right)_\QQ\right). \\ \notag  \Tor_2^\TTT \left(\RRR,H^0\left(I_p,\frac{\D_{\varrho}}{\Fil^+\D_\varrho}\right)^\vee\right)_\QQ &\cong \Tor_2^{\TTT_\QQ} \left(\RRR_\QQ,\left(H^0\left(I_p,\frac{\D_{\varrho}}{\Fil^+\D_\varrho}\right)^\vee\right)_\QQ\right)  \cong  \Tor_2^{\TTT_\QQ} \left(\frac{\TTT_\QQ}{\ker(\varpi)_\QQ},\left(H^0\left(I_p,\frac{\D_{\varrho}}{\Fil^+\D_\varrho}\right)^\vee\right)_\QQ\right).
\end{align}}
The hypotheses in the proposition tell us that $\TTT_\QQ$ is a regular local ring and hence a UFD. The height one prime ideal $\ker(\varpi)_\QQ$ generates a principal ideal inside $\TTT_\QQ$. This forces the projective dimension of the $\TTT_\QQ$-module $\frac{\TTT_\QQ}{\ker(\varpi)_\QQ}$ to equal $1$ . The modules appearing in equation (\ref{tor-iso-2}) must thus vanish. This proves the claim that the $\RRR$-modules appearing in (\ref{PN-tor-2}) are pseudo-null. As a result, for every height one prime ideal $\q$ of $\RRR$, we have
{\small \begin{align}\label{vanish-pseudo-null-lemma}
\Tor_2^\TTT \left(\RRR,H^0(G_\Sigma,\D_\varrho)^\vee\right) \otimes_{\RRR} \RRR_\q \cong  \Tor_2^\TTT\left(\RRR,H^0\left(I_p,\frac{\D_{\varrho}}{\Fil^+\D_\varrho}\right)^\vee\right) \otimes_{\RRR} \RRR_\q \cong 0.
\end{align}}

We now claim that the rows of the following commutative diagram are exact in dimension one (see Appendix \ref{comm-algebra-appendix} for the definition of a sequence that is exact in dimension one):
{\footnotesize
\begin{align} \label{comm-control-selmer}
\xymatrix{
& \left(H^1\left(I_p,\frac{\D_\varrho}{\Fil^+\D_\varrho}\right)^\vee\otimes_\TTT \RRR \right)_{\Gamma_p}\ar[d]\ar[r]^{g_1}&   \left(H^1\left(I_p,\frac{\D_{\varpi \circ \varrho}}{\Fil^+\D_{\varpi \circ \varrho}}\right)^\vee\right)_{\Gamma_p} \ar[r]^{g_2 \qquad} \ar[d]^{\left(\phi^{\Sigma_0}_{\varpi \circ \varrho}\right)^\vee}& \Tor^\TTT_1\left(\RRR,H^0\left(I_p,\frac{\D_\varrho}{\Fil^+\D_\varrho}\right)^\vee\right)_{\Gamma_p}\ar[r]\ar[d]^{u}& 0\\
0 \ar[r]& H^1(G_\Sigma, \D_\varrho)^\vee \otimes_\TTT \RRR \ar[r]& H^1(G_\Sigma,\D_{\pi \circ \varrho})^\vee \ar[r]& \Tor_1^\TTT\left(\RRR,H^0(G_\Sigma,\D_{\varrho})^\vee\right) \rightarrow 0.
}
\end{align}
}
Proposition \ref{control-theorem}  immediately lets us obtain that the bottom row of the commutative diagram (\ref{comm-control-selmer}) is exact in dimension one. To see why the top row of the commutative diagram (\ref{comm-control-selmer}) is exact in dimension one, one needs to use Proposition \ref{control-theorem}, the fact that the $\RRR$-module {\small$\Tor^\TTT_2\left(\RRR,H^0\left(I_p,\frac{\D_\varrho}{\Fil^+\D_\varrho}\right)^\vee\right)$} is pseudo-null, Lemma \ref{almost-compact-exactness} and analyze the following commutative diagram: {\tiny \begin{align*}
\xymatrix{
\Tor^\TTT_2\left(\RRR,H^0\left(I_p,\frac{\D_\varrho}{\Fil^+\D_\varrho}\right)^\vee\right)\ar[r]\ar[d]^{\Frob_p-1}& H^1\left(I_p,\frac{\D_\varrho}{\Fil^+\D_\varrho}\right)^\vee\otimes_\TTT \RRR \ar[r]\ar[d]^{\Frob_p-1}&   H^1\left(I_p,\frac{\D_{\varpi \circ \varrho}}{\Fil^+\D_{\varpi \circ \varrho}}\right)^\vee  \ar[r]\ar[d]^{\Frob_p-1}& \Tor^\TTT_1\left(\RRR,H^0\left(I_p,\frac{\D_\varrho}{\Fil^+\D_\varrho}\right)^\vee\right)\ar[r]\ar[d]^{\Frob_p-1}\ar[r]& 0 \\
\Tor^\TTT_2\left(\RRR,H^0\left(I_p,\frac{\D_\varrho}{\Fil^+\D_\varrho}\right)^\vee\right)\ar[r]& H^1\left(I_p,\frac{\D_\varrho}{\Fil^+\D_\varrho}\right)^\vee\otimes_\TTT \RRR \ar[r]&   H^1\left(I_p,\frac{\D_{\varpi \circ \varrho}}{\Fil^+\D_{\varpi \circ \varrho}}\right)^\vee \ar[r] & \Tor^\TTT_1\left(\RRR,H^0\left(I_p,\frac{\D_\varrho}{\Fil^+\D_\varrho}\right)^\vee\right)\ar[r]& 0
}
\end{align*}}
Since, for each $i \geq 1$, the $\RRR$-modules $\Tor_i^\TTT\left(\RRR,H^0(G_\Sigma,\D_\varrho)^\vee\right)$ and $\Tor_i^\TTT\left(\RRR,H^0\left(I_p,\frac{\D_\varrho}{\Fil^+\D_\varrho}\right)^\vee\right)$ are torsion, we have
{\small\begin{align*}
\Tor_i^\TTT\left(\RRR,H^0(G_\Sigma,\D_\varrho)^\vee\right) \otimes_{\RRR} \Frac(\RRR) \cong \Tor_i^\TTT\left(\RRR,H^0\left(I_p,\frac{\D_\varrho}{\Fil^+\D_\varrho}\right)^\vee\right)\otimes_{\RRR} \Frac(\RRR) \cong 0.
\end{align*}}

Localizing all the modules in the commutative diagram (\ref{comm-control-selmer}) at the zero prime ideal $(0)$ of $\RRR$, we get the following isomorphism:
\begin{align*}
\left(\Sel_{\varrho}(\Q)^\vee \otimes_\TTT \RRR \right)\otimes_\RRR \Frac(\RRR) \xrightarrow {\cong} \left(\Sel_{\varpi \circ \varrho}(\Q)^\vee \right) \otimes_\RRR \Frac(\RRR).
\end{align*}
Therefore, we have the following implications:
\begin{align*}
\Sel_{\varpi \circ \varrho}(\Q)^\vee \text{ is $\RRR$-torsion } & \iff \left(\Sel_{\varpi \circ \varrho}(\Q)^\vee \right) \otimes_\RRR \Frac(\RRR) =0 \\ & \iff \left(\Sel_{\varrho}(\Q)^\vee \otimes_\TTT \RRR \right)\otimes_\RRR \Frac(\RRR) =0  \\ & \iff \text{The ideal $\ker(\varpi)$ is not in the support of $\left(\Sel_{\varrho}(\Q)^\vee \otimes_\TTT \RRR \right)$}
\\ & \iff \text{The ideal $\ker(\varpi)$ is not in the support of $\Sel_{\varrho}(\Q)^\vee$ (Nakayama's Lemma)}.
\end{align*}
This observation gives us the first part of the proposition. \\

Now let us suppose that $\ker(\varpi)$ does not belong to the support of the $\TTT$-module $\Sel_{\varrho}(\Q)^\vee$. Localizing all the modules in the commutative diagram (\ref{comm-control-selmer}) at every height one prime ideal $\q$ of $\RRR$ and using the Snake Lemma, we obtain the following short exact sequence:
\begin{align*}
0 \rightarrow \ker(u)_\q \rightarrow \left(\Sel_{\varrho}(\Q)^\vee \otimes_\TTT \RRR\right)_\q \rightarrow \left(\Sel_{\varpi\circ \varrho}(\Q)^\vee\right)_\q \rightarrow \left(\coker(u)\right)_\q \rightarrow 0.
\end{align*}
Note that assumption (\ref{prop-surjectivity-assumption}) tells us that the global-to-local map $\phi_{\varpi \circ \varrho}^{\Sigma_0}$ is surjective. So, the induced map $\left(\phi_{\varpi \circ \varrho}^{\Sigma_0}\right)^\vee$ on the Pontryagin duals is injective. We get the following equality in the divisor group of $\RRR$:
{\small\begin{align} \label{first-snake-lemma}
\Div\bigg(\ker(u)\bigg)  + \Div\bigg(\Sel_{\varpi\circ \varrho}(\Q)^\vee\bigg)= \Div\bigg(\Sel_{\varrho}(\Q)^\vee \otimes_\TTT \RRR \bigg) + \Div\bigg(\coker(u)\bigg).
\end{align}}
The exact sequence of torsion $\RRR$-modules
{\small\begin{align*}
0 \rightarrow \ker(u) \rightarrow \Tor^\TTT_1\left(\RRR,H^0\left(I_p,\frac{\D_\varrho}{\Fil^+\D_\varrho}\right)^\vee\right)_{\Gamma_p} \xrightarrow {u} \Tor_1^\TTT\left(\RRR,H^0(G_\Sigma,\D_{\varrho})^\vee\right) \rightarrow \coker(u) \rightarrow 0,
\end{align*}}
gives us the following equality of divisors in $\RRR$:
{\small \begin{align} \label{inter-ker-coker}
\Div\bigg(\ker(u)\bigg) &+ \Div\bigg(\Tor_1^\TTT\left(\RRR,H^0(G_\Sigma,\D_{\varrho})^\vee\right)\bigg) =  \Div\bigg(\Tor^\TTT_1\left(\RRR,H^0\left(I_p,\frac{\D_\varrho}{\Fil^+\D_\varrho}\right)^\vee\right)_{\Gamma_p}\bigg) + \Div\bigg(\coker(u)\bigg).
\end{align}}
Combining equations (\ref{first-snake-lemma}) and (\ref{inter-ker-coker}), we get the desired equality of divisors stated in the proposition.
\end{proof}

\section{Theorem \ref{selmer-factorization}} \label{section2-dasgupta-factorization}

\subsection{The setup in Dasgupta's factorization}

We will follow the summary of Hida's works (\cite{hida1986galois}, \cite{MR868300}, etc) given in Section 2 of the paper by Emerton, Pollack and Weston \cite{emerton2006variation}.

\begin{Remark}
We would like to make a remark about the hypotheses \ref{IRR} and \ref{p-Dis}. In general, the Galois representation, associated to the Hida family $F$, is only known to take values in the fraction field of $R$. The hypothesis \ref{IRR} allows us to find an integral model for that Galois representation. See Proposition 2.2.7 in \cite{emerton2006variation}. Similarly, in general (without assuming the validity of \ref{p-Dis}), one only has a $\Gal{\overline{\Q}_p}{\Q_p}$-equivariant filtration (similar to the one given in \ref{filtration-hida-family} below) over the fraction field of $R$. See Proposition 2.2.9 in \cite{emerton2006variation}.
\end{Remark}

Due to the hypothesis \ref{p-Dis}, we have the following short exact sequence of free $\R$-modules that is $\Gal{\overline{\Q}_p}{\Q_p}$-equivariant:
\begin{align} \label{filtration-hida-family}
\tag{Fil-$\rho_F$} 0 \rightarrow \Fil^+L_F \rightarrow L_F \rightarrow \frac{L_F}{ \Fil^+L_F} \rightarrow 0.
\end{align}
Here $\Fil^+L_F$ and $\frac{L_F}{ \Fil^+L_F}$ are free $R$-modules of rank $1$.
\begin{itemize}
\item  We let $\delta_F : \Gal{\overline{\Q}_p}{\Q_p} \rightarrow \R^\times$ denote the action of $\Gal{\overline{\Q}_p}{\Q_p}$ on   $\Fil^+L_F$.
\item We let $\epsilon_F : \Gal{\overline{\Q}_p}{\Q_p} \rightarrow \R^\times$ denote the action of $\Gal{\overline{\Q}_p}{\Q_p}$ on  $\frac{L_F}{\Fil^+L_F}$. Here, the character $\epsilon_F$ is unramified at $p$ and  $\epsilon_F(\Frob_p) = a_p(F)$.
\item Theorem 2.2.1 in \cite{emerton2006variation} describes the following properties of the determinant of the Galois representation $\rho_F$:
\begin{align} \label{determinant-properties}
\text{Image}\left(\det(\rho_F)\right) \subset \O[[x]], \qquad \text{Image}\left(\det(\rho_F)\right) \subsetneq \O, \qquad \ker(\det(\rho_F)) = \Gal{\Q_\Sigma}{H_\infty},
\end{align}
where $H_\infty$ is a finite extension of the cyclotomic $\Z_p$ extension $\Q_\infty$.
\end{itemize}

To the Galois representations $\rho_{\pmb{4},3}:G_\Sigma\rightarrow\Gl_4(T[[\Gamma]])$ and $\pi \circ \rho_{\pmb{4},3}:G_\Sigma\rightarrow\Gl_4(R[[\Gamma]])$, we associated the following modules:
{\footnotesize
\begin{align*}
&L_{\pmb{4},3} := \Hom_\T\left(L_F \otimes_{i_1} \T, L_F \otimes_{i_2} \T \right) (\chi) \otimes_T T[[\Gamma]](\kappa^{-1}), \ && L_{\pi} := \Hom_\R\left(L_F,L_F\right)(\chi) \otimes_R R[[\Gamma]](\kappa^{-1}), \\
&\Fil^+L_{\pmb{4},3} := \Hom_\T\left(L_F \otimes_{i_1} \T, \Fil^+L_F \otimes_{i_2} \T \right) (\chi) \otimes_T T[[\Gamma]](\kappa^{-1}), &&  \Fil^+L_\pi := \Hom_\R \left( L_F, \Fil^+L_F \right)  (\chi) \otimes_R R[[\Gamma]](\kappa^{-1}),\\  &\frac{L_{\pmb{4},3}}{\Fil^+L_{\pmb{4},3}} \cong \Hom_\T\left(L_F \otimes_{i_1} \T, \frac{L_F}{\Fil^+L_F} \otimes_{i_2} \T \right) (\chi) \otimes_T T[[\Gamma]](\kappa^{-1}), && \frac{L_\pi}{\Fil^+L_\pi} \cong \Hom_\R\left( L_F , \frac{L_F}{\Fil^+L_F} \right) (\chi) \otimes_R R[[\Gamma]](\kappa^{-1}).
\end{align*}
}

We get the following filtrations associated to $\rho_{\pmb{4},3}$ and $\pi \circ \rho_{\pmb{4},3}$:
\begin{align}
\tag{Fil-$\rho_{\pmb{4},3}$} & 0\rightarrow \Fil^+L_{\pmb{4},3} \rightarrow L_{\pmb{4},3} \rightarrow \frac{L_{\pmb{4},3}}{\Fil^+L_{\pmb{4},3}} \rightarrow 0. \\
\tag{Fil-$\pi \circ\rho_{\pmb{4},3}$}
& 0\rightarrow \Fil^+L_\pi \rightarrow L_\pi \rightarrow \frac{L_\pi}{\Fil^+L_\pi} \rightarrow 0.
\end{align}

\begin{Remark}
Note that $T[[\Gamma]]$ is a coproduct in the category of complete semi-local Noetherian $\O$-algebras (with respect to the maps $i_1(R) \hookrightarrow T[[\Gamma]]$, $i_2(R) \hookrightarrow T[[\Gamma]]$ and $O[[\Gamma]] \rightarrow T[[\Gamma]]$). One can consider the set $C$ of all continuous $\O$-algebra homomorphisms $\varphi:T[[\Gamma]] \rightarrow \overline{\Q}_p$ that is defined uniquely by considering the following maps:
\begin{align*}
i_1(R) \xrightarrow{\cong} R \xrightarrow {\varphi_1} \overline{\Q}_p, \qquad  i_2(R) \xrightarrow {\cong} R \xrightarrow {\varphi_2} \overline{\Q}_p, \qquad \varphi_\Gamma : \Gamma \rightarrow \overline{\Q}_p^\times.
\end{align*}
Here, we allow the ring homomorphisms $\varphi_1$ and $\varphi_2$ to vary over all the classical specializations of $F$ subject to the following restriction on their weights:
\begin{align*}
\mathrm{weight}(\varphi_2) > \mathrm{weight}(\varphi_1) \geq 2.
\end{align*}
We also allow $\varphi_\Gamma$ to vary over all the continuous group homomorphisms $\Gamma \rightarrow \overline{\Q}_p^\times$ of finite order. This set $C$, which is a subset of $\Hom_{\cont}\left(T[[\Gamma]],\overline{\Q}_p\right)$, is called the critical set of specializations (following Greenberg's terminology in \cite{greenberg1994iwasawa}) associated to $\rho_{\pmb{4},3}$. The $p$-adic $L$-function $\theta_{\pmb{4},3}$ satisfies an interpolation property that relates $\varphi(\theta_{\pmb{4},3})$, for every $\varphi \in C$,  to the value at $s=1$ of a complex $L$-function associated to the Galois representation $\varphi \circ \rho_{\pmb{4},3}$. \\

One can also consider associating a Selmer group to $\rho_{\pmb{4},3}$ corresponding to the following filtration:
{\small \begin{align*}
0 \rightarrow G^+L_{\pmb{4},3} \rightarrow L_{\pmb{4},3} \rightarrow \frac{L_{\pmb{4},3}}{G^+L_{\pmb{4},3}} \rightarrow 0,
\text{ where }  G^+L_{\pmb{4},3}:=\Hom_T \left(\frac{L_F}{\Fil^+L_F} \otimes_{i_1}T , L_F \otimes_{i_2}T \right) (\chi) \otimes T[[\Gamma]](\kappa^{-1}).
\end{align*}
}
This filtration also ensures that $\rho_{\pmb{4},3}$ satisfies the ``Panchishkin condition'' (though one must modify the critical set of specializations). In our arguments, it is essential that $\pi \circ \rho_{\pmb{4},3}$ does not inherit this filtration.
\end{Remark}

We also have the following short exact sequence of free $R[[\Gamma]]$-modules that is $\Gal{\overline{\Q}_p}{\Q_p}$-equivariant:
{\footnotesize
\begin{align}\label{ses-pirho3-fil}
& 0 \rightarrow \Hom_\R \left( \frac{L_F}{\Fil^+L_F}, \frac{L_F}{\Fil^+L_F} \right)  (\chi) \otimes_R R[[\Gamma]](\kappa^{-1}) \rightarrow  \frac{L_\pi}{\Fil^+L_\pi} \rightarrow  \Hom_\R\left(\Fil^+ L_F , \frac{L_F}{\Fil^+L_F} \right) (\chi) \otimes_R R[[\Gamma]](\kappa^{-1}) \rightarrow  0.
\end{align}
}
We have the 3-dimensional trace-zero adjoint representation $\Ad^0(\rho_F) : G_\Sigma \rightarrow \Gl_3(\R)$ associated to $\rho_F$. Let $\Ad^0(L_F)$ denote the free $\R$-module of rank $3$ on which $G_\Sigma$ acts to let us obtain the Galois representation $\Ad^0(\rho_F)$. We can consider the following free $R$-submodules of $\Ad^0(L_F)$ that are $\Gal{\overline{\Q}_p}{\Q_p}$-equivariant:
\begin{align*}
0 \subsetneq &  \underbrace{\Fil^{\mathrm{even}} \Ad^0(L_F)}_{\mathrm{Rank}=1} \subsetneq \underbrace{\Fil^\mathrm{odd} \Ad^0(L_F)}_{\text{Rank}=2} \subsetneq \Ad^0(L_F),
\end{align*}
where
\begin{align*}
& \Fil^\text{odd} \Ad^0(L_F) = \ker\left(\Hom_\R(L_F,L_F) \xrightarrow {\text{Res}} \Hom_\R\left(\Fil^+L_F, \frac{L_F}{\Fil^+L_F}\right)\right) \cap \Ad^0(L_F), \\
& \Fil^\text{even} \Ad^0(L_F) = \Hom_\R\left(\frac{L_F}{\Fil^+L_F},\Fil^+L_F\right).
\end{align*}
As $\Gal{\overline{\Q}_p}{\Q_p}$-modules, we have
\begin{align*}
\text{Fil}^\text{even} \Ad^0(L_F) \cong \R (\epsilon_F^{-1} \delta_F), &&  \frac{\text{Fil}^\text{odd} \Ad^0(L_F) }{\text{Fil}^\text{even} \Ad^0(L_F) } \cong \R,  &&  \frac{\Ad^0(L_F)}{\text{Fil}^\text{odd}\Ad^0(L_F)}  \cong \R(\epsilon_F \delta_F^{-1}).
\end{align*}
Let $L_{\pmb{3},2}$ denote the free $R[[\Gamma]]$-module on which $G_\Sigma$ acts to let us obtain $\rho_{\pmb{3},2}$ (which equals $\Ad^0(\rho_F)(\chi)~\otimes~\kappa^{-1}$). Let
\begin{align*}
\Fil^+{L_{\pmb{3},2}} = \left\{ \begin{array}{cc} \text{Fil}^{\text{even}}\Ad^0(L_F)(\chi)\otimes_R R[[\Gamma]](\kappa^{-1}), & \text{ if $\chi$ is even.}\\ \Fil^{\text{odd}}\Ad^0(L_F)(\chi)\otimes_R R[[\Gamma]](\kappa^{-1}), & \text{ if $\chi$ is odd} \end{array}\right.
\end{align*}
This gives us the filtration (which depends on the parity of $\chi$) for $\rho_{\pmb{3},2}$.
\begin{align}
\tag{Fil-$\rho_{\pmb{3},2}$} 0 \rightarrow \Fil^+L_{\pmb{3},2} \rightarrow L_{\pmb{3},2} \rightarrow \frac{L_{\pmb{3},2}}{\Fil^+L_{\pmb{3},2}} \rightarrow 0.
\end{align}

Our observations along with (\ref{ses-pirho3-fil}) lead us to the following short exact sequence:
{\small
\begin{align}\label{adjoint-local-ses}
0 & \longrightarrow \underbrace{\Hom_\R \left( \frac{L_F}{\Fil^+L_F}, \frac{L_F}{\Fil^+L_F} \right)  (\chi) \otimes_R R[[\Gamma]](\kappa^{-1})}_{R[[\Gamma]](\chi \kappa^{-1})} \longrightarrow \\ & \notag \longrightarrow  \underbrace{\frac{L_\pi}{\Fil^+L_\pi}}_{\left(\frac{\Ad^0(L_F)}{\Fil^\text{even}\Ad^0(L_F)}\right) (\chi) \otimes_R R[[\Gamma]](\kappa^{-1})} \longrightarrow \underbrace{\Hom_\R\left(\Fil^+ L_F , \frac{L_F}{\Fil^+L_F} \right) (\chi) \otimes_R R[[\Gamma]](\kappa^{-1})}_{\left(\frac{\Ad^0(L_F)}{\Fil^\text{odd}\Ad^0(L_F)}\right) (\chi) \otimes_R R[[\Gamma]](\kappa^{-1})}\longrightarrow  0.
\end{align}
}

The Pontryagin duals of the primitive and non-primitive Selmer group for $\rho_{\pmb{1},2}$ can be described in terms of primitive and non-primitive classical Iwasawa modules. Such a description is detailed in \cite{greenberg1989iwasawa} and \cite{greenberg2014p}. Let $L_{\pmb{1},2}$ equal $R[[\Gamma]](\chi \kappa^{-1})$. We let
\begin{align*}
\Fil^+{L_{\pmb{1},2}} = \left\{ \begin{array}{cc} L_{\pmb{1},2}, & \text{ if $\chi$ is even}\\ 0, & \text{ if $\chi$ is odd.} \end{array}\right.
\end{align*}
This gives us the filtration (which depends on parity of $\chi$) for $\rho_{\pmb{1},2}$.
\begin{align}
\tag{Fil-$\rho_{\pmb{1},2}$} 0 \rightarrow \Fil^+L_{\pmb{1},2} \rightarrow L_{\pmb{1},2} \rightarrow \frac{L_{\pmb{1},2}}{\Fil^+L_{\pmb{1},2}} \rightarrow 0.
\end{align}
Similar to (\ref{decomposition-twist}), we also have the following isomorphism of $R[[\Gamma]]$-modules that is $G_\Sigma$-equivariant:
\begin{align}\label{decomposition-twist-lattice}
L_\pi \cong L_{\pmb{3},2} \oplus L_{\pmb{1},2}.
\end{align}

\subsection{The proof of Theorem \ref{selmer-factorization}}

First note that we have a commutative diagram (\ref{comm-diag}) with the following properties: The top row in (\ref{comm-diag}) is exact. The bottom row in (\ref{comm-diag}) is coexact in dimension one (see the definition of an coexact in dimension one sequence in Appendix \ref{comm-algebra-appendix}).
{\small \begin{align} \label{comm-diag}
\xymatrix{
0 \ar[r]& H^1(G_\Sigma, D_{\rho_{\pmb{1},2}}) \ar[d]^{\phi^{\Sigma_0}_{\rho_{\pmb{1},2}}} \ar[r]& H^1(G_\Sigma, D_{\pi \circ \rho_{\pmb{4},3}}) \ar[d]^{\phi^{\Sigma_0}_{\pi \circ \rho_{\pmb{4},3}}}\ar[r]& H^1(G_\Sigma, D_{\rho_{\pmb{3},2}}) \ar[d]^{\phi^{\Sigma_0}_{\rho_{\pmb{3},2}}} \ar[r]& 0 \\
0 \ar[r]&  H^1\left(I_p,\frac{D_{\rho_{\pmb{1},2}}}{\Fil^+ D_{\rho_{\pmb{1},2}}}\right)^{\Gamma_p} \ar[r]^{g_1}& H^1\left(I_p,\frac{D_{\pi \circ \rho_{\pmb{4},3}}}{\Fil^+ D_{\pi \circ \rho_{\pmb{4},3}}}\right)^{\Gamma_p} \ar[r]^{g_2}& H^1\left(I_p,\frac{D_{\rho_{\pmb{3},2}}}{\Fil^+ D_{\rho_{\pmb{3},2}}}\right)^{\Gamma_p}
}
\end{align}}
The exactness of the top row of (\ref{comm-diag}) follows immediately from the decomposition (\ref{decomposition-twist-lattice}) that lets us obtain the following isomorphism:
\begin{align} \label{H1-decomp}
H^1(G_\Sigma, D_{\pi \circ \rho_{\pmb{4},3}})\cong H^1(G_\Sigma, D_{\rho_{\pmb{3},2}}) \oplus H^1(G_\Sigma, D_{\rho_{\pmb{1},2}}).
\end{align}
We now show that the bottom row is coexact in dimension one. It is easier to first consider the case when $\chi$ is even. In this case, (\ref{adjoint-local-ses}) gives us the following implications:
{\small \begin{align}\label{even-decomp}
& \frac{D_{ \rho_{\pmb{1},2}}}{\Fil^+ D_{ \rho_{\pmb{1},2}}} = 0 && \implies  H^1\left(I_p,\frac{D_{ \rho_{\pmb{1},2}}}{\Fil^+ D_{ \rho_{\pmb{1},2}}}\right)^{\Gamma_p} = 0.   \\ \notag
& \frac{D_{\pi \circ \rho_{\pmb{4},3}}}{\Fil^+ D_{\pi \circ \rho_{\pmb{4},3}}} \cong \frac{D_{\rho_{\pmb{3},2}}}{\Fil^+ D_{\rho_{\pmb{3},2}}}
&& \implies
 H^1\left(I_p,\frac{D_{\pi \circ \rho_{\pmb{4},3}}}{\Fil^+ D_{\pi \circ \rho_{\pmb{4},3}}}\right)^{\Gamma_p} \cong H^1\left(I_p,\frac{D_{\rho_{\pmb{3},2}}}{\Fil^+ D_{\rho_{\pmb{3},2}}}\right)^{\Gamma_p}.
\end{align}
}
Now consider the case when $\chi$ is odd. Since the inertia group $I_p$ has $p$-cohomological dimension one, (\ref{adjoint-local-ses}) lets us obtain the commutative diagram whose rows are exact.
{\small \begin{align*}
\xymatrix{
H^0\left(I_p,\frac{D_{\rho_{\pmb{3},2}}}{\Fil^+D_{\rho_{\pmb{3},2}}}\right) \ar[d]^{\Frob_p-1}\ar[r]& H^1\left(I_p,\frac{D_{\rho_{\pmb{1},2}}}{\Fil^+ D_{\rho_{\pmb{1},2}}}\right) \ar[d]^{\Frob_p-1}\ar[r]& H^1\left(I_p,\frac{D_{\pi \circ \rho_{\pmb{4},3}}}{\Fil^+ D_{\pi \circ \rho_{\pmb{4},3}}}\right) \ar[d]^{\Frob_p-1} \ar[r]& H^1\left(I_p,\frac{D_{\rho_{\pmb{3},2}}}{\Fil^+ D_{\rho_{\pmb{3},2}}}\right) \ar[d]^{\Frob_p-1}\ar[r]& 0 \\
H^0\left(I_p,\frac{D_{\rho_{\pmb{3},2}}}{\Fil^+D_{\rho_{\pmb{3},2}}}\right) \ar[r]& H^1\left(I_p,\frac{D_{\rho_{\pmb{1},2}}}{\Fil^+ D_{\rho_{\pmb{1},2}}}\right) \ar[r]& H^1\left(I_p,\frac{D_{\pi \circ \rho_{\pmb{4},3}}}{\Fil^+ D_{\pi \circ \rho_{\pmb{4},3}}}\right) \ar[r]& H^1\left(I_p,\frac{D_{\rho_{\pmb{3},2}}}{\Fil^+ D_{\rho_{\pmb{3},2}}}\right) \ar[r]& 0
}
\end{align*}
}
The Frobenius $\Frob_p$ is a topological generator for the pro-cyclic group $\Gamma_p$. Observe that we have the following isomorphisms:
{\small
\begin{align*}
H^0\left(\Gal{\overline{\Q}_p}{\Q_p},\frac{D_{\rho_{\pmb{3},2}}}{\Fil^+D_{\rho_{\pmb{3},2}}}\right) &\cong H^0\left(\Gamma_p,H^0\left(I_p,\frac{D_{\rho_{\pmb{3},2}}}{\Fil^+D_{\rho_{\pmb{3},2}}}\right)\right) \\ & \cong \ker \left( H^0\left(I_p,\frac{D_{\rho_{\pmb{3},2}}}{\Fil^+D_{\rho_{\pmb{3},2}}}\right)\xrightarrow{\Frob_p-1} H^0\left(I_p,\frac{D_{\rho_{\pmb{3},2}}}{\Fil^+D_{\rho_{\pmb{3},2}}}\right) \right).
\end{align*}
}
We will now show that
\begin{enumerate}[(A)]
\item\label{statement-A} $H^0\left(I_p,\frac{D_{\rho_{\pmb{3},2}}}{\Fil^+D_{\rho_{\pmb{3},2}}}\right)^\vee$ is a torsion $R[[\Gamma]]$-module, and \item\label{statement-B} $H^0\left(\Gal{\overline{\Q}_p}{\Q_p},\frac{D_{\rho_{\pmb{3},2}}}{\Fil^+D_{\rho_{\pmb{3},2}}}\right)^\vee$ is a pseudo-null $R[[\Gamma]]$-module.
\end{enumerate}
These statements along with Lemma \ref{almost-exactness} will then show that the bottom row of (\ref{comm-diag}) is coexact in dimension one. Statement \ref{statement-A} follows from Proposition \ref{greenberg-cyc}. When $\chi$ is odd,  $\Gal{\overline{\Q}_p}{\Q_p}$ acts on $\frac{D_{\rho_{\pmb{3},2}}}{\Fil^+D_{\rho_{\pmb{3},2}}}$ by the character $\delta_F^{-1}\epsilon_F \chi \kappa^{-1}$, which equals the character $\det(\rho_F)^{-1}\kappa^{-1}\chi \epsilon_F^{2}$.  In this case, the group $\Gamma_p$ acts on $H^0\left(I_p,\frac{D_{\rho_{\pmb{3},2}}}{\Fil^+D_{\rho_{\pmb{3},2}}}\right)^\vee$ via  $\varsigma \epsilon_F^{-2}$, for some  character $\varsigma$ of finite order. Observe also that $\epsilon_F(\Frob_p)$, which equals $a_p(F)$, is not an element of $\O$. This is because the value that $a_p(F)$ takes at classical specializations of $F$ varies as one varies the weight (see Lemma 3.2 in \cite{hida1985ap}). Let $Q$ be a prime ideal contained in the support of $H^0\left(\Gal{\overline{\Q}_p}{\Q_p},\frac{D_{\rho_{\pmb{3},2}}}{\Fil^+D_{\rho_{\pmb{3},2}}}\right)^\vee$. The prime ideal $Q$ should contain the element $\varsigma \epsilon_F^{-2}(\Frob_p)-1$, which is a non-zero element of $R$. Proposition $\ref{greenberg-cyc}$ tells us that there exists a monic polynomial $h(s)$ in $R[s]$ such that $h(\gamma_0)$ annihilates $H^0\left(\Gal{\overline{\Q}_p}{\Q_p},\frac{D_{\rho_{\pmb{3},2}}}{\Fil^+D_{\rho_{\pmb{3},2}}}\right)^\vee$. The prime ideal $Q$ must also contain this element $h(\gamma_0)$. As a result, the height of the prime ideal $Q$ must be at least two.  Statement \ref{statement-B}~also~follows.\\

Taking the Pontryagin dual of all the modules in the commutative diagram (\ref{comm-diag}) and considering the localization with respect to the prime ideal $(0)$ in $R[[\Gamma]]$, we get the following isomorphism over the fraction field of $R[[\Gamma]]$:
\begin{align*}
\Sel_{\rho_{\pmb{3},2}}(\Q)^\vee \otimes_{R[[\Gamma]]} \Frac(R[[\Gamma]]) \xrightarrow {\cong} \Sel_{\pi \circ \rho_{\pmb{4},3}}(\Q)^\vee \otimes_{R[[\Gamma]]} \Frac(R[[\Gamma]]).
\end{align*}
We have used the fact that the $R[[\Gamma]]$-module $\Sel_{\rho_{\pmb{1},2}}(\Q)^\vee$ is known to be torsion (see \cite{MR0349627}) and the fact that the global-to-local map $\phi^{\Sigma_0}_{\rho_{\pmb{1},2}}$ defining the non-primitive Selmer group $\Sel_{\rho_{\pmb{1},2}}(\Q)$ is surjective (see Proposition \ref{surjectivity-greenberg}). We have shown that the hypothesis \ref{ad-tor} holds if and only if $\Sel_{\pi \circ \rho_{\pmb{4},3}}(\Q)^\vee$ is a torsion $R[[\Gamma]]$-module.

Suppose \ref{ad-tor} holds. Taking the Pontryagin dual of all the modules in the commutative diagram (\ref{comm-diag}) and considering the localizations with respect to every height one prime ideal $\p$ in $R[[\Gamma]]$, a snake lemma argument will yield the following short exact sequence:
\begin{align*}
0 \rightarrow
\left(\Sel_{\rho_{\pmb{3},2}}(\Q)^\vee\right)_\p \rightarrow \left(\Sel_{\pi \circ \rho_{\pmb{4},3}}(\Q)^\vee\right)_\p \rightarrow \left(\Sel_{\rho_{\pmb{1},2}}(\Q)^\vee\right)_\p \rightarrow 0.
\end{align*}
We have once again used the fact that the $R[[\Gamma]]$-module $\Sel_{\rho_{\pmb{1},2}}(\Q)^\vee$ is known to be torsion and that the map $\phi^{\Sigma_0}_{\rho_{\pmb{1},2}}$ is surjective. This gives us the following equality of divisors in $R[[\Gamma]]$:
\begin{align} \label{equality-divisors}
\Div\left(\Sel_{\pi \circ \rho_{\pmb{4},3}}(\Q)^\vee\right) =
  \Div\left( \Sel_{\rho_{\pmb{3},2}}(\Q)^\vee\right) + \Div\left( \Sel_{\rho_{\pmb{1},2}}(\Q)^\vee \right).
\end{align}

The decomposition in (\ref{decomposition-twist-lattice}) also gives us the following isomorphism for all primes $\nu \in \Sigma_0$:
\begin{align*}
\Loc(\nu, \pi \circ \rho_{\pmb{4},3}) \cong \Loc(\nu, \rho_{\pmb{3},2}) \oplus \Loc(\nu, \rho_{\pmb{1},2}).
\end{align*}
This completes the proof of Theorem \ref{selmer-factorization}.

\begin{Remark}
Consider the following two special cases:
\begin{enumerate}
\item When $\chi$ is even, (\ref{even-decomp}) holds.
\item When $F$ is a CM-Hida family, we have the decomposition of $\Gal{\overline{\Q}_p}{\Q_p}$-modules: $$\frac{D_{\pi \circ \rho_{\pmb{4},3}}}{\Fil^+D_{\pi \circ \rho_{\pmb{4},3}}} \cong \frac{D_{\rho_{\pmb{3},2}}}{\Fil^+D_{\rho_{\pmb{3},2}}} \oplus \frac{D_{\rho_{\pmb{1},2}}}{\Fil^+D_{\rho_{\pmb{1},2}}}.$$
\end{enumerate}
Following the proof of Theorem \ref{selmer-factorization}, it is in these cases that one can establish the following decomposition of non-primitive Selmer groups:
\begin{align*}
\Sel_{\pi \circ \rho_{\pmb{4},3}}(\Q) \cong \Sel_{\rho_{\pmb{3},2}}(\Q) \oplus \Sel_{\rho_{\pmb{1},2}}(\Q).
\end{align*}
\end{Remark}

\section{Issues related to primitivity} \label{primitivity-issues}

\subsection{Differences between primitive and non-primitive Selmer groups}

Throughout this section, we will assume that the hypothesis \ref{ad-tor} holds. That is $\Sel_{\rho_{\pmb{3},2}}(\Q)^\vee$ is a torsion $R[[\Gamma]]$-module. The proof of Proposition \ref{primitive-non-primitive-difference} does not rely on Theorem \ref{selmer-factorization} and Theorem \ref{specialization-result}. Using Theorem \ref{selmer-factorization} and Theorem \ref{specialization-result}, one can deduce that whenever the hypothesis \ref{ad-tor} is valid,  $\Sel_{\rho_{\pmb{4},3}}(\Q)^\vee$ is  a torsion $T[[\Gamma]]$-module. We would first like to observe that as a direct consequence of Proposition \ref{primitive-non-primitive-difference}, we have the following equality in the divisor group of $T[[\Gamma]]$:
\begin{align*}
\Div\left(\Sel_{\rho_{\pmb{4},3}}(\Q)^\vee\right) - \sum \limits_{\nu \in \Sigma_0}\Div\left(\Loc(\nu,\rho_{\pmb{4},3})^\vee\right)=  \Div\left(\S_{\rho_{\pmb{4},3}}(\Q)^\vee\right) - \Div\left(H^0(G_\Sigma,D^*_{\rho_{\pmb{4},3}})^\vee\right).
\end{align*}
We have the following equalities in the divisor group of $R[[\Gamma]]$:
\begin{align*}
\Div\left(\Sel_{\rho_{\pmb{3},2}}(\Q)^\vee\right) - \sum \limits_{\nu \in \Sigma_0}\Div\left(\Loc(\nu,\rho_{\pmb{3},2})^\vee\right)&=  \Div\left(\S_{\rho_{\pmb{3},2}}(\Q)^\vee\right) - \Div\left(H^0(G_\Sigma,D^*_{\rho_{\pmb{3},2}})^\vee\right). \\
\Div\left(\Sel_{\rho_{\pmb{1},2}}(\Q)^\vee\right) - \sum \limits_{\nu \in \Sigma_0}\Div\left(\Loc(\nu,\rho_{\pmb{1},2})^\vee\right) &=  \Div\left(\S_{\rho_{\pmb{1},2}}(\Q)^\vee\right) - \Div\left(H^0(G_\Sigma,D^*_{\rho_{\pmb{1},2}})^\vee\right).
\end{align*}
It is this observation that enables us to write the main conjectures \ref{mainconj-3}, \ref{mainconj-2} and \ref{mainconj-1} in terms of the non-primitive Selmer group (instead of the primitive Selmer group), the local factors away from $p$ and the primitive $p$-adic $L$-function, as mentioned in the introduction.
\begin{proposition} \label{primitive-non-primitive-difference}
We have the following exact sequence:
\begin{align} \label{surj-rho1}
0 \rightarrow \S_{\rho_{\pmb{1},2}}(\Q) \rightarrow \Sel_{\rho_{\pmb{1},2}}(\Q) \rightarrow \prod \limits_{\nu \in \Sigma_0} \ \Loc(\nu,\rho_{\pmb{1},2})  \rightarrow \coker (\phi_{\rho_{\pmb{1},2}}) \rightarrow 0.
\end{align}
In the divisor group of $R[[\Gamma]]$, we have $$\Div\left( \coker (\phi_{\rho_{\pmb{1},2}})^\vee \right)  = \Div\left( H^0(G_\Sigma, D^*_{{\rho}_{\pmb{1},2}})^\vee\right).$$
In addition if we assume \ref{ad-tor} holds, we also have the following exact sequences:
\begin{align} \label{surj-rho3}
&0 \rightarrow \S_{\rho_{\pmb{4},3}}(\Q) \rightarrow \Sel_{\rho_{\pmb{4},3}}(\Q) \rightarrow \prod \limits_{\nu \in \Sigma_0} \ \Loc(\nu,\rho_{\pmb{4},3}) \rightarrow 0.
\end{align}
\begin{align} \label{surj-rho2}
&0 \rightarrow \S_{\rho_{\pmb{3},2}}(\Q) \rightarrow \Sel_{\rho_{\pmb{3},2}}(\Q) \rightarrow \prod \limits_{\nu \in \Sigma_0} \ \Loc(\nu,\rho_{\pmb{3},2})  \rightarrow   \coker (\phi_{\rho_{\pmb{3},2}})  \rightarrow 0.
\end{align}
The map $\phi_{\rho_{\pmb{4},3}}$ is surjective. In the divisor group of $T[[\Gamma]]$, we have $$  \Div(H^0(G_\Sigma, D^*_{\rho_{\pmb{4},3}})^\vee) = 0.$$
In the divisor group of $R[[\Gamma]]$, we have
$$ \Div\left( \coker (\phi_{\rho_{\pmb{3},2}})^\vee \right) = \Div\left( H^0(G_\Sigma, D^*_{\rho_{\pmb{3},2}})^\vee\right).$$
\end{proposition}
The proof of Proposition \ref{primitive-non-primitive-difference} follows from Proposition \ref{surjectivity-greenberg} along with Lemmas \ref{coker-3} - \ref{coker-2-cm}.

\begin{lemma}[Cokernel of $\phi_{\rho_{\pmb{4},3}}$] \label{coker-3}
The map $\phi_{\rho_{\pmb{4},3}}$ is surjective. The $T[[\Gamma]]$-module $H^0(G_\Sigma, D^*_{\rho_{\pmb{4},3}})^\vee$ is pseudo-null .
\end{lemma}

\begin{proof}
Consider a character $\Psi : G_\Sigma \rightarrow A(T)^\times$. Here, $A(T)$ denotes the algebraic closure of the fraction field of $T$. By Proposition \ref{surj-PN-residuefield} (and following the notations there), it suffices to show $H^0(G_\Sigma, \mathcal{V}_{\rho_{\pmb{4},3}^*(\Psi)}) \stackrel{?}{=}0$, where
$$\mathcal{V}_{\rho_{\pmb{4},3}^*(\Psi)} := \Hom_{A(T)}\bigg( L_F\otimes_{i_2} A(T) , L_F \otimes_{i_1} A(T) (\chi^{-1}\chi_p \Psi)\bigg). $$
Recall that $T$ is an integral extension of $\O[[x_1,x_2]]$, where $x_1$ and $x_2$ are used to denote the ``weight variables''. Let $\sigma_1:G_\Sigma \rightarrow \Gl_2(A(T))$ and $\sigma_2 : G_\Sigma \rightarrow \Gl_2(A(T))$ be the Galois representations given by the action of $G_\Sigma$ on $L_F \otimes_{i_1}A(T)( \chi^{-1}\chi_p \Psi)$ and $L_F \otimes_{i_2}A(T)$ respectively. The representations $\sigma_1$ and $\sigma_2$ are both irreducible. We will simply show that that the semi-simplifications of $\sigma_1 \mid_{I_p}$ and $\sigma_2 \mid_{I_p}$ are not the same.

\begin{itemize}
\item Let $\sigma_2 \mid_{I_p} $ denote the restriction of $\sigma_2$ to the inertia subgroup at $p$. The semisimplification of $\sigma_2 \mid_{I_p} $ is a direct sum of two characters: $i_2 \circ \delta_F$ and the trivial character. By (\ref{determinant-properties}), the image of the character $i_2 \circ \delta_F$ lies inside $\O[[x_2]]$ but not inside $\O$.
\item Let $\sigma_1 \mid_{I_p} $ denote the restriction of $\sigma_1$ to the inertia subgroup at $p$.  The semisimplification of $\sigma_1 \mid_{I_p} $ is a direct sum of two characters: $(i_1 \circ \delta_F)  \chi^{-1} \chi_p \Psi$ and $ \chi^{-1} \chi_p \Psi$. By (\ref{determinant-properties}), the image of $i_1 \circ \delta_F$ lies inside $\O[[x_1]]$ but not inside $\O$.
\end{itemize}

Note that $\O[[x_1]] \cap \O[[x_2]] = \O$. As a result, both the ratios $\frac{(i_1 \circ \delta_F)  \chi^{-1} \chi_p \Psi}{\chi^{-1} \chi_p \Psi}$  and $\frac{\chi^{-1} \chi_p \Psi}{(i_1 \circ \delta_F)  \chi^{-1} \chi_p \Psi}$ are valued in $\O[[x_1]]$ and hence cannot equal $\frac{i_2 \circ \delta_F}{1}$. So, the semi-simplifications of $\sigma_1 \mid_{I_p} $ and $\sigma_2 \mid_{I_p} $ cannot be isomorphic. Consequently, $\sigma_1$ and $\sigma_2$ cannot be isomorphic as $G_\Sigma$-representations over $A(T)$. Using these observations, we can now conclude that $H^0(G_\Sigma, \mathcal{V}_{\rho_{\pmb{4},3}^*(\Psi)})=0.$
\end{proof}

\begin{lemma}[Cokernel of $\phi_{\rho_{\pmb{1},2}}$] \label{coker-1}
In the divisor group of $R[[\Gamma]]$, we have the following equality: $$\Div\left( \coker(\phi_{\rho_{\pmb{1},2}})^\vee \right) = \Div\left(H^0(G_\Sigma, D^*_{\rho_{\pmb{1},2}})^\vee \right).$$
\end{lemma}

\begin{proof}

Since $\chi$ is a finite character, it is possible to write $\chi$ as the product $\varsigma \cdot \tau$, where the two characters $\varsigma$ and $\tau$ are described below.
\begin{itemize}
\item $\varsigma : G_\Sigma \rightarrow \O^\times$ is a finite character such that $\overline{\Q}^{\ker(\varsigma)} \cap \Q_\infty = \Q$.
\item $\tau : G_\Sigma \twoheadrightarrow \Gamma \rightarrow R[[\Gamma]]^\times$ is a finite character that factors through $\Gal{\Q_\infty}{\Q}$.
\end{itemize}

The cyclotomic character $\chi_p$ can be written as a product $\omega \varkappa$.  Here $\omega: G_\Sigma \rightarrow \O^\times$ is the Teichmuller character and $\varkappa: G_\Sigma \twoheadrightarrow \Gamma \xrightarrow {\cong} 1+p\Z_p \hookrightarrow \O^\times$ is the natural Galois character that factors through $\Gamma$. The Galois representation $\rho_{\pmb{1},2}^*$ equals $\varsigma^{-1}\omega \tau^{-1} \varkappa \kappa$. Note that the character $\tau^{-1} \varkappa \kappa : G_\Sigma \rightarrow R[[\Gamma]]^\times$ factors through $\Gamma$. First suppose $\varsigma^{-1} \omega$ is non-trivial. Then $G_\Sigma$ acts non-trivially on $V_{\varsigma^{-1} \omega\Psi}$, for all characters $\Psi :G_\Sigma \rightarrow A(R)^\times$ that factor through $\Gamma$. Here, $A(R)$ is the algebraic closure of the fraction field of $R$ and $V_{\varsigma^{-1} \omega\Psi}$ is the one-dimensional $A(R)$ vector space on which $G_\Sigma$ acts by $\varsigma^{-1} \omega\Psi$. By Proposition \ref{surj-PN-residuefield}, the map $\phi_{\rho_{\pmb{1},2}}$ is surjective and $H^0(G_\Sigma, D^*_{\rho_{\pmb{1},2}})^\vee$ is a pseudo-null $R[[\Gamma]]$-module. Hence in this case, $\Div\left( \coker(\phi_{\rho_{\pmb{1},2}})^\vee \right) = \Div\left(H^0(G_\Sigma, D^*_{\rho_{\pmb{1},2}})^\vee \right)=0$. \\

Now let us suppose $\varsigma^{-1}\omega$ is trivial. In this case,
\begin{align} \label{iso-discrete}
H^0\left(G_\Sigma,D^*_{\rho_{\pmb{1},2}}\right)^\vee = H^0\left(G_\Sigma, \hat{R[[\Gamma]]}(\tau^{-1} \varkappa \kappa) \right)^\vee \cong \frac{R[[\Gamma]]}{\left(\tau^{-1} \varkappa \kappa(\gamma_0)-1\right)}.
\end{align}
Also, $\Fil^+D_{\rho_{\pmb{1},2}}$ equals zero since $\chi$ (which equals $\omega \tau$ in this case) is an odd character. The inflation-restriction exact sequence gives us the following equality:
$$ H^1\left(\Gamma_p, H^0\left(I_p,D_{\rho_{\pmb{1},2}}\right)\right) \cong \ker \bigg( H^1(\Gal{\overline{\Q}_p}{\Q_p},D_{\rho_{\pmb{1},2}}) \xrightarrow {\alpha_p} H^1\left(I_p,D_{\rho_{\pmb{1},2}}\right)^{\Gamma_p} \bigg). $$
The residual representation associated to $\rho_{\pmb{1},2}$ is ramified at $p$ (since the residual representation associated to the Teichmuller character $\omega$ is ramified at $p$). This lets us conclude that $H^0\left(I_p,D_{\rho_{\pmb{1},2}}\right)=0$. Hence, $\ker(\alpha_p)=0$. One can apply Proposition \ref{surjectivity-greenberg} to obtain the isomorphism  $\coker(\phi_{\rho_{\pmb{1},2}})^\vee \cong H^1(G_\Sigma, L^*_{\pmb{1},2})_{\text{tor}}$. If $\xi$ is a non-zero annihilator for $H^1(G_\Sigma, L_{\pmb{1},2}^*)_{\text{tor}}$, then so is $\xi (\tau^{-1} \varkappa \kappa(\gamma_0)-1)$. Equation (\ref{useful-tor-h1}) gives us the following isomorphism of $R[[\Gamma]]$-modules:
\begin{align} \label{iso-compact}
\coker(\phi_{\rho_{\pmb{1},2}})^\vee \cong H^1(G_\Sigma, L_{\pmb{1},2}^*)_{\text{tor}}  \cong \frac{(\xi)}{(\xi)(\tau^{-1} \varkappa \kappa(\gamma_0)-1)}\cong \frac{R[[\Gamma]]}{(\tau^{-1} \varkappa \kappa(\gamma_0)-1)}.
\end{align}
From equations (\ref{iso-discrete}) and (\ref{iso-compact}), note that all the $R[[\Gamma]]$-modules appearing in (\ref{iso-compact}) are isomorphic to $H^0\left(G_\Sigma,D^*_{\rho_{\pmb{1},2}}\right)^\vee $. As a result, we have $\Div\left( \coker(\phi_{\rho_{\pmb{1},2}})^\vee \right) = \Div\left(H^0(G_\Sigma, D^*_{\rho_{\pmb{1},2}})^\vee \right)$.
\end{proof}

\begin{Remark} \label{auxillary-remark-character}
Since we haven't placed any restrictions on the Dirichlet character $\chi$, note that Lemma \ref{coker-1} is valid for the cyclotomic deformation of any Dirichlet character. We shall use this fact to prove Lemma \ref{coker-2-cm}.
\end{Remark}

The investigation of $\coker(\phi_{\rho_{\pmb{3},2}})$ will depend on whether $F$ is a CM-Hida family or not (see section 3 of \cite{hida2013image} for a definition of CM  Hida families). If $F$ is a not a CM Hida family, the adjoint representation $\Ad^0(\rho_F)$ is absolutely irreducible over the fraction field of $R$. Using Corollary \ref{abs-irreducible}, we have the following lemma:
\begin{lemma}[Cokernel of $\phi_{\rho_{\pmb{3},2}}$ - Non CM-Hida family]\label{coker-2-noncm}
Suppose $F$ is not a CM Hida family.
\begin{itemize}
\item The $R[[\Gamma]]$-module $H^0(G_\Sigma, D^*_{\rho_{\pmb{3},2}})^\vee$ is pseudo-null.
\item The map $\phi_{\rho_{\pmb{3},2}}$ is surjective.
\end{itemize}
\end{lemma}

Suppose now that $F$ is a CM Hida family with complex multiplication by the imaginary quadratic field $K$ (where the prime $p$ splits). In this case, $\rho_F \cong \text{Ind}^{G_\Q}_{G_K}(\varphi)$, for some continuous character $\varphi : \Gal{\overline{\Q}}{K} \rightarrow \R^\times$ (see Proposition 3.2 in \cite{hida2013image}). That is, we have the following isomorphism of free $R$-modules that is $G_\Sigma$-equivariant:
\begin{align} \label{CM-induction-R}
L_F \cong \Ind^{G_\Q}_{G_K} \left(R(\phi)\right).
\end{align}
Let $\varepsilon : \Gal{\overline{\Q}}{\Q} \rightarrow \{\pm 1\}$ be the quadratic character associated to $K$. Let $c$ denote any lift in $\Gal{\overline{\Q}}{\Q}$ of the non-trivial element of $\Gal{K}{\Q}$. We define the character $\varphi^c : \Gal{\overline{\Q}}{K} \rightarrow R^\times$ as follows - $$\varphi^c(g) := \varphi(c g c^{-1}), \text{\ for all $g \in \Gal{\overline{\Q}}{K}$.}$$

Since the prime $p$ splits in the imaginary quadratic field $K$, one can view the decomposition group $\Gal{\overline{\Q}_p}{\Q_p}$ as a subgroup of $\Gal{\overline{\Q}}{K}$. This observation along with (\ref{CM-induction-R}) tells us that, for a CM Hida-family, the Galois representation $\rho_F$ splits locally at $p$. That is, we have the following isomorphism of free $R$-modules that is $\Gal{\overline{\Q}_p}{\Q_p}$-equivariant:
\begin{align*}
L_F \cong  R(\delta_F) \oplus R(\epsilon_F).
\end{align*}
As a result, we have
\begin{align}
\frac{L_{\pmb{3},2}}{\Fil^+L_{\pmb{3},2}} = \left\{ \begin{array}{ll}   R[[\Gamma]](\delta_F^{-1} \epsilon_F\chi \kappa^{-1}) \oplus R[[\Gamma]](\chi \kappa^{-1}) ,  & \text{if $\chi$ is even} \\  R[[\Gamma]](\delta_F^{-1} \epsilon_F\chi \kappa^{-1}),  & \text{ if $\chi$ is odd} \end{array} \right.
\end{align}
Let $\Omega :G_\Sigma \rightarrow \Gl_2(R[[\Gamma]])$ be the Galois representation given by $\Ind_\Q^K\left( \frac{\varphi}{\varphi^c}\right)$. Let $L_{\Omega(\chi) \otimes \kappa^{-1}}$ and $L_{\varepsilon (\chi) \otimes \kappa^{-1}}$ denote the free $R[[\Gamma]]$-modules on which $G_\Sigma$ acts to let us obtain $\Omega (\chi) \otimes \kappa^{-1}$ and $\varepsilon (\chi) \otimes \kappa^{-1}$.  Using (\ref{CM-induction-R}), we have the following isomorphism of $R[[\Gamma]]$-modules that is $G_\Sigma$-equivariant:
\begin{align} \label{global-decomposition-CM}
L_{\pmb{3},2} \cong L_{\Omega (\chi) \otimes \kappa^{-1}} \oplus L_{\varepsilon (\chi) \otimes\kappa^{-1}}.
\end{align}
One can form primitive Selmer groups $\S_{\Omega (\chi) \otimes \kappa^{-1}}(\Q)$ and $\S_{\varepsilon (\chi) \otimes \kappa^{-1}}(\Q)$ corresponding to the following filtrations:
\begin{align}
\tag{Fil-$\Omega (\chi) \otimes \kappa^{-1}$} & 0 \rightarrow \Fil^+L_{\Omega (\chi) \otimes \kappa^{-1}} \rightarrow L_{\Omega (\chi) \otimes \kappa^{-1}} \rightarrow \frac{L_{\Omega (\chi) \otimes \kappa^{-1}}}{\Fil^+L_{\Omega (\chi) \otimes \kappa^{-1}}} \rightarrow 0.\\
\tag{Fil-$\varepsilon (\chi) \otimes\kappa^{-1}$} & 0 \rightarrow \Fil^+L_{\varepsilon (\chi) \otimes\kappa^{-1}} \rightarrow L_{\varepsilon (\chi) \otimes\kappa^{-1}} \rightarrow \frac{L_{\varepsilon (\chi) \otimes\kappa^{-1}}}{\Fil^+L_{\varepsilon (\chi) \otimes\kappa^{-1}}}\rightarrow 0.
\end{align}
Here, we have $L_{\varepsilon (\chi) \otimes\kappa^{-1}} = R[[\Gamma]](\varepsilon \chi \kappa^{-1})$, and
\begin{align*}
\Fil^+ L_{\Omega (\chi) \otimes \kappa^{-1}} := R[[\Gamma]](\epsilon_F^{-1}\delta_F \chi \kappa^{-1}), & \qquad \Fil^+L_{\varepsilon (\chi) \otimes\kappa^{-1}} := \left\{ \begin{array}{ll} 0, & \text{if $\chi$ is even} \\
R[[\Gamma]](\varepsilon \chi \kappa^{-1}), & \text{if $\chi$ is odd}
\end{array} \right.
\end{align*}

Locally, we have the following decomposition of free $R[[\Gamma]]$-modules that is $\Gal{\overline{\Q}_p}{\Q_p}$-equivariant:
\begin{align} \label{local-decomposition-CM}
 \frac{L_{\pmb{3},2}}{\Fil^+L_{\pmb{3},2}} \cong \frac{L_{\Omega (\chi) \otimes \kappa^{-1}}}{\Fil^+L_{\Omega (\chi) \otimes \kappa^{-1}}} \oplus \frac{L_{\varepsilon (\chi) \otimes \kappa^{-1}}}{\Fil^+ L_{\varepsilon (\chi) \otimes \kappa^{-1}}}.
\end{align}

The decompositions (\ref{global-decomposition-CM}) and (\ref{local-decomposition-CM}) allow us to obtain the following isomorphisms:
\begin{align*}
\S_{\rho_{\pmb{3},2}}(\Q) \cong \S_{\Omega (\chi) \otimes \kappa^{-1}}(\Q) \oplus \S_{\epsilon (\chi) \otimes \kappa^{-1}}(\Q), \qquad & \coker(\phi_{\rho_{\pmb{3},2}}) \cong \coker(\phi_{\Omega (\chi) \otimes \kappa^{-1}}) \oplus \coker(\phi_{\varepsilon (\chi) \otimes \kappa^{-1}}).
\end{align*}

Since $\rho_F$ is absolutely irreducible, the 2-dimensional  representation $\Omega$ is also absolutely irreducible. By Corollary \ref{abs-irreducible}, the $R[[\Gamma]]$-module $H^0\left(G_\Sigma,D^*_{\Omega(\chi)\otimes\kappa^{-1}}\right)^\vee$ is pseudo-null. The isomorphism $$ H^0\left(G_\Sigma, D^*_{\rho_{\pmb{3},2}}\right)  \cong H^0\left(G_\Sigma, D^*_{\varepsilon (\chi) \otimes \kappa^{-1}}\right)  \oplus H^0\left(G_\Sigma,D^*_{\Omega (\chi) \otimes \kappa^{-1}}\right),$$ then gives us the following equality in the divisor group of $R[[\Gamma]]$:
\begin{align*}
\Div\left(H^0\left(G_\Sigma, D^*_{\rho_{\pmb{3},2}}\right)^\vee\right) = \Div\left( H^0\left(G_\Sigma, D^*_{\varepsilon (\chi) \otimes \kappa^{-1}}\right)^\vee\right).
\end{align*}
By Corollary \ref{abs-irreducible}, the map $\phi_{\Omega (\chi) \otimes \kappa^{-1}}$ is surjective. This tells us that $$\coker(\phi_{\rho_{\pmb{3},2}})\cong\coker(\phi_{\varepsilon (\chi) \otimes\kappa^{-1}}).$$
Lemma \ref{coker-1} (see also Remark \ref{auxillary-remark-character}) gives us the following equality in the divisor group of $R[[\Gamma]]$:
$$ \Div \left( \coker(\phi_{\varepsilon (\chi) \otimes\kappa^{-1}})^\vee \right)=\Div \left( H^0(G_\Sigma, D^*_{\varepsilon (\chi) \otimes \kappa^{-1}})^\vee \right).$$
These observations give us the following lemma:

\begin{lemma}[Cokernel of $\phi_{\rho_{\pmb{3},2}}$ - CM Hida family] \label{coker-2-cm}
In the divisor group of $R[[\Gamma]]$, we have the equality $\Div\left(H^0(G_\Sigma, D^*_{\rho_{\pmb{3},2}})^\vee\right) = \Div\left(\coker(\phi_{\rho_{\pmb{3},2}})^\vee\right)$.
\end{lemma}

\subsection{Divisors associated to local factors away from $p$}

Just as the main conjecture formulated in \cite{greenberg1994iwasawa} relates the primitive $p$-adic $L$-function and the primitive Selmer group, one should be able to formulate an equivalent conjecture relating the non-primitive $p$-adic $L$-function and the non-primitive Selmer group. We avoid stating this conjecture here since we avoid mentioning the non primitive $p$-adic $L$-function in this paper. For a description of the non-primitive $p$-adic $L$-functions, the interested reader is referred to \cite{dasgupta2014factorization} (for the 4-dimensional representation $\rho_{\pmb{4},3}$) and  to \cite{MR939477} (for the 3-dimensional representation $\rho_{\pmb{3},2}$). The differences between the non-primitive $p$-adic $L$-functions and the non-primitive Selmer groups can be described in terms of Euler factors at primes $\nu \in \Sigma_0$. The divisors generated by these Euler factors are exactly equal to the divisors generated by the local factors at primes $\nu \in \Sigma_0$ that came up in Proposition \ref{primitive-non-primitive-difference}, while studying the differences between the primitive and the non-primitive Selmer groups. The purpose of Proposition \ref{divisor-local-factors} is to calculate the divisors attached to these local factors at all primes $\nu \in \Sigma_0$. The proof of Proposition \ref{divisor-local-factors} follows the proof of Proposition 2.4 in \cite{greenberg2000iwasawa}. There is one slight difference. In \cite{greenberg2000iwasawa}, the authors work with Galois groups over $\Q_\infty$, while we work with Galois groups over $\Q$. \\

Let $\RRR$ be an integrally closed local domain that is also a finite integral extension of $\Z_p[[u_1,\dotsc,u_n]]$. Let $\nu$ be a non-archimedean prime in $\Sigma_0$. Consider a continuous representation $\varrho : \Gal{\overline{\Q}_\nu}{\Q_\nu} \rightarrow \Gl_d(\mathcal{R})$. Let $\mathcal{L}$ denote the underlying free $\mathcal{R}$-module of rank $d$ on which $ \Gal{\overline{\Q}_\nu}{\Q_\nu}$ acts. Let $\D=\mathcal{L} \otimes_\RRR \hat{\mathcal{R}}$. Let $\mathcal{M}= H^1(I_\nu, \D)^\vee$. By Proposition 3.3 in \cite{MR2290593}, the $\RRR$-module $\mathcal{M}$ is torsion-free. Let $\mathcal{V}=\mathcal{M} \otimes_\RRR \mathcal{K}$, where $\mathcal{K}$ denotes the fraction field of $\mathcal{R}$. We shall let $m$ equal the dimension of $\mathcal{V}$ as a vector space over $\mathcal{K}$. If we let $\mathcal{W}$ equal $\mathcal{L} \otimes_\RRR \mathcal{K}$, local duality gives us the following isomorphism:
\begin{align*}
\mathcal{V} \cong H^0\bigg( I_\nu, \Hom\big(\mathcal{W}, \mathcal{K}(\chi_p)\big)\bigg).
\end{align*}

We denote the Frobenius at the prime $\nu$ by $\text{Frob}_\nu$. The Frobenius  $\text{Frob}_\nu$ is a topological generator for the Galois group $\Gamma_\nu = \Gal{\overline{\Q}_\nu}{\Q_\nu}/I_\nu$. The Frobenius $\text{Frob}_\nu$ acts on the $\mathcal{R}$-module $\mathcal{M}$ and hence on the vector space $\mathcal{V}$. Let $\mathcal{B}$ denote the $m \times m$ matrix that gives the action of $\text{Frob}_\nu$ on $\mathcal{V}$.  Using local duality and the fact that $\Gamma_\nu$ is pro-cyclic, we have the following isomorphisms:
\begin{align}\label{pro-cyc-iso}
\frac{\mathcal{M}}{(\text{Frob}_\nu-1)\mathcal{M}} \cong H^1(\Gamma_\nu, \mathcal{M}) \cong \bigg(H^1(I_\nu,\D)^{\Gamma_\nu}\bigg)^\vee.
\end{align}
Note that all of the $\RRR$-modules, appearing in the isomorphisms given in (\ref{pro-cyc-iso}), are isomorphic to $\Loc(\nu,\varrho)^\vee$. Consider the following hypothesis:
\begin{enumerate}[leftmargin=3cm, style=sameline, align=left, label=\textsc{1-EV}, ref=\textsc{1-EV}]
\item\label{1-ev} $1$ is not an eigenvalue for the matrix $\mathcal{B}$.
\end{enumerate}
Assuming hypothesis \ref{1-ev} holds, we have the following short exact sequence:
\begin{align*}
0 \rightarrow \mathcal{M} \xrightarrow {\text{Frob}_\nu - 1} \mathcal{M} \rightarrow H^1(\Gamma_\nu, \mathcal{M}) \rightarrow 0.
\end{align*}

If $\p$ is a height one prime ideal in $\mathcal{R}$, then $\mathcal{R}_\p$ is a DVR. Since the $\RRR$-module $\mathcal{M}$ is torsion-free, $\mathcal{M}_\p$ is a free $\mathcal{R}_\p$-module of rank $m$. Using these observations, we obtain the following proposition:

\begin{proposition} \label{divisor-local-factors}
If \ref{1-ev} holds, we have the following equality in the divisor group of $\mathcal{R}$:
\begin{align*}
\Div\bigg(H^1(\Gamma_\nu, \mathcal{M})\bigg) = \Div\bigg(H^0(\Gamma_\nu,H^1(I_\nu,\D))^\vee \bigg)= \Div \left(\prod_{i=1}^m (e_i - 1)\right),
\end{align*}
where the values $e_i$ correspond to the eigenvalues of the matrix $\mathcal{B}$.
\end{proposition}
Note that the hypothesis \ref{1-ev} holds if and only if $H^1(\Gamma_\nu,\mathcal{M})$ is $\RRR$-torsion. As a result of Proposition \ref{torsion-local-factors} (which follows from Corollary \ref{H1-torsion-local-factors}), the hypothesis \ref{1-ev} does hold for $\rho_{\pmb{4},3}$, $\pi \circ \rho_{\pmb{4},3}$, $\rho_{\pmb{3},2}$ and $\rho_{\pmb{1},2}$. This completes the study of the divisors associated to the local factors at $\nu \in \Sigma_0$.

\begin{proposition} \label{torsion-local-factors}
Let $\nu \in \Sigma_0$ be a prime. The $\T[[\Gamma]]$-module $\Loc(\nu,\rho_{\pmb{4},3})^\vee$ is torsion. The $R[[\Gamma]]$-modules $\Loc(\nu,\pi \circ \rho_{\pmb{4},3})^\vee$, $\Loc(\nu,\rho_{\pmb{3},2})^\vee$ and $\Loc(\nu,\rho_{\pmb{1},2})^\vee$ are also torsion.
\end{proposition}

\section{Theorem \ref{specialization-result}}

\subsection{Finiteness of the projective dimension of $\Sel_{\rho_{\pmb{4},3}}(\Q)^\vee$}

\begin{proposition} \label{Sel-Fin-Proj}
Suppose the Pontryagin dual of the non-primitive Selmer group $\Sel_{\rho_{\pmb{4},3}}(\Q)^\vee$ is $T[[\Gamma]]$-torsion. For every height two prime ideal $Q$ containing the height one prime ideal $\ker(\pi)$ in $T[[\Gamma]]$, the projective dimension of $\Sel_{\rho_{\pmb{4},3}}(\Q)^\vee \otimes_{T[[\Gamma]]} T[[\Gamma]]_Q$ as a $T[[\Gamma]]_Q$-module is finite.
\end{proposition}

\begin{proof}
Let $Q$ be a height two prime ideal in $T[[\Gamma]]$ containing the height one prime ideal $\ker(\pi)$. Let $S$ denote the multiplicative set $T[[\Gamma]] \setminus Q$. The global-to-local map $\phi^{\Sigma_0}_{\rho_{\pmb{4},3}}$ is surjective (see Proposition \ref{surjectivity-greenberg}) which gives us the following short exact sequence:
{\footnotesize\begin{align*}
0 \rightarrow \left(H^1\left(I_p,\frac{D_{\rho_{\pmb{4},3}}}{\Fil^+D_{\rho_{\pmb{4},3}}}\right)^{\Gamma_p}\right)^\vee\rightarrow H^1(G_\Sigma,D_{\rho_{\pmb{4},3}})^\vee \rightarrow \Sel_{\rho_{\pmb{4},3}}(\Q)^\vee \rightarrow 0.
\end{align*}}
To prove the proposition, it will be sufficient to show that the localizations of {\footnotesize\begin{align*}
H^1(G_\Sigma,D_{\rho_{\pmb{4},3}})^\vee, \qquad \left(H^1\left(I_p,\frac{D_{\rho_{\pmb{4},3}}}{\Fil^+D_{\rho_{\pmb{4},3}}}\right)^{\Gamma_p}\right)^\vee
\end{align*}}
at the multiplicative set $S$ have finite projective dimension over $S^{-1}T[[\Gamma]]$. \\

The $p$-cohomological dimension of $G_\Sigma$ is less than or equal to $2$ (Proposition 8.3.18 in \cite{neukirch2008cohomology}). Also $H^2(G_\Sigma, D_{\rho_{\pmb{4},3}})=0$ (Proposition \ref{coker-weak}). If $S^{-1}\left(H^0(G_\Sigma,D_{\rho_{\pmb{4},3}})^\vee\right)=0$, the projective dimension of  $S^{-1}\left(H^0(G_\Sigma,D_{\rho_{\pmb{4},3}})^\vee\right)$ over $S^{-1}T[[\Gamma]]$ is finite. By Proposition \ref{greenberg-cyc} and Lemma \ref{reg-2}, if $S^{-1}\left(H^0(G_\Sigma,D_{\rho_{\pmb{4},3}})^\vee\right)$ is not zero, then $S^{-1}T[[\Gamma]]$ is a regular local ring. Finitely generated modules over regular local rings have finite projective dimension. As a result, in this case too, the projective dimension of  $S^{-1}\left(H^0(G_\Sigma,D_{\rho_{\pmb{4},3}})^\vee\right)$ over $S^{-1}T[[\Gamma]]$ is finite. By Proposition \ref{H0criterion}, the projective dimension of $S^{-1}\left(H^1(G_\Sigma,D_{\rho_{\pmb{4},3}})^\vee\right)$ over $S^{-1}T[[\Gamma]]$ is~also~finite. \\

The $p$-cohomological dimension of $\Gal{\overline{\Q}_p}{\Q_p}$ is less than or equal to $2$. Let $W_i$ denote the $T[[\Gamma]]$-module $H^i\left(\Gal{\overline{\Q}_p}{\Q_p},\frac{D_{\rho_{\pmb{4},3}}}{\Fil^+D_{\rho_{\pmb{4},3}}}\right)^\vee$. Let $V_0$ denote  $H^0\left(I_p,\frac{D_{\rho_{\pmb{4},3}}}{\Fil^+D_{\rho_{\pmb{4},3}}}\right)^\vee$. We have $W_2=0$ (by Proposition \ref{h2-cohomology}). If $S^{-1}W_0=0$, then $S^{-1}W_0$ has finite projective dimension over $S^{-1}T[[\Gamma]]$. By Proposition \ref{greenberg-cyc} and Lemma \ref{reg-2}, if  $S^{-1}W_0$ is not zero, then $S^{-1}T[[\Gamma]]$ is a regular local ring. As a result in this case too, $S^{-1}W_0$ has finite projective dimension over $S^{-1}T[[\Gamma]]$. By Proposition \ref{H0criterion}, the projective dimension of  $S^{-1}W_1$  over $S^{-1}T[[\Gamma]]$ is finite.    Using a similar argument,  we can conclude that the projective dimension of $S^{-1}V_0$ over $S^{-1}T[[\Gamma]]$ is finite. We have the following short exact sequence:
\begin{align} \label{seq-inertia-proj}
0  \rightarrow V_0^{\Gamma_p} \rightarrow   V_0 \xrightarrow {\Frob_p-1} V_0 \rightarrow W_0 \rightarrow 0.
\end{align}
As a result of (\ref{seq-inertia-proj}), the projective dimension of $S^{-1}\left(V_0^{\Gamma_p}\right)$ over $S^{-1}T[[\Gamma]]$ is also finite. Considering the Pontryagin dual of the short exact sequence given by inflation-restriction, we get the following short exact sequence:
{\footnotesize \begin{align*}
0 \rightarrow \left(H^1\left(I_p,\frac{D_{\rho_{\pmb{4},3}}}{\Fil^+D_{\rho_{\pmb{4},3}}}\right)^{\Gamma_p}\right)^\vee  \rightarrow W_1 \rightarrow V_0^{\Gamma_p} \rightarrow 0.
\end{align*}
}
Our observations now let us conclude that the localization of  $\left(H^1\left(I_p,\frac{D_{\rho_{\pmb{4},3}}}{\Fil^+D_{\rho_{\pmb{4},3}}}\right)^{\Gamma_p}\right)^\vee$ at the multiplicative set $S$ also has finite projective dimension over $S^{-1}T[[\Gamma]]$. The proposition follows.
\end{proof}

\subsection{Control theorem for the non-primitive Selmer group $\Sel_{\rho_{\pmb{4},3}}(\Q)$}

To relate $\Sel_{\pi \circ \rho_{\pmb{4},3}}(\Q)^\vee$ to $\Sel_{\rho_{\pmb{4},3}}(\Q)^\vee \otimes_{T[[\Gamma]]} R[[\Gamma]]$, we can use Proposition \ref{control-theorem-selmer-groups}.

\begin{proposition} \label{control-theorem-selmer-group-dasgupta}
$\Sel_{\pi \circ \rho_{\pmb{4},3}}(\Q)^\vee$ is a torsion $R[[\Gamma]]$-module if and only if the height one prime ideal $\ker(\pi)$ in $T[[\Gamma]]$ does not belong to the support of $\Sel_{\rho_{\pmb{4},3}}(\Q)^\vee$. \\

Suppose $\Sel_{\pi \circ \rho_{\pmb{4},3}}(\Q)^\vee$ is a torsion $R[[\Gamma]]$-module. We have the following equality in the divisor group of $R[[\Gamma]]$:
\begin{align*}
\Div\bigg(\Sel_{\rho_{\pmb{4},3}}(\Q)^\vee \otimes_{T[[\Gamma]]}R[[\Gamma]]\bigg) + \Div\bigg(\Tor_1^{T[[\Gamma]]}\left(R[[\Gamma]], H^0(G_\Sigma,D_{\rho_{\pmb{4},3}})^\vee\right)\bigg)= \Div\bigg(\Sel_{\pi \circ \rho_{\pmb{4},3}}(\Q)^\vee\bigg).
\end{align*}
\end{proposition}

\begin{proof}

Let $$M = H^0(G_\Sigma,D_{\rho_{\pmb{4},3}})^\vee, \qquad N=H^0\left(I_p,\frac{D_{\rho_{\pmb{4},3}}}{\Fil^+D_{\rho_{\pmb{4},3}}}\right)^\vee.$$
To show that a finitely generated $R[[\Gamma]]$-module is torsion we will show that, as a $T[[\Gamma]]$-module, its localization at the prime ideal $\ker(\pi)$ is zero. For all $i \geq 1$, the $R[[\Gamma]]$-modules $\Tor_i^{T[[\Gamma]]}\left(R[[\Gamma]],M\right)$ and $\Tor_i\left(R[[\Gamma]],N\right)$ are torsion due the following isomorphisms:
\begin{align*}
&&\Tor_{i}^{T[[\Gamma]]}\left(R[[\Gamma]],M\right)_{\ker(\pi)} &\cong \Tor_{i}^{T[[\Gamma]]_{\ker(\pi)}}\left(R[[\Gamma]]_{\ker(\pi)},M_{\ker(\pi)}\right) =0. \\
&& \Tor_{i}^{T[[\Gamma]]}\left(R[[\Gamma]],N\right)_{\ker(\pi)} & \cong \Tor_{i}^{T[[\Gamma]]_{\ker(\pi)}}\left(R[[\Gamma]]_{\ker(\pi)},N_{\ker(\pi)}\right) =0.
\end{align*}
To obtain the above isomorphisms, we have made use of Proposition \ref{greenberg-cyc} which provides a monic polynomial $h(s)$ in $T[s]$ with the property that $h(\gamma_0)$ annihilates $M$  and $N$. This element $h(\gamma_0)$ is not in the kernel of the map $\pi: T[[\Gamma]] \rightarrow R[[\Gamma]]$ and consequently $$ M_{\ker(\pi)} = N_{\ker(\pi)} =0.$$
By Proposition \ref{control-theorem-selmer-groups}, $\Sel_{\pi \circ \rho_{\pmb{4},3}}(\Q)^\vee$ is a torsion $R[[\Gamma]]$-module if and only if the height one prime ideal $\ker(\pi)$ in $T[[\Gamma]]$ does not belong to the support of $\Sel_{\rho_{\pmb{4},3}}(\Q)^\vee$. Let us now verify the remaining hypotheses given in Proposition \ref{control-theorem-selmer-groups}.  Any height two prime ideal $Q$ containing $\ker(\pi)$ and in the support of $M$ or $N$ would also have to contain $h(\gamma_0)$. By Lemma \ref{reg-2}, for such a prime ideal $Q$, the local ring $T[[\Gamma]]_Q$ would have to be regular. By Proposition \ref{surjectivity-greenberg}, the global-to-local map defining the non-primitive Selmer group $\Sel_{\pi \circ \rho_{\pmb{4},3}}(\Q)$ is surjective. These observations verify all the remaining hypotheses given in Proposition \ref{control-theorem-selmer-groups}. Since we will need one of the observations later, we state it separately as a lemma.

\begin{lemma} \label{lemma-in-the-middle}
If $Q$ is a height two prime ideal in $T[[\Gamma]]$ containing $\ker(\pi)$ and in the support of the $T[[\Gamma]]$-module $N$, then $T[[\Gamma]]_Q$ is a regular local ring.
\end{lemma}

To complete the proof of this Proposition using Proposition \ref{control-theorem-selmer-groups}, it only remains to show the following equality in the divisor group of $R[[\Gamma]]$:
$$\Div\bigg(\Tor^{T[[\Gamma]]}_1\left(R[[\Gamma]],N\right)_{\Gamma_p}\bigg)=0.$$
In fact, we will prove that $\Tor^{T[[\Gamma]]}_1\left(R[[\Gamma]],N\right)$ is a pseudo-null $R[[\Gamma]]$-module. To do so, let us consider a height two prime ideal $Q$ in $T[[\Gamma]]$ containing $\ker(\pi)$.  Let us assume for the sake of contradiction that $Q$ belongs to the support of $\Tor^{T[[\Gamma]]}_1\left(R[[\Gamma]],N\right)$. It must then belong to the support of $N$ too because we have the following isomorphism:
\begin{align*}
\Tor^{T[[\Gamma]]}_1\left(R[[\Gamma]],N\right)_Q \cong \Tor^{T[[\Gamma]]_Q}_1\left(R[[\Gamma]]_Q,N_Q\right).
\end{align*}

Note that this also forces $h(\gamma_0)$ to belong to $Q$. As a result, $Q \cap T$ must equal the height one prime ideal given by the kernel of the map $\pi_{F,F} : T \rightarrow R$.  The extension $T_{\ker(\pi_{F,F})} \rightarrow T[[\Gamma]]_Q$ is flat. By flat base-change theorems for Tor (Proposition 3.2.9 in \cite{weibel1995introduction}), we have the following isomorphisms:
{\small \begin{align}\label{base-change-control-p}
\Tor^{T[[\Gamma]]_Q}_1\left(R[[\Gamma]]_Q,N_Q\right) & \cong \Tor^{T[[\Gamma]]_Q}_1\left(\frac{T[[\Gamma]]_Q}{\ker(\pi)_Q},N_Q\right)
\\ \notag & \cong \Tor^{T_{\ker(\pi_{F,F})} \otimes_{T_{\ker(\pi_{F,F})} }T[[\Gamma]]_Q}_1\left(\frac{T_{\ker(\pi_{F,F})}}{\ker(\pi_{F,F})} \otimes_{T_{\ker(\pi_{F,F})}} T[[\Gamma]]_Q,N_Q\right) \\ \notag & \cong \Tor^{T_{\ker(\pi_{F,F})}}_1\left(\frac{T_{\ker(\pi_{F,F})}}{\ker(\pi_{F,F})},N_Q\right).
\end{align}}
The modules in equation (\ref{base-change-control-p}) above are zero if $N_Q$ is a flat $T_{\ker(\pi_{F,F})}$-module. To obtain our contradiction, and hence complete the proof of this proposition, it suffices to show that $N_Q$ is a flat $T_{\ker(\pi_{F,F})}$-module.
\begin{lemma} \label{lemma-in-the-middle-of-proof}
$N_Q$ is a flat $T_{\ker(\pi_{F,F})}$-module.
\end{lemma}
By $N_Q$, we mean the localization of $N$ at the multiplicative set $T[[\Gamma]] \setminus Q$. By $N_{\ker(\pi_{F,F})}$, we will mean the localization of $N$ at the multiplicative set $T \setminus \ker(\pi_{F,F})$. We have an inclusion of sets $T \setminus \ker(\pi_{F,F}) \subset T[[\Gamma]] \setminus Q$. If we show that  $N_{\ker(\pi_{F,F})}$ is a flat $T_{\ker(\pi_{F,F})}$-module, then $N_Q$ will turn out to be a flat $T_{\ker(\pi_{F,F})}$-module as well.  We shall show that $N_{\ker(\pi_{F,F})}$ is a free $T_{ \ker(\pi_{F,F})}$-module in the remaining part of the proof; this will complete the proof of the Proposition.

\begin{lemma}
$N_{\ker(\pi_{F,F})}$ is a free $T_{ \ker(\pi_{F,F})}$-module
\end{lemma}

Let $\eta_p$ be the unique prime above $p$ in $\Q_\infty$ and $I_{\eta_p}$ be the corresponding inertia subgroup. Proposition \ref{greenberg-cyc} gives us the following isomorphism of $T[[\Gamma]]$-modules:
{\small \begin{align*}
N = H^0\left(I_p,\frac{D_{ \rho_{\pmb{4},3}}}{\Fil^+D_{ \rho_{\pmb{4},3}}}\right)^\vee \cong H^0\left(I_{\eta_p},\frac{D_{  \rho_{F,F}(\chi)}}{\Fil^+D_{ \rho_{F,F}(\chi)}}\right)^\vee.
\end{align*}}
Note that the discrete modules $D_{  \rho_{F,F}(\chi)}$ and $\Fil^+D_{ \rho_{F,F}(\chi)}$, associated to the Galois representation $\rho_{F,F}(\chi)$, are defined just as the discrete modules $D_{ \rho_{\pmb{4},3}}$ and $\Fil^+D_{ \rho_{\pmb{4},3}}$, associated to the Galois representation $\rho_{\pmb{4},3}$, are defined. The description of these discrete modules given in Section \ref{section2-dasgupta-factorization}, along with the observation that the restrictions to $I_{\eta_p}$ of the characters $\kappa $ and $\epsilon_F$ are trivial, let us obtain the following $I_{\eta_p}$~-~equivariant isomorphism of $T$-modules:
{\small \begin{align*}
 \left(\frac{D_{\rho_{F,F}(\chi)}}{\Fil^+D_{ \rho_{F,F}(\chi)}}\right)^\vee \cong L_F \otimes_{i_1} T (\chi^{-1}),
 \end{align*}}
Using Pontryagin duality (Theorem 2.6.9 in \cite{neukirch2008cohomology}), we get the following isomorphism of $T$-modules:
 \begin{align*}
 N  \cong H_0\left(I_{\eta_p}, L_F \otimes_{i_1} T (\chi^{-1})\right).
\end{align*}
Here $H_0(I_{\eta_p}, L_F \otimes_{i_1} T (\chi^{-1}))$ denotes the maximal quotient of $L_F \otimes_{i_1} T(\chi^{-1})$ on which $I_{\eta_p}$ acts trivially. For the free $T$-module $L_F \otimes_{i_1} T (\chi^{-1})$, one can choose a basis $\{e_1,e_2\}$ such that the action of every element $g$ in $I_{\eta_p}$ is given by the $2\times 2$ matrix {\footnotesize $\left(\begin{array}{cc} \chi^{-1}(g) \det(\rho_F(g)) & d_g \\ 0 & \chi^{-1}(g)\end{array}\right)$}. Here, $d_g$ is an element of $i_1(R)$ for each element $g$ in $I_{\eta_p}$.
Observe that by (\ref{determinant-properties}), the restriction of the character $\det(\rho_{F})$ to $I_{\eta_p}$ is finite order (recall that $\eta_p$ is a prime above $p$ in the cyclotomic $\Z_p$ extension $\Q_\infty$). Observe also that the character $\chi$ is of finite order. So there exists a closed subgroup $\mathcal{I}$ inside $I_{\eta_p}$ of finite index such that for all $g \in \mathcal{I}$, the action of $g$ on $L_F \otimes_{i_1} T (\chi^{-1})$ (with respect to the basis $\{e_1,e_2\}$) is given the $2 \times 2$ matrix $\left( \begin{array}{cc} 1 & d_g \\ 0 & 1 \end{array}\right)$.

Note that $p\notin \ker(\pi_{F,F})$ and hence $p$ is invertible in $T_{\ker(\pi_{F,F})}$. The index $[I_{\eta_p}:\mathcal{I}]$ is invertible in $T_{\ker(\pi_{F,F})}$. By Lemma \ref{interesting-profinite-group-theory} to show that  $N_{\ker(\pi_{F,F})}$ is a free $T_{\ker(\pi_{F,F})}$-module, it is enough to show that the localization of  $H_0(\mathcal{I},L_F \otimes_{i_1}T(\chi^{-1}))$ at the prime ideal $\ker(\pi_{F,F})$ is a free $T_{\ker(\pi_{F,F})}$-module. If the action of $\mathcal{I}$ on $L_F$ is decomposable, then the localization of  $H_0(\mathcal{I},L_F \otimes_{i_1}T(\chi^{-1}))$ at the prime ideal $\ker(\pi_{F,F})$ is a free $T_{\ker(\pi_{F,F})}$-module of rank $2$. Otherwise, the localization of $H_0(\mathcal{I},L_F \otimes_{i_1}T(\chi^{-1}))$ at the prime ideal $\ker(\pi_{F,F})$ is a free $T_{\ker(\pi_{F,F})}$-module of rank $1$. This previous statement used the fact that every non-zero element of $i_1(R)$ is invertible in $T_{\ker(\pi_{F,F})}$. The proposition follows.
\end{proof}

We shall now use the various lemmas stated in the proof of Proposition \ref{control-theorem-selmer-group-dasgupta} to deduce results concerning the pseudo-null submodules of $\Sel_{\rho_{\pmb{4},3}}(\Q)^\vee$. Let us first consider the following implication (which uses Pontryagin duality and the fact that the group $\Gamma_p$ is topologically generated by $\Frob_p$):
{\small\begin{align*}
N = H^0\left(I_p,\frac{D_{\rho_{\pmb{4},3}}}{\Fil^+D_{\rho_{\pmb{4},3}}}\right)^\vee  \implies N [\Frob_p-1] \cong  H^1\left(\Gamma_p,H^0\left(I_p,\frac{D_{\rho_{\pmb{4},3}}}{\Fil^+D_{\rho_{\pmb{4},3}}}\right)\right)^\vee.
\end{align*}}
Let $Q$ be a height two prime ideal in the support of $N[\Frob_p-1]$. We shall now show that the projective dimension of $\left(N[\Frob_p-1]\right)_Q$ over $T[[\Gamma]]_Q$ is less than or equal to one. Lemma \ref{lemma-in-the-middle-of-proof} shows us that $N_Q$ is a flat $T_{\ker(\pi_{F,F})}$-module. Consequently, it is also torsion-free over $T_{\ker(\pi_{F,F})}$. Since $T_{\ker(\pi_{F,F})}$ is integrally closed, the $1$-dimensional local ring $T_{\ker(\pi_{F,F})}$ is a discrete valuation ring. We have $\left(N [\Frob_p-1]\right)_Q \subset N_Q$. So, the $T_{\ker(\pi_{F,F})}$-module $\left(N[\Frob_p-1]\right)_Q$ is also torsion-free. In particular, none of the non-zero elements of $\left(N[\Frob_p-1]\right)_Q$ are annihilated by $\ker(\pi_{F,F})$. Thus, the maximal ideal generated by $Q$ (which contains $\ker(\pi_{F,F})$) inside the $2$-dimensional local ring $T[[\Gamma]]_Q$ cannot be an associated prime ideal for the module $\left(N [\Frob_p-1]\right)_Q$. As a result, $$ \depth_{T[[\Gamma]]_Q}\left(N[\Frob_p-1]\right)_Q \geq 1.$$

Lemma \ref{lemma-in-the-middle} tells us that since $Q$ is in the support of $N$ (as it is in the support of $N[\Frob_p-1]$), the local ring $T[[\Gamma]]_Q$ is regular. We can use the Auslander-Buchsbaum formula over the regular local ring $T[[\Gamma]]_Q$.
\begin{align*}
\text{pd}_{T[[\Gamma]]_Q}\left(N[\Frob_p-1]\right)_Q + \depth_{T[[\Gamma]]_Q}\left(N[\Frob_p-1]\right)_Q &= \depth_{T[[\Gamma]]_Q} T[[\Gamma]]_Q = 2. \\
\implies \text{pd}_{T[[\Gamma]]_Q}\left(N[\Frob_p-1]\right)_Q \leq 1.
\end{align*}

Now let us suppose that $\Sel_{\rho_{\pmb{4},3}}(\Q)^\vee$ is a torsion $T[[\Gamma]]$-module. What we have shown is that if the height two prime ideal $Q$ containing $\ker(\pi)$ is in the support of $N[\Frob_p-1]$, then the local ring $T[[\Gamma]]_Q$ is regular. Our arguments have also shown that in this case, we have $\pd_{T[[\Gamma]]_Q} \left(N[\Frob_p-1]\right)_Q\leq 1$. See Exercise 4.1.2 in \cite{weibel1995introduction} on how projective dimensions behave in short exact sequences.  Suppose $\Sel_{\rho_{\pmb{4},3}}(\Q)^\vee$ is a torsion $T[[\Gamma]]$-module. Note that the Galois representation $\rho_{\pmb{4},3}$ satisfies the (\ref{p-critical}) hypothesis. We have
\begin{flalign*}
\pd_{T[[\Gamma]]_Q} \left(N[\Frob_p-1]\right)_Q\leq 1, \quad  &   \underbrace{\pd_{T[[\Gamma]]_Q}\left(\S^{\Sigma_0,str}_{\rho_{\pmb{4},3}}(\Q)^\vee\right)_Q \leq 1}_{\text{by Proposition \ref{reg-dim-2-no-pn} and Proposition \ref{strict-pseudo}}}  \\
\implies  \underbrace{\pd_{T[[\Gamma]]_Q} \left(\Sel_{\rho_{\pmb{4},3}}(\Q)^\vee\right)_Q\leq 1}_{ \text{ by Proposition \ref{strict-difference}}} &
\end{flalign*} Now using Proposition \ref{reg-dim-2-no-pn}, we can conclude that the $T[[\Gamma]]_Q$-module $\left(\Sel_{\rho_{\pmb{4},3}}(\Q)^\vee\right)_Q$  has no non-trivial pseudo-null submodules. \\

If the height two prime ideal $Q$ containing $\ker(\pi)$ is not in the support of $N[\Frob_p-1]$, we have the following isomorphism (by Proposition \ref{strict-difference}):
\begin{align*}
\left(\Sel_{\rho_{\pmb{4},3}}(\Q)^\vee \right)_Q \cong \left(\S^{\Sigma_0,str}_{\rho_{\pmb{4},3}}(\Q)^\vee\right)_Q.
\end{align*}
In this case too, by Proposition \ref{strict-pseudo}, the $T[[\Gamma]]_Q$-module $\left(\Sel_{\rho_{\pmb{4},3}}(\Q)^\vee\right)_Q$ has no non-trivial pseudo-null submodules.  We have proved the following proposition:

\begin{proposition}\label{Sel-No-PN}
Suppose $\Sel_{\rho_{\pmb{4},3}}(\Q)^\vee$ is a torsion $T[[\Gamma]]$-module. For every height two prime ideal $Q$ containing $\ker(\pi)$, the $T[[\Gamma]]_Q$-module $\left(\Sel_{\rho_{\pmb{4},3}}(\Q)^\vee\right)_Q$~has~no~non-trivial~pseudo-null~submodules.
\end{proposition}

\subsection{Control theorem for local factors away from $p$}

\begin{proposition} \label{local-control}
Let $\nu \in \Sigma_0$. We have the following equality in the~divisor~group~of~$R[[\Gamma]]$:
\begin{align*}
\Div\left( \Loc(\nu,\pi \circ \rho_{\pmb{4},3})^\vee \right) = \Div\left( \Loc(\nu,\rho_{\pmb{4},3})^\vee \otimes_{T[[\Gamma]]} R[[\Gamma]]\right).
\end{align*}
\end{proposition}
\begin{proof}
Let $\eta_\nu$ be a prime above $\nu$ in $\Q_\infty$. Note that $\Gamma$ equals $\Gal{\Q_\infty}{\Q}$. Let $G_{\eta_\nu}$ be a decomposition group inside $\Gal{\overline{\Q}}{\Q_\infty}$ corresponding to the prime $\eta_\nu$. Also, we let $\Delta_{\eta_\nu}$ denote the decomposition group inside $\Gal{\Q_\infty}{\Q}$ corresponding to the prime $\eta_\nu$ lying above $\nu$. Note that $\Delta_{\eta_\nu}$ is isomorphic to $\Z_p$ as a topological group and that the index $[\Gamma : \Delta_{\eta_\nu}]$ is finite. We have the following isomorphisms due to Proposition \ref{local-cohomology-not-p}:
\begin{align*}
\Loc(\nu,\pi \circ \rho_{\pmb{4},3}) \cong \Ind^\Gamma_{\Delta_{\eta_\nu}} H^1(G_{\eta_\nu},D_{\pi_{F,F} \circ \rho_{F,F}(\chi)}), \qquad \Loc(\nu, \rho_{\pmb{4},3}) \cong \Ind^\Gamma_{\Delta_{\eta_\nu}} H^1(G_{\eta_\nu},D_{\rho_{F,F}(\chi)}).
\end{align*}
Furthermore, we have a natural map
\begin{align*}
\alpha: H^1\left(G_{\eta_\nu},D_{\rho_{F,F}(\chi)}\right)^\vee \otimes_{T[[\Gamma]]} R[[\Gamma]] \rightarrow H^1\left(G_{\eta_\nu},D_{\pi_{F,F} \circ \rho_{F,F}(\chi)}\right)^\vee,
\end{align*}
that gives us the following map
\begin{align*}
\tilde{\alpha} : \underbrace{\Ind_{\Delta_{\eta_\nu}}^{\Gamma} H^1\left(G_{\eta_\nu},D_{\rho_{F,F}(\chi)}\right)^\vee \otimes_{T[[\Gamma]]} R[[\Gamma]]}_{\Loc(\nu, \rho_{\pmb{4},3})^\vee \otimes_{T[[\Gamma]]}R[[\Gamma]]} \rightarrow \underbrace{\Ind_{\Delta_{\eta_\nu}}^{\Gamma} H^1\left(G_{\eta_\nu},D_{\pi_{F,F} \circ \rho_{F,F}(\chi)}\right)^\vee}_{\Loc(\nu, \pi \circ \rho_{\pmb{4},3})^\vee}.
\end{align*}
To prove the proposition, we will show that the $R[[\Gamma]]$-modules $\ker(\tilde{\alpha})$ and $\coker(\tilde{\alpha})$ are pseudo-null. Since $\Delta_{\eta_\nu}$ is of finite index in $\Gamma$, the ring $R[[\Delta_{\eta_\nu}]]$ is a finite integral extension of $R[[\Gamma]]$. These observations tell us that to prove the proposition, it will be sufficient to prove that $\ker(\alpha)$ and $\coker(\alpha)$ are pseudo-null $R[[\Delta_{\eta_\nu}]]$-modules. From Corollary \ref{H1-torsion-local-factors}, we can deduce that there exists a monic polynomial $h(s)$ in $R[s]$ such that $h(\gamma_{\eta_\nu})$ annihilates both $H^1\left(G_{\eta_\nu},D_{\rho_{F,F}(\chi)}\right)^\vee \otimes_{T[[\Gamma]]} R[[\Gamma]] $ and $H^1\left(G_{\eta_\nu},D_{\pi_{F,F} \circ \rho_{F,F}(\chi)}\right)^\vee$. Hence, $h(\gamma_{\eta_\nu})$ annihilates both $\ker(\alpha)$ and $\coker(\alpha)$ too. Here, $\gamma_{\eta_\nu}$ is some chosen topological generator for the pro-cyclic group $\Delta_{\eta_\nu}$. A prime ideal $Q$ in $R[[\Delta_{\eta_\nu}]]$ that contains a non-zero element of $R$ and $h(\gamma_{\eta_\nu})$ has height at least two. So, to prove the proposition, it will now be sufficient to prove that there exists a non-zero element of $R$ that annihilates both $\ker(\alpha)$ and $\coker(\alpha)$. From Corollary \ref{H1-torsion-local-factors}, one can easily deduce that both $\ker(\alpha)$ and $\coker(\alpha)$ are finitely generated $R$-modules. It will thus be sufficient to prove that both $\ker(\alpha)$ and $\coker(\alpha)$ are torsion $R$-modules. Let $S$ denote the multiplicative set $T \setminus \ker(\pi_{F,F})$. Recall that the map  $\pi_{F,F} : T \rightarrow R$ was obtained by sending an elementary tensor $x \otimes y$ in $T$ to $xy$. Henceforth, we shall consider $\ker(\alpha)$ and $\coker(\alpha)$ as modules over $T$ instead; it will suffice to show that their localizations at the multiplicative set $S$ equals zero. Also since $R[[\Gamma]] \cong T[[\Gamma]] \otimes_T R$, we have the natural isomorphism of $R$-modules:
\begin{align*} H^1\left(G_{\eta_\nu},D_{\rho_{F,F}(\chi)}\right)^\vee \otimes_T R \cong H^1\left(G_{\eta_\nu},D_{\rho_{F,F}(\chi)}\right)^\vee \otimes_{T[[\Gamma]]} R[[\Gamma]]
\end{align*}
The $T$-module $L_{F,F}$ is free. We have the following exact sequence due to Proposition \ref{control-theorem}:
{\small \begin{align*}
\Tor_2^{T}\left(R,\ H^0(G_{\eta_\nu},D_{\rho_{F,F}(\chi)})^\vee \right) & \rightarrow  H^1\left(G_{\eta_\nu},D_{\rho_{F,F}(\chi)}\right)^\vee \otimes_T R \xrightarrow {\alpha} \\ & \xrightarrow {\alpha} H^1\left(G_{\eta_\nu},D_{\pi_{F,F} \circ \rho_{F,F}(\chi)}\right)^\vee  \rightarrow  \Tor_1^{T}\left(R,H^0(G_{\eta_\nu},D_{\rho_{F,F}(\chi)})^\vee\right) \rightarrow 0.
\end{align*}
}
The proposition would follow if we show that the localization of $ \Tor_i^{T}\left(R,H^0(G_{\eta_\nu},D_{\rho_{F,F}})^\vee\right)$ at the multiplicative set $S$ equals zero, for each $i \geq 1$.  Since localization commutes with Tor (Proposition 3.2.9 in \cite{weibel1995introduction}), it will be sufficient to show that the localization of $H^0(G_{\eta_\nu},D_{\rho_{F,F}})^\vee$ at the multiplicative set $S$ is a free $S^{-1}T$-module. \\

Note that the extension $\Q_\infty/\Q$ is unramified at $\nu$. So, $I_\nu \subset G_{\eta_\nu}$. What we will now show is that the localization of $H^0(I_\nu,D_{\rho_{F,F}(\chi)})^\vee$ at the multiplicative set $S$ is a free $S^{-1}T$-module. This is sufficient for our purposes. To see why, let us suppose that the localization of  $H^0(I_\nu,D_{\rho_{F,F}(\chi)})^\vee$ at the multiplicative set $S$ is a free $S^{-1}T$-module. The group $G_{\eta_\nu}/I_\nu$, which is isomorphic to $\Gal{\Q_\nu^{ur}}{\Q_{\nu,\infty}}$, is of profinite order prime to $p$. Here, $\Q_{\nu,\infty}$ denotes the cyclotomic $\Z_p$ extension of $\Q_\nu$.  The group of $T$-linear automorphisms $\Aut_T\left(H^0(I_\nu,D_{\rho_{F,F}(\chi)})^\vee\right)$ has an open pro-$p$ subgroup (see Lemma 4.5.5 in \cite{ribes2000profinite}). Also the image of  $\Gal{\Q_\nu^{ur}}{\Q_{\nu,\infty}}$ inside the automorphism group $\Aut_T\bigg(S^{-1}\left(H^0(I_\nu,D_{\rho_{F,F}(\chi)})^\vee\right)\bigg)$ factors through its image inside the automorphism group $\Aut_T\left(H^0(I_\nu,D_{\rho_{F,F}(\chi)})^\vee \right)$. These observations
indicate that the action of $\Gal{\Q_\nu^{ur}}{\Q_{\nu,\infty}}$ on $S^{-1}\left(H^0(I_\nu,D_{\pi_{F,F}(\chi)})^\vee\right)$ factors through a finite group (say $\digamma$) that is of order prime to $p$. Note that one can use Lemma 3.2.2 in \cite{Webb_2016} (which one can view as an analog of Maschke's theorem over group rings, where the order of the finite group is invertible). This finally lets us conclude that whenever the localization of  $H^0(I_\nu,D_{\rho_{F,F}(\chi)})^\vee$ at the multiplicative set $S$ is a free $S^{-1}T$-module, then so is the localization of  $H^0(G_{\eta_\nu},D_{\rho_{F,F}(\chi)})^\vee$ at the multiplicative set $S$. It thus remains to show that the localization of  $H^0(I_\nu,D_{\rho_{F,F}(\chi)})^\vee$ at the multiplicative set $S$ is a free $S^{-1}T$-module. \\

Observe also that $p$ is invertible in the localization $S^{-1}T$ since $p \notin \ker(\pi_{F,F})$.  In fact, for any finite group $\digamma'$, the order of $\digamma'$ is invertible in $S^{-1}T$. So, one can still use Lemma 3.2.2 in \cite{Webb_2016}. Our arguments above can be modified to show that whenever $I_\nu$ acts on $L_F$ via a finite group (and hence on $D^\vee_{\rho_{F,F}(\chi)}$ too), the localization of  $H^0(I_\nu, D_{\rho_{F,F}(\chi)})^\vee$  at the multiplicative set $S$ is a free $S^{-1}T$-module. We shall thus assume that $I_\nu$ acts on $L_F$ via an infinite group, say $\mathcal{I}$. We keep the following diagram in mind:
\begin{center}
\begin{tikzpicture}[node distance = 1.2cm, auto]
      \node (Qnu) {$\Q_\nu$};
      \node (Qnuur) [above of=Qnu] {$\Q_{\nu}^{ur}$};
      \node (Qnuaut) [ above of=Qnuur] {$\overline{\Q}_\nu^{\ker(\rho_F\mid_{I_\nu})}$ };
       \node (GQ) [right of = Qnuaut, node distance=7cm] {$\mathcal{I} \cong  \Delta_0 \rtimes P_\nu$, \qquad $P_\nu$ denotes the $p$-sylow subgroup of $\mathcal{I}$};
    \node (tildegamma) [ right of = Kinfty, node distance=7cm]       {};
     \node (gamma) [ right of = Qnuur, node distance=7cm]       {By local class field theory, $P_\nu \cong \Z_p$};
\node (characters) [right of = Q, node distance = 7cm] {$\Delta_0$ : finite and of order prime to $p$};
            \draw[-] (Qnu) to node {$\Gamma_{\nu}$} (Qnuur);
      \draw[-] (Qnuur) to node  {$\mathcal{I}$} (Qnuaut);
    \end{tikzpicture}
  \end{center}
Let $\Upsilon_\nu$ be a topological generator for $P_\nu$. We will first argue that the two eigenvalues given by the action of $\Upsilon_\nu$ on $L_F$ are the same.  For the sake of contradiction, suppose that the eigenvalues were different (say $a$ and $b$). Working over a quadratic extension $K$ of $\Frac(R)$, if necessary, we shall assume that we can find a basis $\{e_1,e_2\}$ over $K$ such that the action of $\Upsilon_\nu$ is given by a diagonal matrix (with diagonal entries $a$ and $b$). Let $\tilde{\Frob}_\nu$ be a lift in $\Gal{\overline{\Q}_\nu^{\ker(\rho_F\mid_{I_\nu})}}{\Q_\nu}$ of $\Frob_\nu$. If we let $l_\nu$ denote the  characteristic of the residue field, local class field theory provides us the following equality, for some element $\delta_0$ in $\Delta_0$:
\begin{align*}
\delta_0 \Upsilon_\nu^{l_\nu} \delta_0^{-1}= \tilde{\Frob_\nu} \Upsilon_\nu \tilde{\Frob_\nu}^{-1}
\end{align*}
The element $\rho_F(\delta_0^{-1} \tilde{\Frob}_\nu)$ belongs to the normalizer of the group of invertible diagonal matrices and hence must act on the set $\{e_1,e_2\}$.  It cannot act trivially on the set $\{e_1,e_2\}$, for otherwise $a$ and $b$ would be roots of unity (since we would have $a=a^{l_\nu}$ and $b=b^{l_\nu}$) and the action of $I_\nu$ on $L_F$ would then have to factor through a finite group. If  the element $\rho_F(\delta_0^{-1} \tilde{\Frob}_\nu)$ permutes the elements of set $\{e_1,e_2\}$, then $a^{l_\nu}=b$ and $b^{l_\nu}=a$. This once again forces $a$ and $b$ to be (two distinct) roots of unity. This also contradicts the fact that the action of $I_\nu$ on $L_F$ factors through an infinite group. What we have thus shown is that the two eigenvalues for the action of $\Upsilon_\nu$ on $L_F$ are equal (to $c$, say), if the action of $I_\nu$ on $L_F$ factors through an infinite group. \\

We will choose a basis $\{e_1,e_2\}$ over $\Frac(R)$ so that the action of $\Upsilon_\nu$ on $L_F \otimes_R \Frac(R)$ can be given by an upper triangle matrix (whose diagonal entries are both equal to $c$). We have the following isomorphism of free $T$-modules (of rank $4$) that is $I_\nu$-equivariant:
$$ D_{\rho_{F,F}(\chi)}^\vee \cong \Hom_{T}\left(L_F \otimes_{i_2} T, L_F \otimes_{i_1} T \right)(\chi^{-1}).$$
Recall that $T$ is an integral extension of $\O[[x_1,x_2]]$, where $x_1$, $x_2$ denote the weight variables. Note that $\ker(\pi_{F,F}) \cap \O[[x_1]]=\{0\}$ and $\ker(\pi_{F,F}) \cap \O[[x_2]]=\{0\}$.  So, the fraction fields of $i_1(R)$ and $i_2(R)$ should be contained inside the multiplicative set $S$. This allows us to choose a basis $\{\sigma_{e_1,e_1}, \sigma_{e_2,e_1}, \sigma_{e_1,e_2}, \sigma_{e_2,e_2}\}$ for $S^{-1}\left(D_{\rho_{F,F}}^\vee\right)$ over $S^{-1}T$, where the elements $\sigma_{e_1,e_1}$, $\sigma_{e_2,e_1}$, $\sigma_{e_1,e_2}$  and $\sigma_{e_2,e_2}$ in $\Hom_{T}\left(L_F \otimes_{i_2} T, L_F \otimes_{i_1} T \right)$ are described below.
\begin{align*}
\sigma_{e_1,e_1} : & i_2(e_1) \rightarrow i_1(e_1),   \quad \  \sigma_{e_2,e_1} : && i_2(e_1) \rightarrow 0, \quad \ \sigma_{e_1,e_2} : && i_2(e_1) \rightarrow i_1(e_2), \quad \ \sigma_{e_2,e_2}:&& i_2(e_1) \rightarrow 0, \\
& i_2(e_2) \rightarrow 0. \quad \ && i_2(e_2) \rightarrow i_1(e_1). \quad \ && i_2(e_2) \rightarrow 0. \quad \ && i_2(e_2) \rightarrow i_1(e_2).
\end{align*}
Note that since the action of $I_\nu$ on $L_F$ factors through an infinite group, the action of $\Upsilon_\nu$ on $L_F \otimes_R \Frac(R)$ is reducible but indecomposable. Since $\det(\rho_F\mid_{I_\nu})$ is finite, the eigenvalue $c$ has to be a root of unity. Also, $\chi$ is of finite order. This allows us to find a positive integer $n$ such that the action of $\Upsilon_\nu^n$ on $L_F \otimes_R \Frac(R)$ (with respect to the basis $\{e_1,e_2\}$) is given by the $2 \times 2$ matrix $\left(\begin{array}{cc} 1 & d \\ 0 & 1\end{array}\right)$, for some non-zero element $d$ in $\Frac(R)$; and such that $\chi(\Upsilon_\nu^n)=1$.  The actions of $\Upsilon_\nu^n$ and $\Upsilon^n_\nu - \text{Id}$ on  $S^{-1}\left(D_{\rho_{F,F}(\chi)}^\vee\right)$ (with respect to the basis described above) are then given by the $4 \times 4$ matrices
{\footnotesize
\begin{align*}
\left(\begin{array}{cccc}
1 & 0 & i_1(d)& 0 \\
-i_2(d) & 1 & -i_1(d)i_2(d) &i_1(d)\\
0 & 0 & 1 &0 \\
0 & 0 &  -i_2(d) & 1
\end{array}\right), \quad
\left(\begin{array}{cccc}
0 & 0 & i_1(d)& 0 \\
-i_2(d) & 0 & -i_1(d)i_2(d) &i_1(d)\\
0 & 0 & 0 &0 \\
0 & 0 &  -i_2(d) & 0
\end{array}\right)
\end{align*}
}
respectively. The elements $i_1(d)$, $i_2(d)$ and $i_1(d)i_2(d)$ belong to the multiplicative set $S$, hence they are invertible in $S^{-1}T$. This lets us conclude that the localization of  $H^0(P_\nu^n,D_{\rho_{F,F}(\chi)})^\vee$ at $S$ is a free $S^{-1}T$-module of rank $2$. By Lemma \ref{interesting-profinite-group-theory}, the localization of  $H^0(I_\nu,D_{\rho_{F,F}(\chi)})^\vee$ at the multiplicative set $S$ is also a free $S^{-1}T$-module. The proposition follows.

\end{proof}

\subsection{Proof of Theorem \ref{specialization-result}}

The proof follows entirely by recalling results from earlier sections. Proposition \ref{control-theorem-selmer-group-dasgupta} tells us that $\ker(\pi)$ does not belong to the support of the $T[[\Gamma]]$-module $\Sel_{\rho_{\pmb{4},3}}(\Q)^\vee$ if and only if the $R[[\Gamma]]$-module $\Sel_{\pi \circ \rho_{\pmb{4},3}}(\Q)^\vee$ is torsion. Let us now assume that the $R[[\Gamma]]$-module $\Sel_{\pi \circ \rho_{\pmb{4},3}}(\Q)^\vee$ is torsion. Let $\theta_{\pmb{4},3}$ equal $\frac{n}{d}$, for two non-zero elements $n$ and $d$ in $T[[\Gamma]]$. Without loss of generality, we shall assume that $n$, $d$ and $\pi(n)$ are non-zero. Following the notations of Proposition \ref{specialization}, we let
{\small
\begin{align*}
Y_1 = \frac{T[[\Gamma]]}{(n)} \bigoplus \left(\oplus_{ \nu \in \Sigma_0} \Loc(\nu,\rho_{\pmb{4},3})^\vee\right), \quad Y_2=\frac{\T[[\Gamma]]}{(d)} \oplus \Sel_{\rho_{\pmb{4},3}}(\Q)^\vee  , \quad M=H^0(G_\Sigma, D_{\rho_{\pmb{4},3}})^\vee.
\end{align*}
}
Let us now verify the hypotheses given in Proposition \ref{specialization}. We will need to show that $Y_1$ and $Y_2$ satisfy \ref{No-PN} and \ref{Fin-Proj}. The $T[[\Gamma]]$-modules $\frac{T[[\Gamma]]}{(n)}$ and $\frac{T[[\Gamma]]}{(d)}$  clearly satisfy \ref{No-PN} and \ref{Fin-Proj}. For every prime $\nu \in \Sigma_0$, Proposition \ref{loc-No-PN} asserts that $\Loc(\nu,\rho_{\pmb{4},3})^\vee$ satisfies \ref{No-PN}. Corollary \ref{H1-torsion-local-factors} and Lemma \ref{reg-2} assert that $\Loc(\nu,\rho_{\pmb{4},3})^\vee$ satisfies \ref{Fin-Proj}. Proposition \ref{Sel-No-PN} asserts that the  $\Sel_{\rho_{\pmb{4},3}}(\Q)^\vee$  satisfies \ref{No-PN}. Proposition \ref{Sel-Fin-Proj} asserts that $\Sel_{\rho_{\pmb{4},3}}(\Q)^\vee$ satisfies \ref{Fin-Proj}. These observations show that $Y_1$ and $Y_2$ satisfy \ref{No-PN} and \ref{Fin-Proj}. As for the remaining hypotheses we have that for every height two prime ideal $Q$ containing $\ker(\pi)$ and in the support of $H^0(G_\Sigma, D_{\rho_{\pmb{4},3}})^\vee$, by Proposition \ref{greenberg-cyc} and Lemma \ref{reg-2}, the local ring $T[[\Gamma]]_Q$ is regular. Note also that the map $\pi :T[[\Gamma]] \rightarrow R[[\Gamma]]$ is surjective. \\

Suppose \ref{euler-inequality} holds. It can be rewritten as the following inequality of divisors in $T[[\Gamma]]$:
\begin{align*}
\Div\left(Y_1\right) \geq \Div(Y_2) - \Div(M).
\end{align*}
Proposition \ref{specialization} then gives us the following inequality of divisors in $R[[\Gamma]]$:
{\footnotesize
\begin{align} \label{specialization-final-result}
&&\Div\left(\frac{R[[\Gamma]]}{\pi(n)}\right)  +   \sum \limits_{\nu \in \Sigma_0} \Div\left(\Loc(\nu,\rho_{\pmb{4},3})^\vee \otimes_{T[[\Gamma]]} R[[\Gamma]] \right) \\ && \notag \geq \Div\left(\frac{R[[\Gamma]]}{\pi(d)}\right) +  \Div\bigg(\Sel_{\rho_{\pmb{4},3}}(\Q)^\vee \otimes_{T[[\Gamma]]}R[[\Gamma]]\bigg) + & \Div\bigg(\Tor_1^{T[[\Gamma]]}\left(R[[\Gamma]], H^0(G_\Sigma,D_{\rho_{\pmb{4},3}})^\vee\right)\bigg) \\ \notag && &-  \Div\left(H^0(G_\Sigma,D_{\pi \circ \rho_{\pmb{4},3}})^\vee\right).
\end{align}}
By Proposition \ref{control-theorem-selmer-group-dasgupta}, we have the following equality in the divisor group of $R[[\Gamma]]$:
{\footnotesize \begin{align} \label{control-the-selmer}
\Div\bigg(\Sel_{\rho_{\pmb{4},3}}(\Q)^\vee \otimes_{T[[\Gamma]]}R[[\Gamma]]\bigg) + \Div\bigg(\Tor_1^{T[[\Gamma]]}\left(R[[\Gamma]], H^0(G_\Sigma,D_{\rho_{\pmb{4},3}})^\vee\right)\bigg) = \Div\bigg(\Sel_{\pi \circ \rho_{\pmb{4},3}}(\Q)^\vee \bigg).
\end{align}}
For all $\nu \in \Sigma_0$, Proposition \ref{local-control} gives us the following equality in the divisor group of $R[[\Gamma]]$:
{\small
\begin{align} \label{control-the-locals}
\Div\left(\Loc(\nu,\rho_{\pmb{4},3})^\vee \otimes_{T[[\Gamma]]} R[[\Gamma]] \right) = \Div\left(\Loc(\nu,\pi \circ \rho_{ \pmb{4},3})^\vee \right)
\end{align}}
Equation (\ref{first-main-inequality}) in Theorem \ref{specialization-result} now follows by combining equations (\ref{specialization-final-result}), (\ref{control-the-selmer}) and  (\ref{control-the-locals}).

Now suppose that the divisor $\Div\left(\Sel_{\rho_{\pmb{4},3}}(\Q)^\vee\right)-\Div(M)$ generates a torsion element in $T[[\Gamma]]$. This implies that $\Div\left(Y_2\right)-\Div(M)$ also generates a torsion element in $T[[\Gamma]]$. Recall that, for each $\nu \in \Sigma_0$, the divisor $\Div(\Loc(\nu,\rho_{\pmb{4},3})^\vee)$ is  principal (see Proposition \ref{divisor-local-factors}). So, the divisor $\Div(Y_1)$ generates the trivial element in the divisor class group of $T[[\Gamma]]$. Proposition \ref{specialization-strict} and Remark \ref{equality-specialization-div} now give us the last part of the theorem. This completes the proof.

\appendix

\section{Results from Commutative Algebra} \label{comm-algebra-appendix}

Let $\RRR$ be an integrally closed local domain that is a finite integral extension of $\Z_p[[u_1,\dotsc,u_m]]$. We shall say that a sequence $ \M_1 \rightarrow  \dotsc \rightarrow  \M_n$ of finitely generated $\RRR$-modules is \textit{exact in dimension one} if for every height one prime ideal $\p$ of $\RRR$, the following sequence of $\RRR_\p$-modules~is~exact:
$$(\M_1)_\p \rightarrow  \dotsb \rightarrow (\M_n)_\p$$
Suppose we have an exact sequence $\C_1 \xrightarrow {g_{1,2}} \C_2 \xrightarrow {g_{2,3}} \C_3 \xrightarrow {g_{3,4}} \C_4 \rightarrow 0$ of finitely generated $\RRR$-modules  along with the following commutative diagram:
{\small
\begin{align*}
\xymatrix{
\C_1 \ar[d]_{u_1} \ar[r]^{g_{1,2}}&  \C_2 \ar[d]_{u_2} \ar[r]^{g_{2,3}} & \C_3 \ar[d]_{u_3} \ar[r]^{g_{3,4}}&  \C_4  \ar[r] \ar[d]_{u_4}& 0 \\
\C_1  \ar[r]^{g_{1,2}}&  \C_2 \ar[r]^{g_{2,3}} & \C_3  \ar[r]^{g_{3,4}}&  \C_4  \ar[r]& 0
}
\end{align*}}
It is possible to split the diagram above into three commutative diagrams given below.
{\footnotesize
\begin{align*}
\xymatrix{
0 \ar[r] & \ker(g_{1,2}) \ar[d] \ar[r]& \C_1  \ar[d] \ar[r]& \text{Im}(g_{1,2}) \ar[d] \ar[r]& 0, &  0 \ar[r] & \text{Im}(g_{1,2}) \ar[d] \ar[r]& \C_2  \ar[d] \ar[r]&  \ker(g_{2,3}) \ar[d]\ar[r]& 0\\ 0 \ar[r] & \ker(g_{1,2}) \ar[r]& \C_1  \ar[r]& \text{Im}(g_{1,2}) \ar[r]& 0 &  0 \ar[r] & \text{Im}(g_{1,2}) \ar[r]& \C_2  \ar[r]&  \ker(g_{2,3}) \ar[r]& 0
}
\end{align*}
\begin{flalign*}
\xymatrix{
 0 \ar[r] & \ker(g_{2,3}) \ar[r]\ar[d]& \C_3 \ar[d] \ar[r]&  \C_4 \ar[r]\ar[d] & 0 \\
 0 \ar[r] & \ker(g_{2,3}) \ar[r]& \C_3  \ar[r]&  \C_4 \ar[r]& 0,
}
\end{flalign*}
}
If $\C_1$ is a torsion $\RRR$-module, then $\ker(u_1)$ is a pseudo-null $\RRR$-module if and only if $\coker(u_1)$ is pseudo-null. A simple application of Nakayama's Lemma tells us that a surjective endomorphism of a finitely generated module over a Noetherian ring is in fact an isomorphism (see Proposition 1.2 in \cite{MR0238839}). Using these observations and by a careful diagram-chasing argument, we obtain the following lemma:

\begin{lemma} \label{almost-compact-exactness}
Suppose $\C_1$ is a torsion $\RRR$-module. Suppose also that either $\ker(u_1)$ or $\coker(u_1)$ is pseudo-null. Then, the following sequence is exact in dimension one:
$$0 \rightarrow \ker(u_2) \xrightarrow {g_{2,3}} \ker(u_3) \xrightarrow {g_{3,4}}  \ker(u_4) \rightarrow \coker(u_2) \xrightarrow {g_{2,3}} \coker(u_3) \xrightarrow {g_{3,4}}  \coker(u_4) \rightarrow 0$$
\end{lemma}

We shall say that a  sequence $ \D_n \rightarrow  \dotsb  \rightarrow \D_1$ of discrete $\RRR$-modules (whose Pontryagin duals are finitely generated $\RRR$-modules) is \textit{coexact in dimension one} if the sequence $\D_1^\vee \rightarrow \dotsb \rightarrow  \D_n^\vee$ of finitely generated $\RRR$-modules is exact in dimension one. Using an argument similar to the one given above, we have the following lemma for discrete modules:

\begin{lemma} \label{almost-exactness}
Suppose we have an exact sequence $\D_1 \xrightarrow {g_{1,2}} \D_2 \xrightarrow {g_{2,3}} \D_3 \xrightarrow {g_{3,4}} \D_4 \rightarrow 0$ of discrete $\RRR$-modules (whose Pontryagin duals are finitely generated as $\RRR$-modules) along with the following commutative diagram:
{\small
\begin{align*}
\xymatrix{
\D_1 \ar[d]_{u_1} \ar[r]^{g_{1,2}}&  \D_2 \ar[d]_{u_2} \ar[r]^{g_{2,3}} & \D_3 \ar[d]_{u_3} \ar[r]^{g_{3,4}}&  \D_4  \ar[r] \ar[d]_{u_4}& 0 \\
\D_1  \ar[r]^{g_{1,2}}&  \D_2 \ar[r]^{g_{2,3}} & \D_3  \ar[r]^{g_{3,4}}&  \D_4  \ar[r]& 0
}
\end{align*}
}
Suppose further that the Pontryagin dual of $\D_1$ is a torsion $\RRR$-module. Assume also that either the Pontryagin dual of $\ker(u_1)$ or the Pontryagin dual of $\coker(u_1)$ is pseudo-null. Then, the following sequence is coexact in dimension one:
\begin{align*}
0 \rightarrow \ker(u_2) \xrightarrow {g_{2,3}} \ker(u_3) \xrightarrow {g_{3,4}} \ker(u_4) \rightarrow \coker(u_2) \xrightarrow {g_{2,3}} \coker(u_3) \xrightarrow {g_{3,4}} \coker(u_4) \rightarrow 0
\end{align*}
\end{lemma}

\section{A group theory lemma}

For this final part of the appendix, we shall let $\RRR$ be a Noetherian domain. The ring $\RRR$ is not assumed to be complete. Suppose $\G$ is a profinite group. Let $\M$ be a free $\RRR$-module with an  $\RRR$-linear action of $\G$. For the purposes of brevity, we shall let $\M_\G$ denote the maximal quotient $\RRR$-module of $\M$ on which $\G$ acts trivially.  Let $\mathcal{H}$ be  an open subgroup of $\G$.  Let $\SSS$ be a multiplicative set in $\RRR$.

\begin{lemma} \label{interesting-profinite-group-theory}
Suppose that $ \Q \subset \SSS^{-1} \RRR$. If $\SSS^{-1}(\M_\mathcal{H})$ is a free $\SSS^{-1} \RRR$-module, then so is $\SSS^{-1}(\M_\mathcal{G})$.
\end{lemma}

\begin{proof}
The action of $\G$ on $\M$ extends naturally to one on $\SSS^{-1}\M$ since the action is $\RRR$-linear. Let $\M_0$ denote the $\RRR$-submodule of $\M$ generated by the elements of the form $(g-1)m$, as $g$ varies over all the elements of group $\G$ and $m$ varies over all the elements of the module $\M$.
We observe that we have the following isomorphisms:
\begin{align*}
\M_\G \cong \frac{\M}{\M_0}, \quad  \SSS^{-1} (\M_\G) \cong \frac{\SSS^{-1}\M}{\SSS^{-1}\M_0} \cong (\SSS^{-1}\M)_\G
\end{align*}
This observation will allow us to assume, without loss of generality, that the set $\SSS$ equals the set of units in $\RRR$ and that $[\G: \mathcal{H}]$ is a unit in $\RRR$. What we will now need to show is that if $\M_\mathcal{H}$ is a free $\RRR$-module, then so is $\M_\mathcal{G}$. Consider the $\G$-module $\Ind^\G_{\mathcal{H}}\left(\Res^\G_\mathcal{H}\M\right)$, which we will denote by $\mathcal{N}$.  We also have a natural $\G$-isomorphism $\mathcal{N}_{\mathcal{H}} \cong \Ind^\G_{\mathcal{H}}\left((\Res^\G_{\mathcal{H}}\M)_{\mathcal{H}}\right)$, which lets us obtain the isomorphism
\begin{align} \label{inv-iso}
\mathcal{N}_\G \cong \left(\mathcal{N}_\mathcal{H}\right)_\G \cong \bigg(\Ind^\G_{\mathcal{H}}\left(\Res^\G_{\mathcal{H}}(\M)_{\mathcal{H}}\right)\bigg)_{\G}.
\end{align}
The kernel of the natural map $\G \rightarrow \Aut(\mathcal{N}_\mathcal{H})$ contains $\mathcal{H}$. The action of $\G$ on $\N_\mathcal{H}$  would thus factor through a finite group, say $\Delta$, of order dividing $[\G:\H]$. The hypothesis that $\Res^\G_{\mathcal{H}}(\M)_\mathcal{H}$ is a free $\RRR$-module combined with  Lemma 3.2.2 in \cite{Webb_2016} (which we view as an analog of Maschke's theorem) for the ring $\RRR[\Delta]$ and equation (\ref{inv-iso}) lets us conclude that $\mathcal{N}_\G$ is a free $\RRR$-module.

We shall now show that $\M$ is a direct summand of $\mathcal{N}$ as an $\RRR[\G]$-module. This tells us that  $\M_\G$ is a free $\RRR$-module whenever $\mathcal{N}_\G$ is. This would complete the proof of the lemma. First observe that the restriction functor $\Res^\G_\mathcal{H}$ is a left-adjoint to the Induction functor $\Ind^\G_\mathcal{H}$ (by Frobenius reciprocity). Since the index $[\G:\mathcal{H}]$ is finite, the induction of a module is also non-canonically isomorphic to its co-induction. So one can view the restriction functor $\Res^\G_\mathcal{H}$ as a right-adjoint to the Induction functor $\Ind^\G_\mathcal{H}$(by Hom-Tensor adjunction). This gives us the following isomorphisms:

\begin{align*}
\Hom_{\G}\left( \M,\mathcal{N}\right) \underbrace{\cong}_{\substack{\text{Frobenius} \\ \text{Reciprocity}}} \Hom_{\mathcal{H}}\left(\Res^G_{\mathcal{H}}\M,Res^G_{\mathcal{H}}\M\right) \underbrace{\cong}_{\substack{\text{Hom-Tensor} \\ \text{ adjunction}}}  \Hom_{\G}\left(\mathcal{N},\M\right).
\end{align*}

The identity map in $\Hom_{\mathcal{H}}\left(\Res^G_{\mathcal{H}}\M,Res^G_{\mathcal{H}}\M\right)$ can be pulled back to a map $i: \M \rightarrow \mathcal{N}$ in $\Hom_{\G}\left(\M,\mathcal{N}\right)$ along with a ``splitting'' map $s: \mathcal{N} \rightarrow M$ in $\Hom_{\G}\left(\mathcal{N},\M\right)$ . Tracing the isomorphisms, one gets that the composition $s \circ i$ equals the identity map. Thus, $\M$ is a direct summand of $\mathcal{N}$ as an $\RRR[\G]$-module.  The Lemma follows.

\end{proof}

{
\bibliographystyle{abbrv}
\bibliography{biblio}}
\end{document}